\documentclass[11pt]{amsart} 
\makeatletter
\providecommand{\@LN}[2]{}
\makeatother
\calclayout

\usepackage{setspace} 
\usepackage{graphicx}
\usepackage{microtype}
\usepackage[pagewise]{lineno}
\linenumbers

\usepackage{pgfplots}
\pgfplotsset{compat=1.15}
\usepackage{mathrsfs}
\usetikzlibrary{arrows}

\usepackage[english]{babel}
\usepackage[maxbibnames=5,minbibnames=4,backend=bibtex]{biblatex}
\usepackage{csquotes}
\usepackage{amssymb}
\usepackage{amsmath}
\usepackage{amsthm}

\newtheorem{thm}{Theorem}[section]
\newtheorem{lem}[thm]{Lemma}
\newtheorem{prop}[thm]{Proposition}
\newtheorem{coro}[thm]{Corollary}

\theoremstyle{remark}
\newtheorem{rema}[thm]{Remark}
\newtheorem{exa}[thm]{Example}
\newtheorem{defi}[thm]{Definition}

\DeclareFontFamily{U} {MnSymbolA}{}
\DeclareFontShape{U}{MnSymbolA}{m}{n}{
  <-6> MnSymbolA5
  <6-7> MnSymbolA6
  <7-8> MnSymbolA7
  <8-9> MnSymbolA8
  <9-10> MnSymbolA9
  <10-12> MnSymbolA10
  <12-> MnSymbolA12}{}
\DeclareFontShape{U}{MnSymbolA}{b}{n}{
  <-6> MnSymbolA-Bold5
  <6-7> MnSymbolA-Bold6
  <7-8> MnSymbolA-Bold7
  <8-9> MnSymbolA-Bold8
  <9-10> MnSymbolA-Bold9
  <10-12> MnSymbolA-Bold10
  <12-> MnSymbolA-Bold12}{}

\DeclareSymbolFont{MnSyA} {U} {MnSymbolA}{m}{n}
\DeclareMathSymbol{\lcirclearrowright}{\mathrel}{MnSyA}{252}
\DeclareMathSymbol{\rcirclearrowleft}{\mathrel}{MnSyA}{250}

\DeclareMathOperator\ad{ad}

\usepackage{stmaryrd}
\usepackage{slashed}
\usepackage{tikz-cd}
\usetikzlibrary{positioning}
\usetikzlibrary{decorations.pathreplacing,calligraphy}
\usetikzlibrary{calc,intersections}
\usepackage{todonotes}
\usepackage[all,cmtip]{xy}
    
\title{Poisson groupoids and moduli spaces of flat bundles over surfaces}
\author{Daniel \'Alvarez}\address{Departament of Mathematics, University of Toronto, Toronto, Ontario}
\email{dalv@math.toronto.edu}

\date{} 
\begin{document}
\linenumbers
\begin{abstract} Let $\Sigma $ be a compact connected and oriented surface with nonempty boundary and let $G$ be a Lie group equipped with a bi-invariant pseudo-Riemannian metric. The moduli space of flat principal $G$-bundles over $\Sigma$ which are trivialized at a finite subset of $\partial\Sigma$ carries a natural quasi-Hamiltonian structure which was introduced by Li-Bland and \v{S}evera. By a suitable restriction of the holonomy over $\partial \Sigma$ and of the gauge action, which is called a decoration of $\partial \Sigma$, it is possible to obtain a number of interesting Poisson structures as subquotients of this family of quasi-Hamiltonian structures. In this work we use this quasi-Hamiltonian structure to construct Poisson and symplectic groupoids in a systematic fashion by means of two observations: \begin{enumerate} \item gluing two copies of the same decorated surface along suitable subspaces of their boundaries determines a groupoid structure on the moduli space associated to the new surface, this procedure can be iterated by gluing four copies of the same surface, thereby inducing a double Poisson groupoid structure; 
\item on the other hand, we can suppose that $G$ is a Lie 2-group, then the groupoid structure on $G$ descends to a groupoid structure on the moduli space of flat $G$-bundles over $\Sigma$. \end{enumerate} 
These two observations can be combined to produce up to three distinct and compatible groupoid structures on the associated moduli spaces. We illustrate these methods by considering symplectic groupoids over Bruhat cells, twisted moduli spaces and Poisson 2-groups besides the classical examples.
\end{abstract}\maketitle
\tableofcontents 

\section{Introduction} 
Symplectic groupoids \cite{weisygr} are a basic tool in the study of Poisson structures, being a fundamental ingredient of what can be called global Poisson geometry \cite{craruipoi,craruimar}. The explicit description of a symplectic groupoid integrating a given Poisson structure sometimes reveals properties of the Poisson structure that would be difficult to discover or prove otherwise. Closely related to symplectic groupoids are the more general Poisson groupoids \cite{weicoi} which contain families of symplectic groupoids in their symplectic foliations \cite{folpoigro}.  

Let $\Sigma$ be a compact, connected and oriented surface with nonempty boundary and let $G$ be a Lie group with a quadratic Lie algebra, meaning that it is equipped with a nondegenerate, symmetric bilinear form which is invariant under the adjoint representation. In this situation, the space of representations $\hom(\pi_1(\Sigma),G)$ admits a canonical quasi-Poisson structure \cite{alemeimal,alemeikos} whose reduction is isomorphic to the classical Poisson structure on the moduli space of flat $G$-bundles over $\Sigma$ which can be identified with $\hom(\pi_1(\Sigma),G)/G$ \cite{atibot,atikno}. Let us take a finite set of points $V\subset \partial \Sigma$ and consider the flat $G$-bundles over $\Sigma$ endowed with a trivialization over $V$. The moduli space of flat $G$-bundles over $\Sigma$ which are trivialized over $V$ can be described by means of the space of groupoid morphisms $\hom(\Pi_1(\Sigma,V),G)$, where $\Pi_1(\Sigma,V)$ is the fundamental groupoid of $\Sigma$ based at $V$. It turns out that $\hom(\Pi_1(\Sigma,V),G)$ still carries a canonical quasi-Poisson structure which allows for greater flexibility in the construction of examples \cite{quisur,quisur2}. In fact, by fixing the holonomies of the arcs inside $\partial \Sigma$ to appropriate submanifolds of $G$ and by restricting the residual gauge action of $G^V$ on $\hom(\Pi_1(\Sigma,V),G)$ in a suitable manner, one can get many smooth examples of Poisson and symplectic structures on the corresponding quasi-Poisson quotients. This is unlike what happens in the classical case which essentially only produces singular quotients. The process of fixing the holonomies along the arcs of $\partial \Sigma$ and restricting the gauge action at the points in $V$ is called {\em coloring} in \cite{quisur,quisur2} or {\em decoration} in \cite{chemazrub}. Instead of the language of quasi-Poisson geometry, here we use mainly the equivalent theory of Hamiltonian spaces for Manin pairs \cite{buriglsev}.

As suggested by \cite[Example 11]{quisur} and by \cite[Example 8]{quisur} for instance, if one takes a collection of arcs $S\subset \partial \Sigma$ between points in $V$ and considers the surface $\Sigma \cup_S \Sigma$ obtained by identifying two copies of $\Sigma $ along $S$, we have that $\hom(\Pi_1(\Sigma \cup_S \Sigma,V\cup_S V),G)$ carries a groupoid structure over $\hom(\Pi_1(\Sigma , V),G)$ which can be reduced to a Poisson groupoid structure by decorating $\partial (\Sigma \cup_S \Sigma)$ in a symmetric way with respect to the identification. In this work we formulate this observation in general in Theorem \ref{thm:poigro}. As a new application of this result, we obtain an alternative integration for the standard Poisson structures on the Bruhat cells inside complex semi-simple Lie groups \cite{lumou2}, see Example \ref{exa:brucel}. In fact, what we obtain is a family of symplectic groupoids over generalized Bruhat cells which is related to a recent construction of J.-H. Lu, V. Mouquin and S. Yu \cite{confla}. Later on, we show how an even more symmetric decoration of $\partial \Sigma$ induces two compatible groupoid structures on the associated moduli space, see Theorem \ref{thm:doupoigro}. This kind of decoration corresponds to gluing four copies of the same surface along a common set of edges and decorating the boundary of the resulting surface symmetrically with respect to the two identifications. This observation provides an explanation for the appearance of double symplectic groupoids integrating Poisson groups \cite{luwei2} (Example \ref{exa:dousym2}) and it also allows the systematic construction of many more non trivial examples of double Poisson and symplectic groupoids. A deeper explanation for the appearance of these Poisson groupoid structures comes from the existence of a classical topological field theory in dimension three which was introduced in \cite{sevmorqua}. From this perspective, the groupoid multiplication can be described in terms of a certain cobordism given by a 3-manifold with corners, see  \cite{sevmorqua} for an example of this. We shall describe this approach in later work. 

Finally, by categorifying the structure group $G$ and considering instead a Lie 2-group $\mathcal{G}$ equipped with a split quadratic Lie 2-algebra, we get a groupoid structure on the representation space $\hom(\pi_1(\Sigma),\mathcal{G})$. Examples of these categorified structure groups come from the twisted moduli spaces \cite{twiwilcha,twimodspa} in which $\mathcal{G} $ is the automorphism 2-group of a semi-simple Lie group. If one glues two copies of the same surface and decorates the boundary in a way compatible with the groupoid structure on $\mathcal{G} $, what we obtain is a double Poisson groupoid structure, see Theorem \ref{thm:doupoigro2}. By combining the conditions of Theorems \ref{thm:doupoigro} and \ref{thm:doupoigro2}, what we get is a triple Poisson groupoid structure on the corresponding moduli spaces, see Theorem \ref{thm:tripoigro}. These last two results are illustrated by considering Poisson 2-groups \cite{poi2gro} and some symplectic groupoids which integrate them. Let us remark that the double Lie groupoids which appear in our constructions in \S \ref{sec:doupoigro} before performing some form of reduction are {\em double quasi-Poisson groupoids}, a concept which generalizes quasi-Poisson Lie 2-groups \cite{poi2gro} and double Poisson groupoids \cite{macdousym}.

\section{Preliminaries}

\subsection{Poisson structures and Poisson groupoids} We follow mainly the conventions of \cite{moeint,dufzun}. A {\em Poisson structure} on a manifold $M$ is a Lie bracket 
\[ \{\,,\,\}: C^\infty(M) \times  C^\infty(M) \rightarrow  C^\infty(M) \] such that the Leibniz rule holds: $\{f,gh\}=g\{f,h\}+\{f,g\}h$ for all $f,g,h\in C^\infty(M)$. A Poisson structure $\{\,,\,\}$ turns $(M,\{\,,\,\})$ into a \emph{Poisson manifold}. A Poisson morphism $J:(P,\{\,,\,\}_P) \rightarrow (Q,\{\,,\,\}_Q)$ between Poisson manifolds is a smooth map that satisfies $J^*\{f,g\}_Q=\{J^*f,J^*g\}_P$ for all $f,g\in C^\infty (Q)$. The bracket $\{\,,\,\}$ can be equivalently described by a bivector field $\pi$ on $M$ defined by $\{f,g\}=\pi(df,dg)$ for all $f,g \in C^\infty(M)$. Let $(M,\pi)$ be a Poisson manifold. A submanifold $C$ of $M$ is {\em coisotropic} if $\pi^\sharp(T^\circ C)\subset TC$, where $T^\circ C$ is the annihilator of $TC$. 

A {\em Lie groupoid} is a groupoid object in the  category of smooth manifolds, denoted $\mathcal{G}  \rightrightarrows M$, such that its source map is a submersion. The structure maps of a groupoid are its source, target, multiplication, unit map and inversion, denoted respectively $\mathtt{s},\mathtt{t},\mathtt{m},\mathtt{u} , \mathtt{i}$; when several Lie groupoids are involved, we put a subindex under the structure maps to specify the groupoid to which they correspond. In this paper we only work with Hausdorff Lie groupoids. Some of the most basic examples of Lie groupoids are: (1) Lie groups and (2) the pair groupoid over a manifold $M$, which is given by the product $M \times M \rightrightarrows M$ and whose multiplication is defined by $\mathtt{m}((a,b),(b,c))=(a,c)$, its source and target being the projections to $M$. A {\em Lie algebroid} consists of a vector bundle $A$ over a manifold $M$ which is equipped with: (1) a bundle map $\mathtt{a}:A\rightarrow TM$ called the {\em anchor} and (2) a Lie algebra structure $[\,,\,]$ on $\Gamma(A)$ such that the Leibniz rule holds
\[ [u,fv]=f[u,v]+\left(\mathcal{L}_{\mathtt{a}(u)}f\right)v, \] 
for all $u,v\in \Gamma(A)$ and $f\in C^\infty(M)$. See \cite{higmacalg,vai} for the definition of a Lie algebroid morphism. The {\em Lie algebroid} $A=A_{\mathcal{G} }$ of a Lie groupoid $\mathcal{G} \rightrightarrows M$ is the vector bundle $A=\ker T\mathtt{s}|_{M}$ endowed with the restriction of $T\mathtt{t}$ to $A$ as the anchor and with the bracket defined by means of right invariant vector fields \cite{moeint,macgen}. We denote the left invariant (respectively, right invariant) extension of $u \in \Gamma (A)$ by $u^l$ (respectively, $u^r$). Note that this convention implies, in particular, that the Lie bracket on the Lie algebra $\mathfrak{g} $ of a Lie group $G$ is defined using right rather than left invariant vector fields as is traditional in differential geometry.

A {\em Poisson groupoid} \cite{weicoi} is a Lie groupoid ${\mathcal{G}  }  \rightrightarrows M$ with a Poisson structure on ${\mathcal{G} }$ such that the graph of the multiplication map is a coisotropic submanifold of $\mathcal{G} \times \mathcal{G} \times \overline{\mathcal{G} }$, where $\overline{\mathcal{G} }$ denotes $\mathcal{G} $ with the opposite Poisson structure. The two extreme examples of Poisson groupoids are: 
\begin{enumerate} \item Poisson groupoids over a point which are called {\em Poisson (or Poisson-Lie) groups} \cite{driham} and \item Poisson groupoids whose bracket is nondegenerate, in which case they are known as {\em symplectic groupoids} \cite{weisygr,karpoi}. \end{enumerate}  
Equivalently, a Lie groupoid $\mathcal{G}  \rightrightarrows M$ is a symplectic groupoid if it is endowed with a symplectic form $\omega $ such that the graph of the multiplication in $\mathcal{G} \times \mathcal{G} \times {\mathcal{G} }$ is Lagrangian with respect to the symplectic form $\text{pr}_1^* \omega +\text{pr}_2^* \omega -\text{pr}_3^* \omega $. The base of a symplectic groupoid inherits a unique Poisson structure such that the target map is a Poisson morphism \cite{weisygr,dufzun}; in this situation, such a Poisson structure is called {\em integrable}.
\subsection{Poisson groupoids and Hamiltonian spaces for Manin pairs}
The finite-dimensional construction of Poisson structures on moduli spaces of flat bundles over surfaces relies on a remarkable analogue of Poisson reduction \cite{alemeimal,alemeikos} that can be subsumed together with classical Marsden-Weinstein reduction into the framework of Courant algebroids and their morphisms \cite{buriglsev}. This language provides us with a very efficient and general way of describing Poisson reduction that can be adapted to the setting of Poisson groupoids in a straightforward manner (see Proposition \ref{pro:mulmp}). However, it has the drawback of being somewhat abstract. For this reason, we include in Example \ref{ex:quaham} below a brief comparison with the classical formulation of moment maps. Also, we review in \S \ref{subsec:quapoi} the equivalent viewpoint of quasi-Poisson geometry. 

\subsubsection{Manin pairs and moment maps} A {\em Courant algebroid} \cite{liuweixu} consists of a vector bundle $E\rightarrow M$, a non-degenerate symmetric bilinear form on the fibers $\langle\,,\, \rangle\in \Gamma (E^* \otimes E^* )$ that we call the {\em metric}, a bilinear bracket $\llbracket\,,\, \rrbracket:\Gamma(E)\times \Gamma(E)\rightarrow \Gamma(E)$ called the {\em Courant bracket} and a vector bundle morphism $\mathtt{a}:E\rightarrow TM$ called the {\em anchor} such that the following identities hold:
\begin{align*} &\mathtt{a}(u)\langle v,w \rangle =\langle \llbracket u,v\rrbracket,w \rangle + \langle v, \llbracket u,w\rrbracket \rangle, \\
& \llbracket u,\llbracket v,w\rrbracket\rrbracket=\llbracket\llbracket u,v],w\rrbracket+\llbracket v,\llbracket u,w\rrbracket\rrbracket,  \\
& \frac{1}{2}\mathtt{a}^*d \langle u,u \rangle = \llbracket u,u\rrbracket^\flat; \end{align*}
	for all $u,v,w \in \Gamma(E)$, where ${}^\flat:E\rightarrow E^*$ is the isomorphism given by the metric. It is a consequence of the axioms for a Courant algebroid $(E,\langle\,,\, \rangle, \llbracket \,,\, \rrbracket, \mathtt{a} )$ that the Leibniz rule holds and $\mathtt{a}$ preserves brackets \cite{uchcou}:
	\begin{align*} & \llbracket u,fv\rrbracket=f \llbracket u,v\rrbracket+\left(\mathcal{L}_{\mathtt{a}(u)}(f)\right)v, 
& \mathtt{a}\left(\llbracket u,v\rrbracket\right)=[\mathtt{a}(u),\mathtt{a}(v)]; \end{align*} 
for all $u,v \in \Gamma(E)$ and all $f\in C^\infty(M)$. The two most basic families of examples of Courant algebroids are: (1) Lie algebras endowed with an Ad-invariant nondegenerate symmetric bilinear form which are also known as {\em quadratic Lie algebras}, and (2) {\em exact Courant algebroids} which consist of the following data. Let $M$ be a manifold and let $H$ be a closed 3-form on $M$. The {\em ($H$-twisted) Courant-Dorfman bracket} \cite{coubey,coudir} and the canonical pairing on $\Gamma(TM\oplus T^*M)$ are defined, respectively, by:
\begin{align*}  & \llbracket X \oplus \alpha,Y \oplus \beta \rrbracket:=[X,Y]\oplus L_X\beta-i_Yd\alpha-i_Xi_Y H, \\
&\langle X\oplus \alpha,Y\oplus \beta\rangle=\langle\alpha,Y\rangle+\langle \beta,X \rangle; \end{align*} 
for all $X,Y\in \Gamma(TM)$, $\alpha,\beta\in \Gamma(T^*M)$. We have that $\mathbb{T}_HM=TM\oplus T^*M$ endowed with the Courant-Dorfman bracket and its canonical pairing is a Courant algebroid; if $H=0$, we get the \emph{standard Courant algebroid}. A Courant algebroid over $M$ is called {\em exact} if it is isomorphic to $\mathbb{T}_HM$ for some $H$. 

A subbundle $L\subset E$ of a vector bundle $E$ endowed with a metric $\langle\,,\, \rangle\in \Gamma (E^* \otimes E^* )$ is called {\em isotropic} if $\langle\,,\, \rangle|_L=0$; a subbundle $L\subset E$ is called {\em Lagrangian} if it satisfies $L=L^\perp$. For example, an isomorphism of Courant algebroids $E\cong\mathbb{T}_HM$ is determined (if it exists) by a section of the anchor map $\mathtt{a}:E \rightarrow TM$ with Lagrangian image, such a section is called a {\em splitting} of $E$. A {\em Dirac structure} in a Courant algebroid $E$ over $M$ is a Lagrangian subbundle $L\subset E$ which is involutive with respect to the restricted Courant bracket. Let $N \hookrightarrow M$ be a submanifold, a {\em Dirac structure with support on $N$} is a Lagrangian subbundle $L \hookrightarrow E|_N$ such that: 
\begin{itemize} \item $\llbracket u,v\rrbracket|_N\in \Gamma (L)$, whenever $u ,v\in \Gamma (E)$ satisfy $u|_N,v|_N\in \Gamma (L)$ and \item $\mathtt{a}(u)\in \Gamma (TN)$ if $u\in \Gamma (L) $. \end{itemize}
If $L$ is a Dirac structure inside the Courant algebroid $E$, then $(E,L)$ is called a \emph{Manin pair}. A Manin pair $(E,L)$ is called {\em exact} if $E$ is an exact Courant algebroid. For simplicity, a Dirac structure in $\mathbb{T} M$ shall be called a Dirac structure on $M$. For example:
\begin{itemize} \item $\text{Graph}(\omega )=\{X\oplus i_X \omega \in \mathbb{T} M |X\in TM\}$ for $\omega \in \Omega^2(M)$ closed and
\item $\text{Graph}(\pi )=\{i_\alpha \pi \oplus \alpha  \in \mathbb{T} M |\alpha \in T^*M\}$ for $\pi $ a Poisson bivector field on $M$ \end{itemize} 
are Dirac structures on $M$.  
 
Dirac structures are a useful tool in Poisson reduction for the following reason. Let $L$ be a Dirac structure on $M$; $f \in C^\infty (M)$ is called {\em admissible} if there exists a vector field $X_f$ such that $X_f\oplus df\in \Gamma(L)$. Let $f,g$ be admissible functions for $L$, the bracket $\{f,g\}=L_{X_f}g$ makes the space of admissible functions into a Lie algebra; if the distribution $\mathcal{N} =L \cap(TM\oplus 0)$ is of constant rank, its induced foliation is called {\em the null foliation of $L$}. We shall also say that $\mathcal{N}$ is the {\em kernel of $L$} and we denote it by $\ker L$. A function is admissible if and only if it is constant along the leaves of $\mathcal{N} $, so the space of admissible functions can be identified with the space of smooth functions on the leaf space of $\mathcal{N} $. A Dirac structure is called {\em reducible} if $\mathcal{N} $ is a simple foliation. In this situation, the bracket on admissible functions induces a Poisson structure on the leaf space of $\mathcal{N} $. 

Let us emphasize that, being a Lie algebroid, a Dirac structure $L$ on an arbitrary Courant algebroid $E$ over $M$ induces an involutive distribution $\mathtt{a}(L)\subset TM$. We will say that the leaves of this foliation are {\em $L$-orbits}. On the other hand, the null foliation discussed above only makes sense if $E=\mathbb{T}M$ (and if $\ker L$ is of constant rank). If $(E,L)$ is a Manin pair on $M$, we say that a submanifold $S \hookrightarrow M$ is {\em $L$-invariant} if it is a union of $L$-orbits. 

Let $E,F$ be Courant algebroids over $M,N$ respectively. We denote by $\overline{F} $ the Courant algebroid $F$ with the opposite inner product. Given a smooth map $f:M \rightarrow N$, let us denote by $\text{Graph}(f) \subset M \times N $ the graph of $f$. A {\em Courant morphism} \cite{alexu} between $E$ and $F$ over $f$ is a Dirac structure $R\subset {E} \times \overline{F}  $ with support on $ \text{Graph}(f)$; we will write $u\sim_R v$ if and only if $(u,v)\in R$ in what follows. More generally, we can replace $\text{Graph}(f)$ with any smooth relation $S$ between $M$ and $N$ and obtain a more general notion of Courant morphism. A Courant morphism $R$ over a relation $S \subset M \times N$ as before is {\em exact} if the sequence
\[ \xymatrix{0 \ar[rr]  && T^\circ S \ar[rr]^-{(\mathtt{a}_E \times \mathtt{a}_F) ^*} && R \ar[rr]^-{\mathtt{a}_E \times \mathtt{a}_F}&& T S \ar[rr] &&0} \]
is exact. For instance, if $f:M \rightarrow N$ is a smooth map, then we get a {\em canonical exact Courant morphism over $f$ (or $\text{Graph}(f)$)} 
\begin{align}  R_f:\mathbb{T}M \rightarrow \mathbb{T}N,\quad R_f=\{(X\oplus f^* \alpha ,Tf(X)\oplus \alpha)| X\oplus \alpha \in TM\oplus f^*(T^*N) \}. \label{eq:canmpm} \end{align} 
We will use the notation $S:M\dashrightarrow N$ for relations between manifolds as above; also, the transpose of $S$ is the relation $S^\top: N \dashrightarrow M $ defined by $S^\top=\{(n,m)|(m,n)\in S\}$. Relations between manifolds can be composed under clean intersection assumptions, see \cite[Appendix A]{liedir}. The same considerations apply to Courant algebroids. The composition of two Courant morphisms  
\[ R:E \rightarrow E' \quad \text{and }\quad R':E' \rightarrow F, \] 
defined respectively over two smooth relations $S:M \dashrightarrow N$ and $S':N \dashrightarrow P$, is given by the pointwise composition of set-theoretic relations $R'\circ R$ which is again a Courant morphism over the composition of relations $S'\circ S$ as long as 
\begin{itemize} \item $R\times R'$ intersects cleanly with $E \times E'_\Delta \times F= \{(u,v,v,w)\in E \times E' \times E' \times F|u\in E,\,v\in E',\, w\in F \}$ and 
\item the projection $(R \times R') \cap (E \times E'_\Delta \times F )\rightarrow E \times F$ has constant rank. \end{itemize} 
In this situation, it is said that $R'$ and $R$ are {\em cleanly composable} \cite[Proposition 1.4]{liedir}. In the example of exact Courant algebroids described above, if $L$ is a Dirac structure on $N$ and $f:M \rightarrow N$ is a smooth map, $L\circ R_f$ is denoted as $\mathfrak{B}_f (L)$ and is also a Dirac structure as long as $L$ and $R_f$ compose cleanly. A {\em morphism between Manin pairs} \cite{buriglsev} $(E,L)$ and $(F,K)$ over a smooth map $f: M \rightarrow N$ is a Courant morphism $R$ between $E$ and $F$ such that the composition satisfies $R\circ L=K|_{f(M)} $ and $\ker R\cap L=0$. A morphism of Manin pairs $R:(E,L) \rightarrow (F,K)$ has the following key property: it induces a Lie algebra morphism 
\[ \rho_R:\Gamma (K) \rightarrow \mathfrak{X}(M)\] given by $\rho_R(u)=\mathtt{a}_E(v)$ for all $u\in \Gamma (K)$, where $v\in \Gamma (L)$ is the only section that satisfies $(v_p,  u_{f(p)})\in R$ for all $p\in M$. 

\begin{defi}\label{def:quaham} A manifold $M$ is a {\em quasi-Hamiltonian space for a Manin pair $(E,L)$} if there exists an exact morphism of Manin pairs $R:(\mathbb{T}M,TM) \rightarrow (E,L)$ (over a map) such that $E$ is exact. \end{defi}  

\begin{rema} Quasi-Hamiltonian spaces are called {\em exact quasi-Hamiltonian} or {\em quasi-symplectic} in \cite{quisur2}. We prefer to use this terminology for simplicity and also to distinguish them among the more general quasi-Poisson manifolds we will describe in \S \ref{subsec:quapoi}. \end{rema}

Take a quasi-Hamiltonian space $M$ for a Manin pair $(E,L)$ with associated Courant morphism $R$. Then the null foliation of the pullback Dirac structure $L\circ R$ is the tangent distribution to the $L$-orbits on $M$. So the leaf space of the $L$-action on $M$ becomes a Poisson manifold if $L\circ R$ is reducible. This is an abstraction of what happens in the following well known situations.  
\begin{exa}\label{ex:quaham} The two main families of examples of quasi-Hamiltonian spaces come from the theory of moment maps.
\begin{enumerate} \item Suppose that $(E,L)=(\mathbb{T} \mathfrak{g}^*, \text{Graph}(\pi_{LP})  ) $, where $\text{Graph}(\pi_{LP})$ is the graph of the canonical Lie-Poisson (or Kirillov-Kostant-Souriau) linear Poisson structure on $\mathfrak{g}^* $, the dual of a Lie algebra $\mathfrak{g} $. Then a Hamiltonian $\mathfrak{g} $-space $(M,\Omega  )$ (in the classical sense of symplectic geometry) with moment map $\mu:M \rightarrow \mathfrak{g}^*$ determines a quasi-Hamiltonian space for the Manin pair $(\mathbb{T} \mathfrak{g}^*, \text{Graph}(\pi_{LP})  )$ with the Manin pair morphism
\[ R=\{(X\oplus \mu^* \alpha -i_X \Omega  ,T\mu(X)\oplus \alpha )\in \mathbb{T}M \times \mathbb{T} \mathfrak{g}^*  |X\oplus \alpha \in TM \oplus \mu^*T^*\mathfrak{g}^* \}. \] 
\item Let $G$ be a Lie group such that its Lie algebra $\mathfrak{g} $ is equipped with an Ad-invariant nondegenerate symmetric pairing $\langle \cdot,\cdot \rangle $. Let $\theta^r,\theta^l$ be, respectively, the right and left invariant Maurer-Cartan forms on $G$. Then we have the {\em Cartan-Dirac structure $L_G$} on $G$ 
\[ L_G= \left\{u^r-u^l \oplus \frac{1}{2}\langle \theta^r +\theta^l,u \rangle \in \mathbb{T}G|u\in \mathfrak{g}   \right\}\]
which is involutive with respect to the Courant-Dorfman bracket twisted by the Cartan 3-form
\[ H= \frac{1}{12}\langle \theta^l,[\theta^l,\theta^l] \rangle \in \Omega^3(G). \]
A manifold $M$ equipped with a map $\mu:M \rightarrow G$ and a 2-form $\Omega  \in \Omega^2(M)$ is a quasi-Hamiltonian $\mathfrak{g} $-manifold in the sense of \cite[Def. 5.1]{purspi} (which is the infinitesimal version of the quasi-Hamiltonian $G$-spaces of \cite{alemeimal}) if and only if 
\[ R=\{(X\oplus \mu^* \alpha -i_X \Omega  ,T\mu(X)\oplus \alpha )\in \mathbb{T}M \times \mathbb{T}_H G  |X\oplus \alpha \in TM \oplus \mu^* T^*G  \} \]
is a morphism of Manin pairs $R: (\mathbb{T}M ,TM) \rightarrow (\mathbb{T}_HG ,L_G )$, see \cite[Thm 5.2]{purspi} and \cite[Ex. 2.11]{buriglsev}.    
\end{enumerate} 
In fact, the Manin pair morphism associated to a general quasi-Hamiltonian space can be represented in terms of a 2-form as in the two examples above as long as a splitting of the target Courant algebroid is chosen, see \cite{purspi,buriglsev}. \end{exa}  
\subsubsection{Multiplicative Manin pairs and Poisson groupoids} Now we shall promote classical Poisson reduction and moment maps to the level of Lie groupoids. A VB-groupoid is a groupoid in the category of vector bundles. A {\em CA-groupoid} \cite{muldir,liedir} is a VB-groupoid $\mathcal{E}   \rightrightarrows \mathcal{E}_0 $ over some Lie groupoid $\mathcal{G}  \rightrightarrows M$ such that $\mathcal{E}  $ possesses a Courant algebroid structure $(\mathcal{E} , \langle \,,\, \rangle , \llbracket \,,\, \rrbracket )$ which is multiplicative: the graph of the multiplication of $\mathcal{E}  \rightrightarrows  \mathcal{E}_0$ is a Dirac structure inside $\mathcal{E}   \times \mathcal{E}   \times \overline{\mathcal{E}  } $ whose support is the graph of the multiplication of $\mathcal{G} $. A {\em multiplicative Manin pair} $(\mathcal{E},\mathcal{L}   )$ is a CA-groupoid $\mathcal{E}   \rightrightarrows \mathcal{E}_0$ and a Dirac structure $\mathcal{L}  \subset \mathcal{E}  $ which is also a VB-subgroupoid $\mathcal{L} \rightrightarrows \mathcal{L}_0 $ of $\mathcal{E}  \rightrightarrows \mathcal{E}_0 $; in this situation, $\mathcal{L} $ is called a {\em multiplicative Dirac structure} \cite{muldir}. A {\em morphism of multiplicative Manin pairs} is a morphism of Manin pairs whose support is the graph of a Lie groupoid morphism and which is also, as a relation, a VB-subgroupoid inside the corresponding product of the CA-groupoids. 
\begin{exa} Take a Manin pair $(E,L)$ over $M$. Then the pair groupoid $\left(\mathcal{E} \rightrightarrows \mathcal{E}_0\right) =\left( E \times \overline{E} \rightrightarrows E\right)$ is a CA-groupoid over the pair groupoid on $M$ and the pair groupoid $(\mathcal{L} \rightrightarrows \mathcal{L}_0)= ( L \times L \rightrightarrows L)$ is a multiplicative Dirac structure inside $\mathcal{E} \rightrightarrows \mathcal{E}_0$ determining then a multiplicative Manin pair $(\mathcal{E},\mathcal{L})$.  \end{exa} 

\begin{defi}\label{def:mulquaham} A {\em multiplicative quasi-Hamiltonian space} for a multiplicative Manin pair $(\mathcal{E},\mathcal{L}   )$ is a Lie groupoid $\mathcal{G}\rightrightarrows M $ equipped with a morphism of multiplicative Manin pairs $R:(\mathbb{T} \mathcal{G},T \mathcal{G}) \rightarrow (\mathcal{E},\mathcal{L}  )$ which makes it into a quasi-Hamiltonian space for $(\mathcal{E},\mathcal{L}   )$.  \end{defi} 

\begin{defi}\label{def:mulfol} Let $\mathcal{G} \rightrightarrows M$ be a Lie groupoid. A VB-subgroupoid $(\mathcal{D} \rightrightarrows \mathcal{D}_0 ) \hookrightarrow (T \mathcal{G} \rightrightarrows TM)$ which is an involutive distribution induces a {\em multiplicative foliation} \cite{muldir}. We shall say that a multiplicative foliation is {\em multiplicatively simple} if it is simple and if its leaf space inherits a Lie groupoid structure such that the quotient map is a groupoid morphism. \end{defi}  
\begin{prop}\label{pro:mulmp} Let $(\mathcal{E},\mathcal{L}  )$ be an exact multiplicative Manin pair and let $\mathcal{A}\subset \mathcal{E} $ be a multiplicative Dirac structure such that $\mathcal{L}\cap \mathcal{A}$ is of constant rank. Let $\mathcal{G}\rightrightarrows M $ be a multiplicative quasi-Hamiltonian space for $(\mathcal{E},\mathcal{L}   )$ with Manin pair morphism $R:(\mathbb{T} \mathcal{G},T \mathcal{G}) \rightarrow (\mathcal{E},\mathcal{L}  )$ over $\Phi:\mathcal{G} \rightarrow \mathcal{P} $. Then the following statements hold. 
\begin{enumerate} \item Suppose that $\mathcal{L} $ and $R$ are cleanly composable and suppose that the null foliation of $\mathcal{L}\circ R$ is multiplicatively simple with leaf space ${\mathcal{G} }/\mathcal{L} $. Then the induced Poisson structure on ${\mathcal{G} }/\mathcal{L} $ makes it into a Poisson groupoid. 
\item Suppose that $\mathcal{O}\subset \mathcal{P}$ is a Lie subgroupoid which is a union of $\mathcal{A}  $-orbits and suppose that $\Phi$ has clean intersection with $\mathcal{O} $. Let $i:\Phi^{-1}(\mathcal{O} ) \hookrightarrow \mathcal{G} $ be the inclusion. Suppose that $\mathcal{A} $, $R$ and the canonical Courant relation $R_i$ over $i$ are cleanly composable. Then we have that:
\begin{itemize} \item the null foliation of $\mathfrak{B}_i(\mathcal{A} \circ R)$ is given by the image of $(\mathcal{L} \cap \mathcal{A})|_{\mathcal{O} }$ under the action map $\Phi^*\mathcal{L} \rightarrow T \mathcal{G}$ defined by $R$; 
\item if the null foliation of $\mathfrak{B}_i(\mathcal{A} \circ R)$ is multiplicatively simple with leaf space $\Phi^{-1}(\mathcal{O} )/(\mathcal{L} \cap \mathcal{A} )$, then the Poisson structure induced by $\mathfrak{B}_i( \mathcal{A} \circ R)$ on $\Phi^{-1}(\mathcal{O} )/(\mathcal{L} \cap \mathcal{A} ) $ makes it into a Poisson groupoid. \end{itemize}  
\item In the situation of the previous item, if $T\mathcal{O}$ is spanned by the $ \mathcal{A} $-action on $\mathcal{P} $, then $\Phi^{-1}(\mathcal{O} )/(\mathcal{L} \cap \mathcal{A} )$ becomes a symplectic groupoid. \end{enumerate}  \end{prop}\begin{proof} Item (1) follows from item (2) applied to $\mathcal{A}=\mathcal{L} $ and $\mathcal{O}=\mathcal{P} $. On the other hand, items (2) and (3) follow from adapting \cite[Thm. 1.1]{quisur2} to this situation. Let $\mathcal{O}\hookrightarrow \mathcal{P}$ be a Lie subgroupoid and let $\mathcal{A} \hookrightarrow \mathcal{E} $ be a multiplicative Dirac structure as in item (2). Since $\mathcal{E}$ is an exact Courant algebroid, it is isomorphic to $\mathbb{T}_H \mathcal{P}  $ for some $H\in \Omega^3( \mathcal{P} )$ closed. Since $R$ a Dirac structure that covers $T \text{Graph}(\Phi)$ surjectively, it is of the form
\begin{align*}  R=\left\{ X\oplus \Phi^* \alpha -i_X \Omega  ,T \Phi(X)\oplus \alpha )\in \mathbb{T} \mathcal{G} \times \mathbb{T} \mathcal{P}   |X\oplus \alpha \in T \mathcal{G} \oplus \Phi^*T^*\mathcal{P}  \right\} \label{eq:mulmp1} \end{align*} 
for some $\Omega \in \Omega^2(\mathcal{G} )$ such that $d \Omega  =-\Phi^*H$ as in Example \ref{ex:quaham}. Suppose that $\mathcal{S}\hookrightarrow \mathcal{O}  $ is an $\mathcal{A}$-orbit. So $\mathcal{A}|_{\mathcal{S} }=\{V\oplus \xi|i^*\xi=i_V \omega \}$ for some $\omega \in \Omega^2(\mathcal{S} )$. As a consequence, 
\[ \mathfrak{B}_j( \mathcal{A} \circ R)=\text{Graph}((\Phi\circ j)^* \omega - j^*\Omega  )  \] 
where $j: \Phi^{-1}( \mathcal{S}) \hookrightarrow \mathcal{G} $ is the inclusion. Now take $V\oplus \xi\in ( \mathcal{L}\cap\mathcal{A} )|_{\Phi(x)} $ for $x\in \Phi^{-1}(\mathcal{S} )$. Since $R$ is Manin pair morphism, there is a unique $X\in T_x \mathcal{G} $ such that $X\oplus 0$ is related to $V\oplus \xi$ via $R$, this means that $T\Phi(X)=V$ and $\Phi^* \xi-i_X \Omega =0 $. This last equation implies that $X\in \ker((\Phi\circ j)^* \omega - j^*\Omega )$ because $V\oplus \xi\in \mathcal{A} $ as well and so $\langle \xi, T\Phi(Y) \rangle =(\Phi^*\omega) (X,Y)$ for all $Y\in T_x \Phi^{-1}(\mathcal{S} )$. Conversely, suppose that $Z\in \ker((\Phi\circ j)^* \omega - j^*\Omega )$. Since $\Phi$ intersects $\mathcal{S} $ cleanly as well, the sequence
\[ \begin{tikzcd}
T_x\Phi^{-1}(\mathcal{S}) \arrow[rrr, "{Y\mapsto (T\Phi(Y),Y)}"] &  &  & T_{\Phi(x)}\mathcal{S}\times T_x\mathcal{G} \arrow[rrr, "{(V,Y)\mapsto V-T\Phi(Y)}"] &  &  & T_{\Phi(x)}\mathcal{P}
\end{tikzcd} \]
is exact and so is its dual sequence. As a consequence, there exists $\eta\in T^*_{\Phi(x)}\mathcal{P}$ such that $(\eta|,-\Phi^*\eta)=(i_{T\Phi(Z)} \omega ,-i_Z \Omega  )\in T_{\Phi(x)}^*\mathcal{S} \times T^*_x \mathcal{G}  $. But this means that $T\Phi(Z)\oplus \eta \in \mathcal{A}$ and also that $(Z\oplus 0,T\Phi(Z)\oplus \eta)\in R$ from which it follows that $T\Phi(Z)\oplus \eta \in \mathcal{L}$. We have verified then that $Z$ lies in the image of $(\mathcal{L}\cap \mathcal{A})|_{\mathcal{S} }$ in $T\left(\Phi^{-1}(\mathcal{S} )\right)$ under the action map. Since $\mathcal{O}$ is a union of $\mathcal{A}$-orbits, the previous argument applied to each of these orbits shows that the null foliation of $\mathfrak{B}_i( \mathcal{A} \circ R) $ is the tangent distribution to the $ (\mathcal{L}\cap \mathcal{A})$-orbits on $\Phi^{-1}(\mathcal{O})$. Since $R$ is multiplicative Courant morphism, $\mathfrak{B}_i( \mathcal{A} \circ R) $ is a multiplicative Dirac structure and so, under our assumptions, the leaf space of its null foliation inherits a Poisson groupoid structure according to \cite[Thm. 4.4]{leaspa}. Moreover, we also got as a result of our construction that the symplectic leaves of $\Phi^{-1}(\mathcal{O} )/(\mathcal{L} \cap \mathcal{A} ) $ are precisely the connected components of the subquotients
\[ \Phi^{-1}(\mathcal{S})/(\mathcal{L} \cap \mathcal{A} ) \hookrightarrow \Phi^{-1}(\mathcal{O})/(\mathcal{L} \cap \mathcal{A} ) \]  
for all the $\mathcal{A}$-orbits $\mathcal{S} \hookrightarrow \mathcal{P} $. It follows that, if $T \mathcal{O}$ is spanned by the tangent spaces to the $\mathcal{A}$-orbits, then at each point $x\in \Phi^{-1}(\mathcal{O} )$ we have that $ T_{[x]} \left(\Phi^{-1}(\mathcal{O} )/(\mathcal{L} \cap \mathcal{A} )\right)  =T_{[x]}\left(\Phi^{-1}(\mathcal{S})/(\mathcal{L} \cap \mathcal{A} ) \right) $, where $\mathcal{S}\hookrightarrow \mathcal{O}  $ is the $\mathcal{A}$-orbit that contains $\Phi(x)$. Therefore, $\Phi^{-1}(\mathcal{O} )/(\mathcal{L} \cap \mathcal{A} ) $ is symplectic and hence item (3) holds as well. \end{proof} 
    
\begin{rema} Note that in item (2) of Proposition \ref{pro:mulmp} the symplectic leaves of $\Phi^{-1}(\mathcal{O} )/(\mathcal{L} \cap \mathcal{A} ) $ are given by the connected components of $\Phi^{-1}(\mathcal{S})/(\mathcal{L} \cap \mathcal{A} )$ for all the $\mathcal{A}$-orbits $\mathcal{S} \hookrightarrow \mathcal{P} $. Based on the proof of Proposition \ref{pro:mulmp}, we can find explicit formulas for the symplectic structures on these leaves by choosing a splitting of $\mathcal{E} $. However, let us point out that the auxiliary differential forms $\Omega \in \Omega^2(\mathcal{G} )$, $H\in \Omega^3(\mathcal{P} )$ we introduced in the proof of Proposition \ref{pro:mulmp} are not necessarily multiplicative forms in the sense of \cite{burcab}, see Example \ref{exa:dousym}. The properties of these differential forms can be explained within the framework of shifted symplectic geometry but for reasons of space we shall describe this viewpoint in a separate upcoming work. 
 \end{rema} 
\subsubsection*{Notation} Let us conclude this section by introducing the following notation.
 If $\mathcal{I} \hookrightarrow \mathcal{E} $ is a (multiplicative) isotropic subbundle over a Lie subgroupoid $\mathcal{O} \hookrightarrow \mathcal{P} $, we will denote by $\Phi^{-1}(\mathcal{O} )/(\mathcal{L} \cap \mathcal{I} )$ the leaf space of the foliation on $\Phi^{-1}(\mathcal{O} )$ induced by the action of $\mathcal{L} \cap \mathcal{I} $. In our applications, $\mathcal{L}$ will be a vector bundle of the form $\mathfrak{l}\times  \mathcal{P}$ and $\mathcal{I}= \mathfrak{i}  \times \mathcal{P}$ for some Lie algebras $\mathfrak{l}$ and $\mathfrak{i} $, so we will denote $\Phi^{-1}(\mathcal{O} )/(\mathcal{L} \cap \mathcal{I} )=\Phi^{-1}(\mathcal{O} )/(\mathfrak{l}  \cap \mathfrak{i}  )$ as well.

Also, we shall frequently use the following notation: if $X$ is a set, $X_\Delta$ denotes the diagonal inside $X \times X$.

\section{Poisson groupoids and flat bundles over decorated surfaces}\label{sec:quisur} Due to the length of this section, we will start with a brief description of its main result. Let $\Sigma$ be a compact and oriented surface such that each of its connected components has nonempty boundary. Let $V\subset \partial \Sigma$ be a finite set which intersects every connected component of $ \Sigma$. In this situation, $(\Sigma,V)$ is called a {\em marked surface} \cite[\S 4]{quisur}; if in addition, $V$ intersects every component of $\partial \Sigma$, we will say that $(\Sigma,V)$ is a {\em fully marked surface}. Let us consider the fundamental groupoid $\Pi_1(\Sigma,V)$ of $\Sigma$ based at $V$, which is the set of homotopy classes of paths in $\Sigma$ with fixed endpoints in $V$. Let $G$ be a connected Lie group such that its Lie algebra $\mathfrak{g} $ is endowed with an Ad-invariant nondegenerate symmetric bilinear form $\langle \cdot,\cdot \rangle $. The space of groupoid morphisms $\hom(\Pi_1(\Sigma,V),G)$ parameterizes the flat $G$-bundles over $\Sigma$ endowed with a trivialization over $V$. Let us consider the induced orientation on $\partial \Sigma$. This orientation allows us to define a {\em boundary graph $\Gamma $} in which $V$ is the set of vertices and the edges are the oriented boundary arcs in $\partial \Sigma$ lying between pairs of points in $V$ (so $\Gamma $ only contains edges lying in the marked components of $\partial \Sigma$). Let $E$ be the set of edges associated to $\Gamma $. It turns out that there is a distinguished exact Manin pair $(\mathcal{E}_\Gamma , \mathcal{L}_\Gamma ) $ over the product $G^E=\{(g_e)_{e\in E}|g_e\in G\}$ and a canonical Manin pair morphism
\[ R_{\Sigma,V}:(\mathbb{T} \hom(\Pi_1(\Sigma,V),G), T \hom(\Pi_1(\Sigma,V),G))\rightarrow ( \mathcal{E}_\Gamma ,\mathcal{L}_\Gamma )\] 
over the moment map $\mu: \hom(\Pi_1(\Sigma,V),G)  \rightarrow G^E $ which is given by restricting a morphism to each of the edges of $E$ with their positive orientation \cite[\S 5.2]{quisur2}; moreover, $R_{\Sigma, V}$ is an exact morphism of Manin pairs if $(\Sigma,V)$ is fully marked, see \cite[Thm. 4.1]{quisur2}. We will comment on terminology and the construction of $R_{\Sigma,V}$ in \S \ref{subsec:quapoimodspa}. If $G$ is 1-connected, $R_{\Sigma, V}$ can be constructed by reducing the canonical symplectic structure on the space of flat connections on the trivial $G$-bundle over $\Sigma$ which is given by \cite[Thm. 3.1]{meiwoover}. See \cite{alemeimal,cabguamei} for the special case in which $V$ consists of a single point for each connected component of $\partial \Sigma$. 

Suppose that $\mathcal{A} \subset \mathcal{E}_\Gamma$ is a Dirac structure and $H\subset G^E$ is an $\mathcal{A} $-invariant submanifold. Let us suppose as well that $\mathcal{A}$, $R_{\Sigma,V}$ and $R_i$ are cleanly composable, where 
\[ i:\mu^{-1}(H) \hookrightarrow \hom(\Pi_1(\Sigma,V),G) \] 
is the inclusion. It follows from \cite[Thm. 1.1]{quisur2} that the leaf space of the $\mathcal{L}_{\Gamma }\cap \mathcal{A} $-action on $\hom(\Pi_1(\Sigma,V),G)$ that we denote by
\[ \mathfrak{M}_G({\Sigma},{V})_{{H},{\mathcal{A}}}:=\mu^{-1}(H)/(\mathcal{L}_{\Gamma }\cap \mathcal{A}) \] 
inherits a Poisson structure which is symplectic if $(\Sigma,V)$ is fully marked and $TH$ is spanned by the $ \mathcal{A} $-action on $G^E$ (under the usual regularity assumptions on this quotient). Let us denote by $\mathfrak{d} =\mathfrak{g} \oplus \overline{ \mathfrak{g}} $ the direct sum Lie algebra equipped with the direct sum pairing that restricts to $\langle \cdot,\cdot \rangle $ on the first summand and $-\langle \cdot,\cdot \rangle$ on the second one and let us call it the {\em double of $\mathfrak{g} $}. In most of the known examples, the choice of $(H, \mathcal{A})$ depends on what we shall call a {\em decoration of $\partial\Sigma$} which consists of the following {\em boundary data}: 
\begin{enumerate} \item an immersed submanifold $H_e \hookrightarrow G$ for each $e\in E$ and 
\item a Lagrangian Lie subalgebra $\mathfrak{a}_v \hookrightarrow \mathfrak{d}$ for each $v\in V$. \end{enumerate}  
In this situation, we can take 
\[ H=\prod_{e\in E}H_e \quad \text{and } \mathcal{A}=\prod_{v\in V}\mathfrak{a}_v \times G^E. \]
If $V$ consists of a single point for each component of $\partial \Sigma$, one of the most natural decorations is given by taking $H_e \hookrightarrow G$ to be some conjugacy class for each edge $e$ and the diagonal Lie subalgebra $\mathfrak{a}_v=\mathfrak{g}_\Delta$ for every $v\in V$. Then $\mathcal{L}_{\Gamma }=\mathcal{A}$ and $\mathfrak{M}_G({\Sigma},{V})_{{H},{\mathcal{A} }}$ is a (singular) symplectic space whose connected components are the symplectic leaves of the Atiyah-Bott Poisson structure on $\hom(\Pi_1(\Sigma,V),G)/\mathcal{L}_\Gamma  $. 

Now suppose that we take two copies of $\Sigma$ and we identify them along some of the edges in $E$, thus producing a new surface $\widehat{\Sigma} $. In this situation, we can study those decorations $\widehat{H}$, $\widehat{\mathcal{A}} $ of $\widehat{\Sigma}$ which are symmetric with respect to the identification. It turns out that the identification also allows us to induce a groupoid structure on $\mathfrak{M}_G(\widehat{\Sigma},\widehat{V})_{\widehat{H},\widehat{\mathcal{A} }} $ which makes it into a Poisson groupoid, see Theorem \ref{thm:poigro}. 
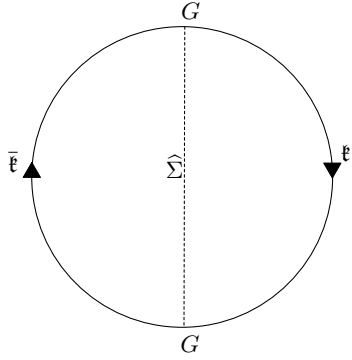
\begin{figure} \begin{center}
\begin{tikzpicture}[line cap=round,line join=round,>=stealth,x=1cm,y=1cm,scale=0.7]
\clip(-3.4825100212811275,-3.8) rectangle (4.672064095702574,1);
\draw [line width=0.9pt] (0.5705491909705938,-1.3763425004443581) circle (2cm);
\draw [<-,line width=0.9pt,dashed] (0.6052317972368269,0.6233567561496383)-- (0.594777037210723,-3.356884689106777);
\begin{scriptsize}

\draw [fill=black,shift={(-1.428533167320169,-1.31576424716282)},rotate=180] (0,0) ++(0 pt,3.75pt) -- ++(3.2475952641916446pt,-5.625pt)--++(-6.495190528383289pt,0 pt) -- ++(3.2475952641916446pt,5.625pt);
\draw[color=black] (-1.6783104978984837,-1.1704395039905213) node {${\mathfrak{a} } $};
\draw [fill=black,shift={(2.567029438536195,-1.257739713460263)},rotate=360] (0,0) ++(0 pt,3.75pt) -- ++(3.2475952641916446pt,-5.625pt)--++(-6.495190528383289pt,0 pt) -- ++(3.2475952641916446pt,5.625pt);
\draw[color=black] (2.9,-1.048120892235766) node {$\mathfrak{a}^\vee $};

\draw[color=black] (0.6865159960267896,0.8376243723167159) node {$G$};

\draw[color=black] (0.67632277838056,-3.6) node {$G$};
\draw[color=black] (0.3,-1.241792027514129) node {$\widehat{\Sigma}$};
\end{scriptsize}
\end{tikzpicture}\end{center}\caption{A decorated disk which induces a Poisson group structure on $\mathfrak{M}_G(\widehat{\Sigma},\widehat{V})_{\widehat{H},\widehat{\mathcal{A} }}$}\label{fig:bigon} \end{figure} 
   
Let us describe the simplest non trivial situation. We can take $\widehat{\Sigma}$ to be a disk with two marked points on its boundary and we can imagine that it is obtained by identifying two disks along one edge, see Figure \ref{fig:bigon}. If there is a Lagrangian Lie subalgebra $\mathfrak{a} \hookrightarrow \mathfrak{d} $ which intersects $\mathfrak{g}_\Delta \hookrightarrow \mathfrak{d} $ trivially, we can decorate both edges of $\partial\widehat{\Sigma}$ with $G$, one vertex with $\mathfrak{a} $ and the other one with $\mathfrak{a}^\vee =\{u\oplus v|v\oplus u\in \mathfrak{a} \}$. Notice that $\mathfrak{M}_G(\widehat{\Sigma},\widehat{V})_{\widehat{H},\widehat{\mathcal{A}}} =G$ as a manifold. Then what we get out of our general observation is a Poisson group structure on $G$, see Example \ref{exa:poigr}.

\subsection{The $\Gamma$-twisted Cartan-Dirac structure}\label{subsec:twicar} Let $\Gamma =(E,V)$ be the boundary graph associated to a marked surface as before. For every $e\in E$, let us denote by $\mathtt{S}(e)$ its source vertex and by $\mathtt{T}(e)$ its target vertex. Consider the action of $\mathfrak{d}^V=\{X_v\oplus Y_v \in \mathfrak{d} |X_v,Y_v\in \mathfrak{g},\, v\in V \}$ on $G^E$ given by 
\begin{align} (X_v\oplus Y_v)_{v\in V}\mapsto (Y_{\mathtt{T}(e)}^r-X_{\mathtt{S}(e)}^l)_{e\in E}\in \mathfrak{X}(G^E)   \label{eq:infact} \end{align}  
for all $X_v\oplus Y_v \in \mathfrak{d} $ (recall that our convention is that the Lie bracket on a Lie algebra is defined using right invariant vector fields). This action has coisotropic stabilizers (in fact Lagrangian) with respect to the product metric on $\mathfrak{d}^V$, so there is a canonical action Courant algebroid structure on $\mathcal{E}_\Gamma := \mathfrak{d}^V \times G^E$ \cite{coupoi}. Let us point out that $\mathcal{E}_{\Gamma }$ is an exact Courant algebroid, see \cite[Rem. 3.1]{quisur2}.  
\begin{defi}[\cite{quisur2}] The diagonal inclusion $\mathfrak{g}^V_\Delta \hookrightarrow \mathfrak{d}^V $ induces a Dirac structure $\mathcal{L}_\Gamma :=\mathfrak{g}^V_\Delta \times G^E$ inside $\mathcal{E}_{\Gamma } $ which we shall call the {\em $\Gamma$-twisted Cartan-Dirac structure}. If $V$ consists of a single point, we recover the Cartan-Dirac structure of Example \ref{ex:quaham}, see \cite[\S 3.3]{purspi}. \end{defi} 

As we have mentioned, there is a canonical morphism of Manin pairs:
\begin{align}  R_{\Sigma,V}:(\mathbb{T}M,TM) \rightarrow (\mathcal{E}_\Gamma ,\mathcal{L}_\Gamma ),\label{eq:canmp} \end{align}   
where $M:=\hom(\Pi_1(\Sigma,V),G)$ and the moment map $\mu: M \rightarrow G^E $ is given by applying a representation to each of the oriented edges of $E$; moreover, $R_{\Sigma, V}$ is an exact morphism of Manin pairs if $(\Sigma,V)$ is fully marked, we will review this construction in \S \ref{subsec:quapoimodspa} based on \cite{quisur,quisur2}. The associated $\mathcal{L}_\Gamma $-action on $M$ integrates to the residual gauge action of $ G^V$ on $M$ given by 
\begin{align} (g\cdot \rho)(e)=g_{\mathtt{T}(e)}\rho(e)g_{\mathtt{S}(e)}^{-1}\quad  \forall \rho \in M, \label{eq:gauact} \end{align}   
where $e\in \Pi_1(\Sigma,V)$ and $g\in G^V$. 
\begin{rema}\label{rem:infact} It is notationally convenient to distinguish between $G^V$ and $G^E$. In case we need to identify $G^V\cong G^E$, we will use the identification $E\cong V$ given by $e\mapsto \mathtt{T}(e)$ for all $e\in E$. So if $\partial \Sigma$ consists of a single component, we shall denote the elements of $V$ by $v_i$ and the elements of $E$ as $e_i$ in such a way that $v_i=\mathtt{T}(e_i)$ and $\mathtt{S} (e_i)=v_{i-1}$, where $i=1\dots n$ and $n$ is the number of elements in $V$. In this situation, the $\mathfrak{g}^n$-action \eqref{eq:infact} on $G^E=G^n$ is given by: 
\[(u_1\oplus v_1,u_2\oplus v_2,\dots,u_n\oplus v_n) \mapsto ( v_{n}^r-u_{n-1}^l,v_{n-1}^r-u_{n-2}^l,\dots,v_1^r-u_{n}^l)\in \mathfrak{X}(G^n), \] 
and it is canonically integrated by the $G^n$-action 
\[ (g_n,\dots,g_1)\cdot (a_n,\dots,a_1)=(g_na_ng_{n-1}^{-1},g_{n-1}a_{n-1}g_{n-2}^{-1},\dots,g_1a_1g_n^{-1}), \] 
for all $g_i,a_i\in G$ and $i=1\dots n$. \end{rema}
\begin{exa}\label{exa:modspatri} For simplicity, a disk $\Sigma$ with $n$ marked points on its boundary shall be called a $n$-gon (or the corresponding polygon) in what follows and shall be represented accordingly in the figures. If $(\Sigma,V)$ is a triangle, then we can use the homotopy classes of two positively oriented boundary edges of $\Sigma$ to identify $ \hom(\Pi_1(\Sigma,V),G)\cong G^2$. Introducing a splitting of $\mathcal{E}_{\Gamma }$ \cite[Rem. 3.1]{quisur2}, $R_{\Sigma,V}$ is given by
\begin{align}  &R_{\Sigma,V}\cong  \{X\oplus \mu^*\alpha -i_X \omega ,T\mu(X)\oplus \alpha )\in \mathbb{T}G^2 \times \mathbb{T}_\eta G^3|X\oplus \alpha \in TG^2 \oplus \mu^*T^*G^3  \} \notag \\
&\mu:G^2 \rightarrow G^3, \quad \mu(a_2,a_1)=(a_2,a_1,a_1^{-1}a_2^{-1}) , \quad \omega_{(a_1,a_2)}=\frac{1}{2}\langle {a_2}^*\theta^l,{a_1}^*\theta^r \rangle\quad \forall a_1,a_2\in G ;\label{eq:polwie} 
\end{align} 
where $\theta^l,\theta^r$ are the Maurer-Cartan 1-forms on $G$ and $\eta=\sum_{i=1}^3 \text{pr}_i^* H$ for $H$ the Cartan 3-form as in Example \ref{ex:quaham}.

The 2-form $\omega\in \Omega^2(G^2) $ serves as the building block to construct $ R_{\Sigma,V}$ for a general surface, see \S \ref{subsec:qhammodspa}. \end{exa}  
\subsection{Boundary decorations} Now let us fix a Lagrangian Lie subalgebra $\mathfrak{a}_v\subset \mathfrak{d} $ for every $v\in V$. The product $\mathcal{A}=\prod_{v\in V} \mathfrak{a}_v \times G^E$ is a Dirac structure inside $\mathcal{E}_\Gamma$. Suppose that $H_e \hookrightarrow G$ is an immersed submanifold for every $e\in E$ such that $H=\prod_{e \in E}H_e$ is a union of $\mathcal{A} $-orbits. The choice of $(H,\mathcal{A})$ is called a {\em coloring of $\partial \Sigma$} in \cite{quisur2}, we shall call it a {\em decoration} in what follows. Let $K_v \hookrightarrow G$ be the connected subgroup integrating $\mathfrak{a}_v\cap \mathfrak{g}_\Delta \hookrightarrow \mathfrak{g}_\Delta $. Suppose that $\mu$ is transverse to $H$ and suppose that $\mathcal{A} $ and $R_{\Sigma,V} $ are cleanly composable. If the action of $K=\prod_{v\in V} K_v \hookrightarrow G^V$ restricted to $\mu^{-1}(H)$ is free and proper, then \cite[Thm. 1.1]{quisur2} implies that there is a canonical Poisson structure on the {\em moduli space associated to the decoration $(H,\mathcal{A} )$}:  
\[ \mathfrak{M}_G(\Sigma,V)_{H,\mathcal{A}}:= \mu^{-1}(H)/K=\mu^{-1}(H)/(\mathcal{L}_{\Gamma } \cap \mathcal{A}  );   \] 
for the sake of brevity, we will also call $\mathfrak{M}_G(\Sigma,V)_{H,\mathcal{A}}$ a {\em decorated moduli space}. Let us stress once again that this Poisson structure is symplectic if $(\Sigma,V)$ is fully marked and $TH$ is spanned by the $\mathcal{A}$-action on $G^E$. 

We will always assume that our decorated moduli spaces are smooth as above. However, many decorated moduli spaces are singular quotients. For instance, one can take $G=SL_2(\mathbb{C})$, $H_e=G$ for every $e\in E$ and $\mathfrak{a}_v $ as in \eqref{eq:lagcoi} for $\mathfrak{c} \hookrightarrow \mathfrak{g}  $ some Borel subalgebra; in this situation, $K_v$ is a Borel subgroup for every $v\in V$. Then $\mathfrak{M}_G(\Sigma,V)_{H,\mathcal{A} }$ is a decorated character variety in the sense of \cite{chemazrub}. If $\mathcal{A}=\mathcal{L}_{\Gamma }  $ and $H=G^E$, then we get the Atiyah-Bott Poisson structure on the full moduli space of flat $G$-bundles over $\Sigma$ which is singular in general.
\subsection{Poisson and symplectic groupoids via gluing and decorating surfaces}\label{subsec:symdec} Let us continue using the notation of \S \ref{subsec:twicar}. Suppose that we take two copies of $\Sigma$ and let $i_a:\Sigma \hookrightarrow \Sigma \coprod \Sigma$ be the two inclusions for $a=1,2$. Now we identify the two surfaces along a proper subset $S\subset E$ in order to get a new surface $\widehat{\Sigma}=\Sigma\cup_{S} \Sigma$:
\[ \Sigma\cup_{S} \Sigma= \Sigma\coprod \Sigma/\sim, \quad i_1(x)\sim i_2(x)\quad  \forall x \in e,\, \forall e\in S. \] 
In order to get an orientation of $\widehat{\Sigma}$, let us reverse the orientation of the copy of $\Sigma$ corresponding to $i_2$. Let us consider the vertex set 
\[ \widehat{V}=i_1(V_0)\cup i_2(V_0)\hookrightarrow \partial\widehat{\Sigma}, \] 
where $V_0\subset V$ is obtained by removing from $V$ the vertices adjacent to edges in $S$. We will assume that $V$ is such that $\widehat{V}$ has non trivial intersection with every component of $\partial \widehat{\Sigma}$ and each connected component of $\widehat{\Sigma} $ has non empty boundary. We will denote by $j_a:E \hookrightarrow \widehat{E}$ for $a=1,2$ the two natural inclusions.

Let $\widehat{E}$ be the edge set associated to $\widehat{V} $ and denote by $\widehat{\Gamma }=(\widehat{E},\widehat{V}) $ the boundary graph of $\widehat{\Sigma} $. Let $\mathfrak{a}_v \hookrightarrow \mathfrak{d} $ be a Lagrangian subalgebra for every $v\in V_0$ and let $H_e\subset G$ be an immersed submanifold for every $e\in {E}-S $. Then we can label the elements of $\widehat{V}$ with the $\mathfrak{a}_v$ and the elements in $\widehat{E}$ with the $H_e$ in a symmetric way with respect to the identified edges in $S \hookrightarrow \widehat{\Sigma}$: 
\begin{align}   \widehat{\mathfrak{a}}_{i_1(v)}=\mathfrak{a}_v, \quad \widehat{\mathfrak{a}}_{i_2(v)}= \mathfrak{a}^\vee_v, \quad \forall v\in V_0, \quad  \widehat{H}_{j_1(e)}=H_e,\quad \widehat{H}_{j_2(e)}=H_e^{-1},  \quad \forall e\in E-S; \label{eq:symdec} \end{align} 
where $\mathfrak{a}^\vee_v=\{u\oplus v\in \mathfrak{d} :v\oplus u\in \mathfrak{a}_v\}$. Notice that $j_1(e)=j_2(e)$ if $e$ is adjacent to an edge in $S$, so the submanifolds $H_e$ corresponding to such an $e$ are forced to satisfy $H_e=H_e^{-1}$ in order to have $\widehat{H}_{j_1(e)}=\widehat{H}_{j_2(e)}$ well defined. In fact, we shall assume that, in such a situation, $H_e$ is a Lie subgroup of $G$. So we obtain boundary data for $\widehat{\Sigma} $:
\[     \widehat{H}=\prod_{e\in \widehat{E}}\widehat{H}_e  \subset G^{\widehat{E}} \quad \text{and}\quad\widehat{ \mathcal{A} } =\prod_{v\in \widehat{V}} \widehat{\mathfrak{a} }_v \times G^{\widehat{E}}  \subset \mathcal{E}_{\widehat{\Gamma }}; \]
provided that $\widehat{H}$ is $\widehat{\mathcal{A} }$-invariant, see Figure \ref{fig:dousur}. 
\begin{defi}\label{def:symdec} In the situation described above, we shall say that the decoration $(\widehat{H},\widehat{\mathcal{A}} )$ of $(\widehat{\Sigma},\widehat{V})  $ is {\em symmetric with respect to the decomposition $\widehat{\Sigma}= \Sigma\cup_{S} \Sigma$}. \end{defi} 

\begin{thm}\label{thm:poigro} Suppose that $H_e $ is a subgroup of $G$ for all $e\in E$ which are adjacent to an edge in $S$ and suppose that the moduli space $\mathfrak{M}_G(\widehat{\Sigma},\widehat{V})_{\widehat{H},\widehat{\mathcal{A} }} $ is smooth. Then there is a groupoid structure on $\mathfrak{M}_G(\widehat{\Sigma},\widehat{V})_{\widehat{H},\widehat{\mathcal{A}}} $ which together with its canonical Poisson structure makes it into a Poisson groupoid. In particular, if the Poisson bracket on $\mathfrak{M}_G(\widehat{\Sigma},\widehat{V})_{\widehat{H},\widehat{\mathcal{A}}} $ is non-degenerate, then we get a symplectic groupoid structure. \end{thm}
\begin{rema} Theorem \ref{thm:poigro} is formulated for the fully marked surface $(\widehat{\Sigma},\widehat{V})$. However, the case of a marked but not fully marked surface can be reduced to this situation as follows. We can add a single marked point for every unmarked boundary component of $\partial \widehat{\Sigma}$ and we can decorate it with the diagonal Lie subalgebra $\mathfrak{g}_\Delta \hookrightarrow \mathfrak{d} $. The presence of these additional decorated marked points does not change the corresponding decorated moduli space, see \S \ref{subsec:forver}. \end{rema}   
\begin{figure}
    \centering
    \begin{tikzpicture}[line cap=round,line join=round,>=triangle 45,x=1cm,y=1cm,scale=0.8]
\clip(1.753545534118291,6.1086062742733054) rectangle (14.633379703553494,10.548421890992035);
\draw [rotate around={-85.60129464505421:(3.04463280670525,7.612680960620123)},line width=1pt] (3.04463280670525,7.612680960620123) ellipse (1.1196333302822727cm and 1.0815611920004022cm);
\fill[rotate around={-85.60129464505421:(3.04463280670525,7.612680960620123)},line width=1pt,fill=black,fill opacity=0.1] (3.04463280670525,7.612680960620123) ellipse (1.1196333302822727cm and 1.0815611920004022cm);
\draw [rotate around={-88.22853025996632:(8.967238417201926,8.08449115452098)},line width=1pt] (8.967238417201926,8.08449115452098) ellipse (1.1073956672484433cm and 0.9675500931955642cm);
\fill[rotate around={-88.22853025996632:(8.967238417201926,8.08449115452098)},line width=1pt,fill=black,fill opacity=0.1] (8.967238417201926,8.08449115452098) ellipse (1.1073956672484433cm and 0.9675500931955642cm);
\draw [shift={(6.417901263749217,3.075720101840686)},line width=0.5pt]  plot[domain=1.0726635963850113:2.2325463216769985,variable=\t]({1*4.545315699480983*cos(\t r)+0*4.545315699480983*sin(\t r)},{0*4.545315699480983*cos(\t r)+1*4.545315699480983*sin(\t r)});
\draw [shift={(6.352854872174809,2.9500861032469357)},line width=0.5pt]  plot[domain=1.1392557067182094:2.1718700812926475,variable=\t]({1*6.830165919331009*cos(\t r)+0*6.830165919331009*sin(\t r)},{0*6.830165919331009*cos(\t r)+1*6.830165919331009*sin(\t r)});
\draw [line width=1pt] (13.032955350446898,7.827786354316632) circle (1.3325771132126754cm);
\draw [line width=1pt] (13.032955350446898,7.827786354316632) circle (0.5393091036934362cm);
\begin{scriptsize}
\draw [fill=black] (3.2444583005926733,8.711721177000946) circle (1.5pt);
\draw[color=black] (3.5,8.931343559827033) node {${\mathfrak{a} }_v $};
\draw[color=black] (3.5,8.3) node {$H_e$};
\draw [fill=black] (9.927510929494265,8.21215744228239) circle (1.5pt);
\draw[color=black] (10.2,8.455811583422125) node {$\mathfrak{a}^\vee_v$};
\draw[color=black] (8.5,8.5) node {$H_e^{-1}$};
\draw[color=black] (6.144628039644512,8.6) node {$\widehat{\Sigma}$};
\draw [fill=black] (11.81523099496611,8.368997178974755) circle (1.5pt);
\draw[color=black] (11.517127602857467,8.476046986673396) node {$\mathfrak{a}_v$};
\draw[color=black] (12.802075709313296,8.1) node {$G$};
\draw[color=black] (12.802075709313296,9.4) node {$H_e$};
\draw[color=black] (12.802075709313296,7) node {$\Sigma$};
\end{scriptsize}
\end{tikzpicture}
\caption{If $S$ is one of the boundary components of the annulus $\Sigma$ with one marked point on each boundary component, then $\widehat{\Sigma}$ is identified with $\Sigma$ itself; the base of the Poisson groupoid structure corresponding to a symmetric decoration is determined by $\Sigma$ with only one marked point }\label{fig:dousur} 
\end{figure} 
\begin{rema} Symplectic groupoids have also appeared in the context of moduli spaces of singular flat connections, see \cite[Proposition 7]{quahammer}. We will explain how to adapt our construction to that setting in a separate work. \end{rema} 
The base of the groupoid structure on $\mathfrak{M}_G(\widehat{\Sigma},\widehat{V})_{\widehat{H},\widehat{\mathcal{A}}} $ is the moduli space corresponding to a decoration of $(\Sigma,V_0)$. Let $\Gamma_0=(E_0,V_0)$ be the boundary graph of $(\Sigma,V_0)$. Note that the set of edges in $E_0$ that do not contain a vertex in $V-V_0$ is canonically identified with the set of edges in $E$ that are not adjacent to an edge in $S$. Also, $(\Sigma,V_0)$ may not be fully marked as Figure \ref{fig:dousur} shows. Let us define $\check{H}=\prod_{e\in E_0}\check{H}_e$, $\check{\mathcal{A}}=\prod_{v\in V_0}\check{\mathfrak{a}}_v \times G^E $, where 
\[ \check{H}_e=\begin{cases} H_e,\quad \text{if $e$ does not contain a vertex in $V-V_0$,} \\

G,\quad \text{if $e$ contains a vertex in $V-V_0$;} \end{cases}  \check{\mathfrak{a}}_v= \mathfrak{a}_v ,\quad \text{if $v\in V_0$,} \]
see Figure \ref{fig:dousur} for an example of this. The base of a Poisson groupoid inherits a unique Poisson structure such that the target map is a Poisson morphism \cite{weicoi}. In our situation, we shall see that the base Poisson manifold of $\mathfrak{M}_G(\widehat{\Sigma},\widehat{V})_{\widehat{H},\widehat{\mathcal{A} }}$ is $ \mathfrak{M}_G(\Sigma,V_0)_{\check{H},\check{\mathcal{A}}}$. Note that if $S=\emptyset$, then $V=V_0$ and $\mathfrak{M}_G(\widehat{\Sigma},\widehat{V})_{\widehat{H},\widehat{\mathcal{A} }} \rightrightarrows \mathfrak{M}_G(\Sigma,V_0)_{\check{H},\check{\mathcal{A}}}$ is the pair groupoid over $ \mathfrak{M}_G(\Sigma,V_0)_{\check{H},\check{\mathcal{A}}}$. 

\begin{prop}\label{rem:baspoigro} The base manifold of the Poisson groupoid structure on $\mathfrak{M}_G(\widehat{\Sigma},\widehat{V})_{\widehat{H},\widehat{\mathcal{A} }}$ defined by Theorem \ref{thm:poigro} is canonically identified with $\mathfrak{M}_G(\Sigma,V_0)_{\check{H},\check{\mathcal{A}}}$. \end{prop}   
This last proposition implies that, if $\mathfrak{M}_G(\widehat{\Sigma},\widehat{V})_{\widehat{H},\widehat{\mathcal{A}}} $ is symplectic, then it is an integration of the Poisson structure on $\mathfrak{M}_G(\Sigma,V_0)_{\check{H},\check{\mathcal{A}}}$. We will give a proof of Proposition \ref{rem:baspoigro} in \S \ref{subsec:quoid}.

Let us use the notation $M:=\hom(\Pi_1(\Sigma,V),G)$ with moment map $\mu:M \rightarrow G^E$ in the course of this section. The proof of Theorem \ref{thm:poigro} is based on the following two observations.
\begin{enumerate} \item The canonical morphism of Manin pairs 
\[ R_{\widehat{\Sigma},\widehat{V}  }: (\mathbb{T}\hom (\Pi_1(\widehat{\Sigma},\widehat{V} ),G),T\hom (\Pi_1(\widehat{\Sigma},\widehat{V} ),G)) \rightarrow (\mathcal{E}_{\widehat{\Gamma}} ,\mathcal{L}_{\widehat{\Gamma }}) \]
inherits a multiplicative structure from the pair groupoid structure on
\[ M \times M \rightrightarrows M. \]
In particular, this means that the canonical moment map $\widehat{\mu}:\hom (\Pi_1(\widehat{\Sigma},\widehat{V} ),G) \rightarrow G^{\widehat{E} }$ is a Lie groupoid morphism with respect to suitable groupoid structures on $\hom (\Pi_1(\widehat{\Sigma},\widehat{V} ),G)$ and on $G^{\widehat{E} }$.
\item A symmetric decoration $(\widehat{H},\widehat{\mathcal{A}} )$ of $(\widehat{\Sigma},\widehat{V})$ consists of a multiplicative Dirac structure $\widehat{\mathcal{A}  }\hookrightarrow \mathcal{E}_{\widehat{\Gamma } } $ and of a Lie subgroupoid $\widehat{H}  \hookrightarrow G^{\widehat{E} } $ with respect to the CA-groupoid structure induced on $\mathcal{E}_{\widehat{\Gamma}} $. \end{enumerate}  
Once these two observations are established, we can deduce the main claim by an application of Proposition \ref{pro:mulmp}.
 
Let us start by describing the groupoid structure on $G^{\widehat{E}  }$. Denote by $E'\hookrightarrow E$ the set of edges that are not adjacent to edges in $S$. Let $\widehat{E}'' $ be the set of edges in $\widehat{E}$ which arise from gluing edges in two different copies of $\Sigma$ and observe that $\widehat{E} -  \widehat{E}'' =j_1(E') \cup j_2(E')$. Let us denote by $G_e$ the copy of $G$ inside $G^{\widehat{E}}$ which corresponds to the edge $e\in E$ and let us denote by $S_0$ the set of edges in $S$ that are not loops. We will need a special label for a factor $G_e$ which corresponds to $e \in \widehat{E}''$, namely, $G_e$ will be denoted by either $G_{\mathtt{T}(f)}$ or $G_{\mathtt{S}(f)}$ depending on whether the vertex $\mathtt{T}(f)$ or $\mathtt{S}(f)$ lies in it for some $f\in S_0$. A factor $G_{\mathtt{T}(f)}$ will be equipped with the group structure on $G$ and we will equip $G_{\mathtt{S}(f)}$ with the opposite group structure $G_{\text{opp} }$ (despite being isomorphic, it is convenient to distinguish between $G$ and $G_{\text{opp} }$). Define the direct product groupoid structure on $G^{\widehat{E} }$ coming from the identification thus established: 
\begin{align}  \left(G^{\widehat{E} } \rightrightarrows G^{E'}\right)\cong \left( \prod_{e\in S_0}\left(G_{\mathtt{T}(e)} \times G_{\mathtt{S}(e)}\right) \rightrightarrows \{1\}\right) \times \left(\prod_{e\in E'} \left(G_{j_1(e) }\times G_{j_2(e)} \rightrightarrows G_e\right)\right), \label{eq:gpdtar1} \end{align} 
where the first factor is a product of copies of the Lie group $G$ and its opposite $G_{\text{opp}}$ and the second factor is the pair groupoid over $G^{E'}$. 

We will denote by $\mathfrak{d}_v$ the copy of $\mathfrak{d} $ corresponding to a vertex $v\in \partial \Sigma$ and we shall label its elements as $X_v\oplus Y_v\in \mathfrak{d}_v $ in what follows. To simplify the notation, in the following discussion, we will use the identification $\mathcal{E}_{\widehat{\Gamma}}=\left(\prod_{v\in V_0}\mathfrak{d}_{i_1(v)} \times \mathfrak{d}_{i_2(v)}\right) \times G^{\widehat{E} }\cong\left(\mathfrak{d}^{V_0} \times \mathfrak{d}^{V_0}\right)\times G^{\widehat{E} } $ defined by considering $  (X_v\oplus Y_v, W_v\oplus Z_v)\in \mathfrak{d}_v \times \mathfrak{d}_v\subset \mathfrak{d}^{V_0} \times \mathfrak{d}^{V_0}  $ as an element of $\mathfrak{d}_{i_1(v)}\times \mathfrak{d}_{i_2(v)} $ for all $v\in V_0$. The VB-groupoid structure on $\mathcal{E}_{\widehat{\Gamma}} \rightrightarrows \mathcal{E}_{\widehat{\Gamma}}^0$ over $G^{\widehat{E} } \rightrightarrows G^{E'}$ is induced by identifying it with the following product of groupoids
\[ \left(\mathcal{E}_{\widehat{\Gamma}} \rightrightarrows \mathcal{E}_{\widehat{\Gamma}}^0\right)\cong \left(\mathfrak{d}^{V_0} \times \mathfrak{d}^{V_0} \rightrightarrows \mathfrak{d}^{V_0}  \right) \times  \left(G^{\widehat{E} } \rightrightarrows G^{E'}\right);   \]
where $\mathfrak{d}^{V_0} \times \mathfrak{d}^{V_0} \rightrightarrows \mathfrak{d}^{V_0} $ is a pair groupoid again. The metric on $\mathcal{E}_{\widehat{\Gamma}}$ is the product of the metrics of all the factors $\mathfrak{d}_v $. In order to make sure the groupoid structures above and the Courant algebroid structure on $\mathcal{E}_{\widehat{\Gamma } } $ determine a CA-groupoid, we have to express them as in the following lemma.
\begin{lem}\label{lem:mulca} The graph of the multiplication on $\mathcal{E}_{\widehat{\Gamma}}$ is a Dirac structure in $\mathcal{E}_{\widehat{\Gamma}} \times \mathcal{E}_{\widehat{\Gamma}} \times \overline{\mathcal{E}_{\widehat{\Gamma}}} $ supported on the graph of the multiplication on $G^{\widehat{E}} \rightrightarrows G^{E'}$. \end{lem} 
\begin{proof} Let us define the multiplication on $G^{E' }\times G^{E'} \rightrightarrows G^{E'}$ as follows:
\[ \mathtt{s}(a,b)=b^{-1},\quad \mathtt{t}(a,b)=a, \quad \mathtt{m}((a,b^{-1}),(b,c))=\mathtt{m}(a,c), \quad \forall a,b\in G^{E'}. \]
Then the unit embedding $\mathtt{u}:G^{E'} \rightarrow G^{E'}\times G^{E'}$ is given by $a\mapsto (a,a^{-1})$ for all $a\in G^{E'}$. The multiplication on $\mathfrak{d}^{V_0} \times \mathfrak{d}^{V_0} \rightrightarrows \mathfrak{d}^{V_0} $ is determined by
\[ \mathtt{m}((W_v\oplus X_v,Y_v\oplus Z_v)_{v\in V_0},(Z_v\oplus Y_v,W'_v\oplus X'_v)_{v\in V_0})=(W_v\oplus X_v,W'_v\oplus X'_v)_{v\in V_0} \]
for all $W_v,W'_v,X_v,X_v',Y_v,Z_v\in \mathfrak{g} $. Now we have that the graph of the multiplication on $\mathcal{E}_{\widehat{\Gamma } }$ is given by: (1) the product of the Lagrangian Lie subalgebra $\mathfrak{L}\hookrightarrow   (\mathfrak{d}^{V_0}\times{\mathfrak{d}^{V_0}} ) \times (\mathfrak{d}^{V_0}\times{\mathfrak{d}^{V_0}}  ) \times \overline{(  \mathfrak{d}^{V_0}\times{\mathfrak{d}^{V_0}})}$
\begin{align*} & \mathfrak{L}=\left\{((W_v\oplus X_v,Y_v\oplus Z_v),(Z_v\oplus Y_v,W'_v\oplus X'_v),(W_v\oplus X_v,W'_v\oplus X'_v))\right\}_{v\in V_0},  \end{align*}  
for all $W_v,W'_v,X_v,X_v',Y_v,Z_v\in \mathfrak{g} $; and (2) the graph $\text{Graph}(\mathtt{m}_{G^{\widehat{E} }} )$ of the multiplication $\mathtt{m}_{G^{\widehat{E} }}$ on $G^{\widehat{E} }$. So we just have to verify that elements in $\mathfrak{L} $ induce tangent vectors to $\text{Graph}(\mathtt{m}_{G^{\widehat{E} }} )$ by means of \eqref{eq:infact}. We will omit this simple verification, see Figure \ref{fig:cagpd}. \end{proof}  
\begin{figure}
    \centering
    \begin{tikzpicture}[line cap=round,line join=round,>=stealth,x=1cm,y=1cm,scale=0.5]
\clip(-12.566694956002456,-2.168078365545015) rectangle (14.116985715454623,8.183749343701505);
\fill[line width=1pt,fill=black,fill opacity=0.0] (-10,6) -- (-10,0) -- (-4,0) -- (-4,6) -- cycle;
\fill[line width=1pt,fill=black,fill opacity=0.0] (-2,6) -- (-2,0) -- (4,0) -- (4,6) -- cycle;
\fill[line width=1pt,fill=black,fill opacity=0.0] (6,6) -- (6,0) -- (12,0) -- (12,6) -- cycle;
\draw [->,line width=1pt] (-10,6)-- (-10,0);
\draw [->,line width=1pt] (-10,0)-- (-4,0);
\draw [->,line width=1pt,dashed] (-7,0)-- (-7,6);
\draw [->,line width=1pt] (-4,0)-- (-4,6);
\draw [->,line width=1pt] (-4,6)-- (-10,6);
\draw [->,line width=1pt] (-2,6)-- (-2,0);
\draw [->,line width=1pt] (-2,0)-- (4,0);
\draw [->,line width=1pt,dashed] (1,0)-- (1,6);
\draw [->,line width=1pt] (4,0)-- (4,6);
\draw [->,line width=1pt] (4,6)-- (-2,6);
\draw [->,line width=1pt] (6,6)-- (6,0);
\draw [->,line width=1pt] (6,0)-- (12,0);
\draw [->,line width=1pt,dashed] (9,0)-- (9,6);
\draw [->,line width=1pt] (12,0)-- (12,6);
\draw [->,line width=1pt] (12,6)-- (6,6);
\begin{scriptsize}
\draw[color=black] (-9.855689436977075,6.4091057169908465) node {$W_v\oplus X_{v}$};
\draw[color=black] (-9.855689436977075,-0.5) node {$W_{v'}\oplus X_{v'}$};
\draw[color=black] (-3,-0.5) node {$Y_{v'}\oplus Z_{v'}\leftrightarrow Z_{v'}\oplus Y_{v'}$};
\draw[color=black] (-3,6.4091057169908465) node {$Y_{v}\oplus Z_{v}\leftrightarrow Z_{v}\oplus Y_{v}$};
\draw[color=black] (-10.22764715850433,3.2288671979328027) node {$a$};
\draw[color=black] (-7.028810753369923,-0.7) node {$h$};
\draw[color=black] (-3.6439954874718867,3.2288671979328027) node {$b$};
\draw[color=black] (-6.935821322988108,6.7) node {$g$};
\draw[color=black] (3.4,-0.5) node {$W'_{v'}\oplus X'_{v'}$};
\draw[color=black] (3.5,6.4091057169908465) node {$W'_{v}\oplus X'_{v}$};
\draw[color=black] (-2.5,3.27) node {$b^{-1}$};
\draw[color=black] (1.0612696898479113,-0.7) node {$h'$};
\draw[color=black] (3.7,3.2288671979328027) node {$c$};
\draw[color=black] (1.0612696898479113,6.7) node {$g'$};
\draw[color=black] (6.8,6.4091057169908465) node {$W_v\oplus X_v$};
\draw[color=black] (6.9,-0.5) node {$W_{v'}\oplus X_{v'}$};
\draw[color=black] (12.145609791360162,-0.5) node {$W'_{v'}\oplus X'_{v'}$};
\draw[color=black] (12.145609791360162,6.4091057169908465) node {$W'_{v}\oplus X'_{v}$};
\draw[color=black] (5.1,3.2288671979328027) node {$\Rightarrow$};
\draw[color=black] (6.3,3.2288671979328027) node {$a$};
\draw[color=black] (9.1,-0.7) node {$h'h$};
\draw[color=black] (12.368784424276516,3.2288671979328027) node {$c$};
\draw[color=black] (9.058360702683931,6.7) node {$gg'$};
\end{scriptsize}
\end{tikzpicture}
    \caption{The CA-groupoid structure on $\mathcal{E}_{\widehat{\Gamma }}$ corresponding to a disk with four marked points $(\Sigma,V)$ and $S$ consisting of one edge}
    \label{fig:cagpd}
\end{figure} 

\begin{rema}\label{rem:muldir} Notice that we can view $\mathcal{L}_{\widehat{\Gamma } }$ as the VB-subgroupoid of $\mathcal{E}_{\widehat{\Gamma}} \rightrightarrows \mathcal{E}_{\widehat{\Gamma}}^0$ given by the product 
\[ \left(\mathcal{L}_{\widehat{\Gamma } } \rightrightarrows \mathcal{L}_{\widehat{\Gamma } }^0\right)= \left(\mathfrak{g}^{ V_0}_\Delta \times \mathfrak{g}^{ V_0}_{\Delta} \rightrightarrows \mathfrak{g}^{ V_0}_\Delta \right) \times \left(  G^{\widehat{E} } \rightrightarrows  G^{E'}\right) \hookrightarrow \left( \mathfrak{d}^{ V_0} \times \mathfrak{d}^{ V_0} \rightrightarrows \mathfrak{d}^{ V_0}\right) \times \left(  G^{\widehat{E} } \rightrightarrows  G^{E'}\right)\cong \mathcal{E}_{\widehat{\Gamma}}, \] 
we have thus described the multiplicative Manin pair structure on $(\mathcal{E}_{\widehat{\Gamma}} ,\mathcal{L}_{\widehat{\Gamma }})$ over $G^{\widehat{E} }$. Similarly, if $(\widehat{H},\widehat{\mathcal{A}}) $ are symmetric boundary data as in \eqref{eq:symdec} for $(\widehat{\Sigma},\widehat{V} )$, we have that $\widehat{ H} \hookrightarrow G^{\widehat{E} }$ is a Lie subgroupoid
\[ \left(\widehat{H}\rightrightarrows \widehat{H}_0\right)  := \left(\left( \prod_{e\in \widehat{E}'' } (H_{j_1(e)=j_2(e)} \rightrightarrows \{1\} )\right)\times\left( \prod_{e\in E'}( H_{j_1(e)} \times H_{j_2(e)} \rightrightarrows H_e) \right)\right) \hookrightarrow \left(  G^{\widehat{E} } \rightrightarrows  G^{E'}\right); \]
and $\widehat{\mathcal{A}}$ is the multiplicative Dirac structure $\widehat{\mathcal{A} }\rightrightarrows \widehat{\mathcal{A} }_0$ given by the product
\[  \left(\left(\prod_{v\in V_0}\left(\mathfrak{a}_v \times { \mathfrak{a}^\vee_v} \rightrightarrows \mathfrak{a}_v \right) \right)\times   \left(  G^{\widehat{E} } \rightrightarrows  G^{E'}\right)\right) \hookrightarrow \left( \mathfrak{d}^{ V_0} \times \mathfrak{d}^{ V_0} \rightrightarrows \mathfrak{d}^{ V_0}\right) \times \left(  G^{\widehat{E} } \rightrightarrows  G^{E'}\right)\cong \mathcal{E}_{\widehat{\Gamma}}. \] 
\end{rema} 
\begin{prop}\label{pro:paigpdmulmp} The morphism of multiplicative Manin pairs 
\[ (R_{\Sigma,V},R_{\Sigma,V}):(\mathbb{T}M \times\overline{\mathbb{T}  M} , T M \times TM) \rightarrow  ( \mathcal{E}_\Gamma\times \overline{\mathcal{E}_\Gamma},\mathcal{L}_\Gamma \times \mathcal{L}_\Gamma )    \]  
over the morphism of pair groupoids 
\[ (\mu,\mu):(M \times M \rightrightarrows M) \rightarrow (G^E \times G^E \rightrightarrows G^E). \] determines via Courant reduction a structure of morphism of multiplicative Manin pairs on 
\[ R_{\widehat{\Sigma},\widehat{V}  }: (\mathbb{T}\hom (\Pi_1(\widehat{\Sigma},\widehat{V} ),G),T\hom (\Pi_1(\widehat{\Sigma},\widehat{V} ),G)) \rightarrow (\mathcal{E}_{\widehat{\Gamma}} ,\mathcal{L}_{\widehat{\Gamma }}). \]

 \end{prop} 
\begin{proof} Note that $(R_{\Sigma,V},R_{\Sigma,V})$ is nothing but $R_{{\Sigma}\coprod \overline{\Sigma},V\coprod V}$, where $\overline{\Sigma} $ denotes $\Sigma$ with the opposite orientation. The Courant reduction this Proposition refers to is the so called quasi-Hamiltonian reduction described in \cite[Thm. 3.1, Thm. 3.2]{quisur2} which we briefly summarize in \S \ref{subsec:redcou}. This construction relies on specifying:
\begin{itemize} \item a coisotropic Lie subalgebra $\mathfrak{c} \hookrightarrow (\mathfrak{d}^V \times \mathfrak{d}^V)$ and 
\item a submanifold $N \hookrightarrow G^E \times G^E$ \end{itemize}
such that $(\mathfrak{c} ,N)$ form reductive data in the sense of \cite[Def. 3.2]{quisur2}. In order to make the output of this reduction process more evident, we will apply it in two steps (but see \S \ref{subsec:poigro} for how to combine them):
\begin{itemize} \item our first step is given by using reductive data coming from the sewing construction of \cite[\S 4.1]{quisur2}, geometrically, this amounts to constructing $\widehat{\Sigma} $ by gluing two copies of $\Sigma$ along $S$;
\item the second step is deleting from $\widehat{\Sigma} $ the vertices adjacent to edges in $S$. \end{itemize}  
{\em Step (1)}. For every $e\in S_0$ (i.e. $\mathtt{T}( e)\neq \mathtt{S}(e) $), let us define the following coisotropic Lie subalgebras: $ \mathfrak{c}_{\mathtt{T}( e)}\subset \mathfrak{d}_{\mathtt{T}(e)}  \times\overline{\mathfrak{d}_{\mathtt{T}(e)} }  $ and $\mathfrak{c}_{\mathtt{S}(e)} \subset   \mathfrak{d}_{\mathtt{S}(e)} \times\overline{\mathfrak{d}_{\mathtt{S}(e)} }$:
\begin{align*} \mathfrak{c}_{\mathtt{T}( e)} &=\{( X\oplus Y,Z\oplus Y, ) \in \mathfrak{d}^2 |X,Y,Z\in \mathfrak{g}     \}, \\ 
\mathfrak{c}_{\mathtt{S}( e)} &=\{(X\oplus Y, X\oplus Z ) \in \mathfrak{d}^2 |X,Y,Z\in \mathfrak{g}     \}; \end{align*}   
on the other hand, if $v=\mathtt{S}(e)=\mathtt{T}(e)$, let us put $\mathfrak{c}_{v}=\{(X\oplus Y,X\oplus Y)|X,Y\in \mathfrak{g} \}\hookrightarrow  \mathfrak{d}_v \times\overline{\mathfrak{d}_v }  $. With this setting, we can define 
\begin{align}  &\mathfrak{c} = \left(\prod_{v\in V_0} (\mathfrak{d}_v \times \overline{\mathfrak{d}_v}) \times \prod_{v\in V-V_0} \mathfrak{c}_v\right) \hookrightarrow \left(\prod_{v\in V_0} ( \mathfrak{d}_v\times\overline{\mathfrak{d}_v }) \times \prod_{v\in V-V_0} ( \mathfrak{d}_v\times\overline{\mathfrak{d}_v} ) \right)= (\mathfrak{d}^V\times\overline{\mathfrak{d}^V } ); \label{eq:coured1} \\
&N=\left(\prod_{e\in E-S} (G_e \times G_e) \times \prod_{e\in S } G_\Delta\right) \hookrightarrow \left(\prod_{e\in E-S} (G_e \times G_e) \times \prod_{e\in S} (G_e \times G_e) \right)= (G^E \times G^E). \label{eq:coured2}  \end{align}  
This choice of reductive data is precisely what is necessary to identify two copies of $\Sigma$ along $S$ and thus to produce the canonical Manin morphism corresponding to the surface $(\widehat{\Sigma},V\cup_S V)$, where $V\cup_S V\subset \partial\widehat{\Sigma} $ is the vertex set which is obtained by taking $i_1(V)\cup i_2(V) \hookrightarrow \partial\widehat{\Sigma} $ and removing the images of those vertices which come from loops in $S$, see \cite[\S 4.1]{quisur2}. Note that $\left(\mathfrak{c}^\perp \cap (\mathfrak{g}_\Delta^V \times \mathfrak{g}_\Delta^V)\right)\cong \mathfrak{g}^{U} $, where $U\subset V$ is the set of vertices that are adjacent to loops in $S$. As a consequence of \cite[Thm. 4.1]{quisur2}, the quotient $\mathcal{M}$ of the $\left(\mathfrak{c}^\perp \cap (\mathfrak{g}_\Delta^V \times \mathfrak{g}_\Delta^V)\right)$-action on $(\mu,\mu)^{-1}(N) $ is identified with $\hom(\Pi_1(\widehat{\Sigma},V\cup_S V),G)$. We will denote by $\Gamma'=(E\cup_S E,V\cup_S V)$ the boundary graph of $(\widehat{\Sigma},V\cup_S V)$ and by $\iota:(\mu,\mu)^{-1}(N) \hookrightarrow M$ the inclusion.  

Note that both $\mathfrak{c}$ and $N$ are Lie subgroupoids of the corresponding pair groupoids. This implies that the reduced Manin pair morphism associated with $(\mathfrak{c},N)$ is multiplicative again. In fact, let us start by observing that the quotient map $q:(\mu,\mu)^{-1}(N) \rightarrow \mathcal{M} $ is given by taking the orbits of the gauge action \eqref{eq:gauact} restricted to $G^{U}_\Delta \hookrightarrow  G^{U} \times G^{U} \hookrightarrow G^V \times G^V$ and this is an action by groupoid automorphisms. So $\mathcal{M} $ inherits a groupoid structure from $M \times M \rightrightarrows M$. According to \cite[Thm. 4.1]{quisur2}, the Manin pair morphism obtained by reducing $R_{{\Sigma}\coprod \overline{\Sigma},V\coprod V}$ along $(\mathfrak{c} ,N)$ is determined by the following diagram
\[ \begin{tikzcd}[column sep=4.5ex,row sep=2.75ex]
{( \mathbb{T}M \times\overline{\mathbb{T}  M}, T M \times TM)} \arrow[rrr, "{R_{\Sigma\coprod\overline{\Sigma} ,V\coprod V}}"] &  &  & {(\mathcal{E}_\Gamma \times\overline{\mathcal{E}_\Gamma} ,\mathcal{L}_\Gamma \times \mathcal{L}_\Gamma )} \arrow[dd, "R_\chi"] \\
{(\mathbb{T}(\mu,\mu)^{-1}(N),T (\mu,\mu)^{-1}(N))}  \arrow[u, "R_{\iota}"]                                         &  &  &                                                                                                                                \\
{(\mathbb{T}\mathcal{M},T\mathcal{M})}\arrow[u, "R_q^\top"] \arrow[rrr, "{R_{\widehat{\Sigma},V\cup_S V}}"']                                               &  &  & {(\mathcal{E}_{\Gamma'},\mathcal{L}_{\Gamma'});}                                                                               
\end{tikzcd}\] 
where the vertical arrows in the left hand side are the canonical relations over the corresponding maps and $R_\chi$ is given by the product $R_{\mathfrak{c} } \times R_{N,\mathfrak{c} }$ of relations given as follows. The Courant morphism $R_{\mathfrak{c} }:(\mathfrak{d}^V \times \mathfrak{d}^V) \rightarrow \mathfrak{c}/\mathfrak{c}^\perp =\mathfrak{d}^{V\cup_S V} $ is the one induced by the coisotropic reduction of $\mathfrak{c}$ as in \S \ref{subsec:redcou}:
\begin{align*} &(X_v\oplus Y_v,X_v\oplus Y_v) \sim_{R_{\mathfrak{c} }} 0,  \quad \text{if $v=\mathtt{T}(e)=\mathtt{S}(e) $ for $e\in S-S_0$}\\
&(X_v\oplus Y_v,Z_v\oplus Y_v) \sim_{R_{\mathfrak{c} }} X_v\oplus Z_v\in \mathfrak{d}_{i_1(v)},  \quad \text{if $v=\mathtt{T}(e)$ for $e\in S_0$}\\
&(X_v\oplus Y_v,X_v\oplus Z_v) \sim_{R_{\mathfrak{c} }} Z_v\oplus Y_v\in \mathfrak{d}_{i_1(v)}, \quad \text{if $v=\mathtt{S}(e)$ for $e\in S_0$} \\
&(X_v\oplus Y_v,W_v\oplus Z_v) \sim_{R_{\mathfrak{c} }} (X_v\oplus Y_v,Z_v\oplus W_v)\in \mathfrak{d}_{i_1(v)} \times \mathfrak{d}_{i_2(v)}, \quad \text{if $v\in V_0$. }
 \end{align*}  
On the other hand, the relation $R_{N, \mathfrak{c} }: G^E \times G^E \dashrightarrow G^{E\cup_S E}$ is obtained by composing the transpose of the inclusion $N \hookrightarrow G^E \times G^E$ and the quotient map $N \rightarrow N/\mathfrak{c}^\perp \cong G^{E\cup_S E}$. Since the bottom arrow in the diagram above is the composite of multiplicative Courant relations $R_{\widehat{\Sigma},V\cup_S V}=R\chi\circ R_{ \Sigma\coprod\overline{\Sigma},V\coprod V}\circ R_\iota\circ R_q^\top$, we have that $R_{\widehat{\Sigma},V\cup_S V}$ is multiplicative as well. Explicitly, $\mathcal{E}_{\Gamma'}$ inherits the following structure of a product of groupoids 
\begin{align}   \left(\prod_{v\in V_0} (\mathfrak{d}_{i_1(v)} \times {\mathfrak{d}_{i_2(v)}} \rightrightarrows \mathfrak{d}_v) \times \prod_{v\in W} (\mathfrak{d}_v \rightrightarrows  \mathfrak{g}_\Delta) \right) \times \left(\prod_{e\in E-S} G_{j_1(e)} \times G_{j_2(e)} \rightrightarrows G_e\right); \label{eq:caredsur} \end{align}  
where $W=V\cup_S V -i_1(V_0)-i_2(V_0)$ and $\mathfrak{g}_\Delta \hookrightarrow \mathfrak{d} $ is the diagonal and all the factors are pair groupoids with $\mathfrak{d}_v \rightrightarrows  \mathfrak{g}_\Delta$ being isomorphic to the pair groupoid over $\mathfrak{g} $. The structure maps of these pair groupoids have to be defined as in the proof of Lemma \ref{lem:mulca} to make them compatible with the Courant algebroid structure on $\mathcal{E}_{\Gamma'}$ (thus determining a CA-groupoid). We omit the details since the approach is the same. It is straightforward to verify that $R_\chi$ is a multiplicative Courant morphism with respect to this groupoid structure.
  
{\em Step (2)}. In order to produce $R_{\widehat{\Sigma},\widehat{V}  }$ from $R_{\widehat{\Sigma},V\cup_S V }$ we have to forget the vertices induced by edges in $S$. This can be accomplished by introducing reductive data in $(\mathcal{E}_{\Gamma'},\mathcal{L}_{\Gamma'})$ corresponding to $R_{\widehat{\Sigma},V\cup_S V}$:
\begin{align} &\mathfrak{c}'= \left(\prod_{v\in V_0} \left(\mathfrak{d}_{i_1(v)} \times {\mathfrak{d}_{i_2(v)}}\right) \times \prod_{v\in W} \mathfrak{g}_\Delta \right) \hookrightarrow   \mathfrak{d}^{V\cup_S V}, \quad N'=G^{E\cup_S E}.\label{eq:coured3} \end{align}  
We claim that the Manin pair morphism obtained from $R_{\widehat{\Sigma},V\cup_S V}$ by applying \cite[Thm. 3.1]{quisur2} to the reduction data $(\mathfrak{c}',N')$ is canonically isomorphic to $R_{\widehat{\Sigma},\widehat{V}  }$. In other words, we have to verify that $R_{\widehat{\Sigma},\widehat{V}  }$ can be identified with $\left(R_{\widehat{\Sigma},V\cup_S V}\right)_{\mathfrak{c}',N' }:=R_\zeta\circ R_{\widehat{\Sigma},V\cup_S V  }  \circ R_\Psi^\top$ as in the diagram below
\[ \begin{tikzcd}[column sep=2ex,row sep=2.5ex]
{(\mathbb{T}\mathcal{M},T\mathcal{M})} \arrow[rrr, "{R_{\widehat{\Sigma},V\cup_S V}}"]                                                                                                    &  &  & {(\mathcal{E}_{\Gamma'},\mathcal{L}_{\Gamma'})} \arrow[dd, "R_\zeta"] \\
                                                                                                                                                                                          &  &  &                                                                       \\
{(\mathbb{T}\hom (\Pi_1(\widehat{\Sigma},\widehat{V} ),G),T\hom (\Pi_1(\widehat{\Sigma},\widehat{V} ),G))} \arrow[rrr, "{R_{\widehat{\Sigma},\widehat{V}  }}"'] \arrow[uu, "R_\Psi^\top"] &  &  & {(\mathcal{E}_{\widehat{\Gamma}} ,\mathcal{L}_{\widehat{\Gamma }});}  
\end{tikzcd}\]
where the Courant relations above are the following. Firstly, $\Psi: \mathcal{M}=\hom (\Pi_1(\widehat{\Sigma},V\cup_S V ),G) \rightarrow \hom (\Pi_1(\widehat{\Sigma},\widehat{V} ),G) $ is the restriction map (equivalently, the quotient projection of $\mathcal{M}$ divided by the gauge action of $G^W$) and, secondly, $R_\zeta$ is the product of the quotient Courant morphism $\mathfrak{d}^{V\cup_S V} \rightarrow \mathfrak{c}'/\left( \mathfrak{c}'\right)^\perp \cong  \mathfrak{d}^{V_0} \times {\mathfrak{d}^{V_0} } \cong \mathfrak{d}^{\widehat{V} }$ and the graph of the quotient map $G^{E\cup_S E}\rightarrow G^{E\cup_S E}/\left( \mathfrak{c}'\right)^\perp\cong G^{\widehat{E} }$ as in \S \ref{subsec:redcou}. Note that $R_\zeta$ is a multiplicative Courant morphism with respect to the CA-groupoid structures of \eqref{eq:caredsur} and of Lemma \ref{lem:mulca}. It follows that $\left(R_{\widehat{\Sigma},V\cup_S V}\right)_{\mathfrak{c}',N' }$ is a morphism of multiplicative Manin pairs as well. At the level of smooth maps, it is not difficult to check that $\left(R_{\widehat{\Sigma},V\cup_S V}\right)_{\mathfrak{c}',N' }$ is supported on the graph of the moment map $\widehat{\mu} $. The claim about the morphisms of Manin pairs follows from the discussion in \S \ref{subsec:forver}, specifically from \eqref{eq:forver} which relates forgetting vertices with quasi-Poisson reduction as in \cite[Corollary 2]{quisur}.   \end{proof}

\begin{rema} We can replace the reductive data \eqref{eq:coured3} in Step (2) of the proof above by reductive data which correspond to forgetting some but not necessarily all of the points in $W \hookrightarrow V\cup_S V$ from $\widehat{\Sigma}$. By a similar argument, we can show that the resulting morphism of Manin pairs is also multiplicative in such a situation. \end{rema}
  
\begin{proof}[Proof of Theorem \ref{thm:poigro}] Note that $R_{\widehat{\Sigma},\widehat{V}  }$ can be equipped with the structure of a morphism of multiplicative Manin pairs as in Proposition \ref{pro:paigpdmulmp}. Since $\widehat{\mathcal{A} } \hookrightarrow \mathcal{E}_{\widehat{\Gamma} } $ is a multiplicative Dirac structure and $\widehat{H}  \hookrightarrow G^{\widehat{E} }$ is a Lie subgroupoid as observed in Remark \ref{rem:muldir}, we can apply Proposition \ref{pro:mulmp} to conclude the result. \end{proof}
\begin{rema} Thanks to Proposition \ref{pro:paigpdmulmp}, we can replace $\widehat{H} $ in Theorem \ref{thm:poigro} by any $\widehat{\mathcal{A} }$-invariant subgroupoid of $G^{\widehat{E} } \rightrightarrows G^{E'}$. We will see an important example of this in \S \ref{subsec:symfol}. \end{rema}

\subsubsection{Direct description of the Lie groupoid structure on $\hom(\Pi_1(\widehat{\Sigma},\widehat{V}),G)$} The gluing and vertex deletion of Steps (1) and (2) in the proof of Proposition \ref{pro:paigpdmulmp} amount to taking $N$ given by \eqref{eq:coured2} and then the quotient
\[ (\mu,\mu)^{-1}(N) / G^{V-V_0}_\Delta \]
determined by the gauge action \eqref{eq:gauact} on $(\mu,\mu)^{-1}(N)$ restricted to 
\[ G^{V-V_0}_\Delta  \hookrightarrow  G^{V-V_0} \times G^{V-V_0}\hookrightarrow G^V \times G^V. \]
In other words, the action induced by forgeting each vertex in $V-V_0$ is a (diagonal) action by automorphisms of a copy of $G$ on $M \times M \rightrightarrows M$. In this regard, let us note that the action map corresponding to the action \eqref{eq:gauact} of $G^V \times G^V \rightrightarrows G^V$ on $M\times M \rightrightarrows M$ is a Lie groupoid morphism with respect to the product of the pair groupoid structures. This gives a direct construction of the Lie groupoid structure on $(\mu,\mu)^{-1}(N) / G^{V-V_0}_\Delta $. 

We can combine the reductive data $(\mathfrak{c} ,N)$ and $(\mathfrak{c}',N')$ of \eqref{eq:coured1}, \eqref{eq:coured2} and \eqref{eq:coured3} into 
\begin{align}  &\mathfrak{p} = \left(\prod_{v\in V_0} (\mathfrak{d}_v \times \overline{\mathfrak{d}_v}) \times \prod_{v\in V-V_0} \mathfrak{d}_\Delta  \right) \hookrightarrow \left(\prod_{v\in V_0} ( \mathfrak{d}_v\times\overline{\mathfrak{d}_v }) \times \prod_{v\in V-V_0} ( \mathfrak{d}_v\times\overline{\mathfrak{d}_v} ) \right)= (\mathfrak{d}^V\times\overline{\mathfrak{d}^V } ); \label{eq:reddat} \end{align} 
together with $N \hookrightarrow G^E \times G^E$ as in \eqref{eq:coured2}. By applying Theorem \ref{thm:coured1} to $R_{\Sigma,V} \times R_{\Sigma,V}$ with the reductive data $(\mathfrak{p},N)$ we have to reduce $(\mu,\mu)^{-1}(N)$ by the action of $\mathfrak{p}^\perp \cap \mathfrak{g}_\Delta^V$ so we get the quotient above. As explained in the proof of Proposition \ref{pro:paigpdmulmp}, we can identify $(\mu,\mu)^{-1}(N) / G^{V-V_0}_\Delta  $ with $\hom(\Pi_1(\widehat{\Sigma},\widehat{V}),G)$ equipped with the moment map $\widehat{\mu}: \hom(\Pi_1(\widehat{\Sigma},\widehat{V}),G) \rightarrow G^{\widehat{E} }$. So the base manifold of the groupoid structure on $\hom(\Pi_1(\widehat{\Sigma},\widehat{V}),G)$ is $M/G^{V-V_0}$ which is the quotient obtained by restricting the action \eqref{eq:gauact} to $G^{V-V_0}\subset G^V$ and hence $M/G^{V-V_0}$ can be identified with $\hom(\Pi_1({\Sigma},{V}_0),G)$. We get then a Lie groupoid 
\begin{align} \hom(\Pi_1(\widehat{\Sigma},\widehat{V}),G) \rightrightarrows \hom(\Pi_1({\Sigma},{V}_0),G). \label{eq:gromodspa} \end{align}
By construction, there is a geometric interpretation of the structure maps of this groupoid: its target and source maps are induced by pulling back a flat framed principal $G$-bundle on $(\widehat{\Sigma},\widehat{V}) $ (i.e. equipped with a trivialization over $ \widehat{V}$) using the inclusions $i_a:(\Sigma,V_0) \hookrightarrow (\widehat{\Sigma},\widehat{V})$ for $a=1,2$ respectively; in our convention, $i_2$ is distinguished because it is orientation reversing. So two flat framed $G$-bundles $P,P'$ over $(\widehat{\Sigma}, \widehat{V}) $ are composable if their respective pullbacks $i_2^*P$ and $ i_1^*P'$ are isomorphic as flat framed $G$-bundles over $(\Sigma,V_0)$ (in the natural sense). Then we can glue $P$ and $P'$ to form a flat framed $G$-bundle $P\circ P'$ over $(\widehat{\Sigma},\widehat{V}) $ such that $i_1^*(P\circ P')\cong i_1^*P$ and $i_2^*(P\circ P')\cong i_2^*P'$. This operation corresponds to the multiplication. 

\subsubsection{Direct construction of the groupoid structure on $\mathfrak{M}_G(\widehat{\Sigma},\widehat{V})_{\widehat{H},\widehat{\mathcal{A}}} $  } For future reference, it is convenient to construct $\mathfrak{M}_G(\widehat{\Sigma},\widehat{V})_{\widehat{H},\widehat{\mathcal{A}}} $ directly from the pair groupoid $M \times M \rightrightarrows M$. 

Let us promote $\widehat{H} $ to a suitable Lie subgroupoid $\mathcal{H}\subset G^E \times G^E$. We will keep denoting by $G_e$ the copy of $G$ inside $G^E$ which corresponds to the edge $e\in E$. Let $\mathcal{H}\subset G^E \times G^E$ be the product of the following Lie subgroupoids $\mathcal{H}_e \hookrightarrow G_e \times G_e \rightrightarrows G_e$: 
\begin{align} \mathcal{H}_e =\begin{cases} H_e \times H_e \rightrightarrows H_e, \quad \text{if $e$ is not adjacent to an edge in $S$}  \\
G_\Delta  \hookrightarrow G_e \times G_e \rightrightarrows G_e, \quad  \text{if $e\in S$} \\
H_e \times G_e \hookrightarrow G_e \times G_e \rightrightarrows G_e, \quad \text{if $e\in E-S$ but $\mathtt{S}( e)\in V-V_0$} \\
G_e \times H_e \hookrightarrow G_e \times G_e \rightrightarrows G_e, \quad \text{if $e\in E-S$ but $\mathtt{T}( e)\in V-V_0$}.
  \end{cases} \label{eq:holgro}\end{align}   
The groupoid $H_e \times G_e \rightrightarrows G_e$ in \eqref{eq:holgro} denotes the action groupoid associated to the action of $H_e$ on $G_e$ by left translation: its multiplication is defined by $\mathtt{m}((a,bg),(b,g))=(ab,g)$ for all $a,b\in H_e$ and all $g\in G_e$. It is included in the pair groupoid $G_e \times G_e \rightrightarrows G_e$ by means of $(h,g)\mapsto (hg,g)$. Similarly, $G_e \times H_e$ denotes the action groupoid $G_e \times H_e \rightrightarrows G_e$ with multiplication $\mathtt{m}((g,a),(ga,b))=(g,ab)$ for all $a,b\in H_e$ and all $g\in G_e$; this groupoid is included in $G_e \times G_e \rightrightarrows G_e$ via the map $(g,h)\mapsto (g,gh)$ for all $g\in G_e$ and all $h\in H_e$. So $\mathcal{H}$ is a Lie subgroupoid of $G^E \times G^E \rightrightarrows G^E$.
 
Finally, let us denote by $G_v$ the copy of $G$ inside $G^V$ which corresponds to the vertex $v\in V$. Let $K_v \hookrightarrow G_v$ be the connected subgroup integrating $\mathfrak{a}_v\cap \mathfrak{g}_\Delta \hookrightarrow \mathfrak{g}_\Delta $ and let us denote $\mathcal{K}=\prod_{v\in V} \mathcal{K}_v $, where

\[ \mathcal{K}_v =\begin{cases} K_v \times K_v \rightrightarrows K_v, \quad \text{if $v\in V_0$},  \\
G_\Delta  \hookrightarrow G_v \times G_v \rightrightarrows G_v, \quad  \text{if $v\in V-V_0$}. \end{cases} \] 
\begin{prop}\label{pro:quoid} The reduction of morphisms between Manin pairs as in \S \ref{subsec:redcou} induced by the reductive data \eqref{eq:reddat} determines a diffeomorphism $ (\mu,\mu)^{-1}(\mathcal{H} )/ \mathcal{K}  \cong \mathfrak{M}_G(\widehat{\Sigma},\widehat{V})_{\widehat{H},\widehat{\mathcal{A}  }} $. \end{prop}  

\begin{proof} Let $Q$ be the relation $G^E \times G^E \dashrightarrow G^{\widehat{E} }$ given by taking the leaf space of the foliation on $N$ as in \eqref{eq:coured2} induced by the action of $\mathfrak{p}^\perp$, see \S \ref{subsec:redcou}. Let us write $(a_e,b_e)_{e\in E}\in G^E \times G^E$, where $a_e,b_e\in G_e$. We can define $Q$ in terms of its components as follows: $Q$ is the graph of the map on $N$ that forgets the components which correspond to edges in $S$ and, for all the others, it satisfies
\[ (a_e,b_e)\sim_Q\begin{cases} (a_e,b_e^{-1}), \quad  \text{if $e$ is not adjacent to an edge in $S$}  \\ a_eb_e^{-1} 
, \quad \text{if $\mathtt{S}( e)\in V-V_0$} \\ b_e^{-1}a_e 
, \quad \text{if $\mathtt{T}( e)\in V-V_0$}; 
  \end{cases} \] 
where $\sim_Q$ stands for being related via $Q$, see Figure \ref{fig:glu2sur}. We have the following commutative diagram
\[ \begin{tikzcd}[column sep=4ex,row sep=2.75ex]
M \times M \arrow[rrr, "{(\mu,\mu)}"]                                       &  &  & G^E \times G^E \arrow[dd, "Q", dashed] \\
{(\mu,\mu)^{-1}(\mathcal{H})} \arrow[d, "p"'] \arrow[u, "i"]                &  &  &                                        \\
{\hom(\Pi_1(\widehat{\Sigma},\widehat{V}),G)} \arrow[rrr, "\widehat{\mu}"'] &  &  & G^{\widehat{E}}                       
\end{tikzcd} \]
where $p$ is the restriction of the quotient map $(\mu,\mu)^{-1}(N) \rightarrow  (\mu,\mu)^{-1}(N) / G^{V-V_0}_\Delta $ to $(\mu,\mu)^{-1}(\mathcal{H} )$ and $i$ is the inclusion. Using this description, it is straightforward to verify that $ (\mu,\mu)^{-1}(\mathcal{H} )/ \mathcal{K}  =\widehat{\mu}^{-1}(\widehat{H})/( \mathcal{L}_{\widehat{\Gamma } }  \cap \widehat{\mathcal{A} })$. \end{proof}

\subsubsection{The symplectic leaves of $ \mathfrak{M}_G(\widehat{\Sigma},\widehat{V})_{\widehat{H},\widehat{\mathcal{A}  }}$}\label{subsec:symfol} Given a symmetric decoration $(\widehat{H} ,\widehat{\mathcal{A}})$ of a surface $(\widehat{\Sigma},\widehat{V})$ obtained from a marked surface $(\Sigma,V)$ as in Theorem \ref{thm:poigro}, the symplectic leaves of the Poisson groupoid $ \mathfrak{M}_G(\widehat{\Sigma},\widehat{V})_{\widehat{H},\widehat{\mathcal{A}  }}$ are obtained as follows. According to (the proof of) Proposition \ref{pro:mulmp}, these symplectic leaves are the connected components of the subquotients $\widehat{\mu}^{-1}(\mathcal{O} )/(\mathcal{L}_{\widehat{\Gamma } }\cap \widehat{\mathcal{A}  })$ for all the $\widehat{\mathcal{A}  }$-orbits $\mathcal{O} \hookrightarrow H$. On the other hand, if such an $\widehat{\mathcal{A}  }$-orbit $\mathcal{O} $ happens to be a subgroupoid of $G^{\widehat{E} }$, item (3) of Proposition \ref{pro:mulmp} tells us that $\widehat{\mu}^{-1}(\mathcal{O} )/(\mathcal{L}_{\widehat{\Gamma } }\cap \widehat{\mathcal{A}  })$ is a symplectic groupoid itself. We will see in item (2) of Proposition \ref{pro:intla} that all the $\widehat{\mathcal{A}}$-orbits that pass through the unit submanifold $G^{E'} \hookrightarrow G^{\widehat{E} }$ are subgroupoids and so we get the following result.
\begin{prop}\label{thm:symfol} We have that $ \widehat{\mu}^{-1}(\mathcal{O} )/(\mathcal{L}_{\widehat{\Gamma }} \cap \widehat{\mathcal{A} }) \hookrightarrow \mathfrak{M}_G(\widehat{\Sigma},\widehat{V})_{\widehat{H},\widehat{\mathcal{A} }} $ is a symplectic subgroupoid whenever $\mathcal{O} \hookrightarrow \widehat{H}  \hookrightarrow  G^{\widehat{E} }$ is an $\widehat{\mathcal{A} }$-orbit that passes through the unit submanifold $G^{E'} \hookrightarrow G^{\widehat{E} }$. \qed \end{prop}
On a related note, let us remark that any symplectic leaf that passes through the unit submanifold of a Poisson groupoid contains a symplectic subgroupoid \cite[Prop. 12]{folpoigro}. 

\begin{figure}
\begin{center}  
\begin{tikzpicture}[line cap=round,line join=round,>=stealth,x=1cm,y=1cm,scale=0.7]
\clip(-1.0615388628632676,-2.794238184730122) rectangle (7.696752181111826,5.646564808900906);
\draw [->,line width=0.8pt] (3.240971612489498,4.803579295918301) -- (5.879406789486994,3.281826227027622);
\draw [->,line width=0.8pt] (5.879406789486994,3.281826227027622) -- (3.372345978149124,1.3112107421332186);
\draw [->,line width=0.8pt] (3.240971612489498,4.803579295918301) -- (0.4383184784174672,3.10666040614812);
\draw [->,line width=0.8pt] (0.4383184784174673,3.10666040614812) -- (3.372345978149124,1.3112107421332186);
\draw [line width=0.8pt,dashed] (3.240971612489498,4.803579295918301)-- (3.372345978149124,1.3112107421332186);
\draw [<-,line width=0.8pt] (3.394241705759062,0.9499312365692447) -- (6.437747843540406,-0.34191669241708644);
\draw [<-,line width=0.8pt] (6.437747843540406,-0.3419166924170864) -- (3.635094709468377,-2.2358971306767073);
\draw [<-,line width=0.8pt] (3.635094709468377,-2.2358971306767073) -- (0.10988256426840115,-0.4623431942717444);
\draw [<-,line width=0.8pt] (0.10988256426840123,-0.46234319427174436) -- (3.394241705759062,0.9499312365692447);
\begin{scriptsize}

\draw[color=black] (3.3285545229292484,5.038958367725132) node {$v_3$};

\draw[color=black] (2.8,1.3714239930605476) node {$v_2$};

\draw[color=black] (5.966989699926746,3.5172052988344538) node {$v_1$};
\draw[color=black] (4.67514177094042,4.294503628987246) node {$b_1$};
\draw[color=black] (4.984620408990108,2.258200961263029) node {$b_2$};

\draw[color=black] (0.2631526575379654,3.35298734175992) node {$v_1$};
\draw[color=black] (1.7630099988187005,4.26166003757234) node {$a_1$};
\draw[color=black] (1.7630099988187005,2.126826595603402) node {$a_2$};
\draw[color=black] (3.1752844296596843,2.1778215208804172) node {$a_3=b_3$};

\draw[color=black] (3.7818246161988124,1.085310308376076) node {$v_3$};

\draw[color=black] (6.525330753980158,-0.10653762061025485) node {$v_1$};
\draw[color=black] (5.0035776850894855,0.616021390517693) node {$b_1^{-1}$};

\draw[color=black] (4.1,-2.208527471164285) node {$v_2$};
\draw[color=black] (5.6,-1.3108026391568348) node {$b_2^{-1}$};

\draw[color=black] (0.19746547470815218,-0.22696412246491282) node {$v_1$};
\draw[color=black] (1.411142712087628,-1.4312291410114928) node {$a_2$};
\draw[color=black] (1.620687769354105,0.5393863438829107) node {$a_1$};

\draw[color=black] (2.2447160062373306,3.4) node {$\Sigma$};

\draw[color=black] (4.423340903426135,3.4) node {$\Sigma$};

\draw[color=black] (3.525616071418688,-0.41601625865994395) node {$\widehat{\Sigma}$};
\end{scriptsize}
\end{tikzpicture}
\caption{The gluing of $\widehat{\Sigma}=\Sigma\cup_S \Sigma$}\label{fig:glu2sur} 
\end{center} 
\end{figure}
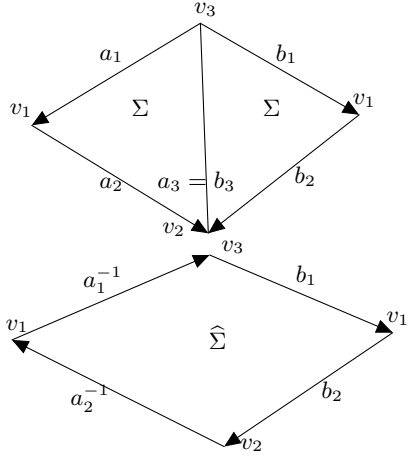 

\subsection{Quasi-Poisson structures on spaces of representations} Let us review an alternative description of quasi-Hamiltonian spaces for Manin pairs that is useful for obtaining explicit formulae for the Poisson structures on the Poisson quotients.
\subsubsection{Quasi-Poisson manifolds and quasi-Poisson groupoids}\label{subsec:quapoi} A {\em Lie quasi-bialgebroid} \cite{roycou} is a Lie algebroid $L$ over $M$ endowed with a degree one derivation $\delta : \Gamma (\wedge^k L) \rightarrow \Gamma ( \wedge^{k+1} L)$ for all $k\in \mathbb{N} $ which is a derivation of the bracket on $L$, 
\[ \delta([u,v])=[\delta(u),v]+(-1)^{p-1}[u,\delta(v)]\] 
for all $u\in \Gamma(\wedge^p L) $ and $v\in \Gamma(\wedge^\bullet L) $ which satisfies $\delta^2(\cdot)=[\chi,\cdot]$, where $\chi\in \Gamma (\wedge^3 L)$ is such that $\delta(\chi)=0$. Since $\delta$ is a derivation, it is determined by its restriction to elements of degree zero and degree one in $\Gamma (\wedge^\bullet L)$, where it is given, respectively, by a vector bundle map $\mathtt{a}_*:L^* \rightarrow TM$ and a map $\Gamma ( L) \rightarrow \Gamma ( \wedge^{2} L)$ called the {\em cobracket}. For example, a {\em Lie bialgebroid} $(L,L^*)$ is a Lie quasi-bialgebroid $(L,\delta,\chi)$ in which $\chi=0$ and hence the differential $\delta$ satisfies $\delta^2=0$ \cite{macxu}. Since the dual of a differential which squares to zero is a Lie bracket, a Lie bialgebroid consists of a pair of Lie algebroid structures on $L$ and $L^*$ which are compatible in a suitable sense. A {\em Lie bialgebra} $(\mathfrak{g},\mathfrak{g}^*)$ is a Lie bialgebroid over a point \cite{driham}. A {\em quasi-Poisson manifold} (or a {\em Hamiltonian space}) for a Lie quasi-bialgebroid $(L,\delta,\chi)$ on $M$ \cite{quapoi} is given by an action $\rho:J^*L \rightarrow TS$ of $L$ on a smooth map $J:S \rightarrow M$ and a bivector field $\pi$ on $S$ such that:
 \begin{align*}  \frac{1}{2}[\pi,\pi]=\rho(\chi),\quad
 \mathcal{L}_{\rho(u)}\pi=\rho(\delta(u)), \quad \forall u\in \Gamma (L) \quad \pi^\sharp J^*=\rho\circ \mathtt{a}^*_*; \end{align*}
			where $\mathtt{a}_*:L^* \rightarrow TS $ is the component of $\delta$ in degree zero as before. Given a Manin pair $({E},{L})$ and a Lagrangian complement for $L$, i.e. a Lagrangian subbundle $K\hookrightarrow {E}$ such that ${E}\cong L\oplus K$, we obtain a canonical Lie quasi-bialgebroid structure on $L$ \cite{roycou}. If $R:(\mathbb{T}M,TM) \rightarrow ({E},L)$ is a Manin pair morphism and $K$ is a Lagrangian complement for $L$, the pullback $K\circ R$ is the graph of a bivector field on $M$ which, together with the induced $L$-action on $M$, makes it into a quasi-Poisson manifold, see \cite{buriglsev}. 

In order to describe the global counterpart of a Lie quasi-bialgebroid, we shall recall a basic fact about the duality of VB-groupoids. The \emph{core of a VB-groupoid $\mathcal{V}  \rightrightarrows \mathcal{V}_0$} over $\mathcal{G}  \rightrightarrows M $ is the vector bundle over $M$ defined by $C=\mathtt{u}_\mathcal{G}^*\ker(\mathtt{s}_{\mathcal{V} })\to M$. Let $\mathcal{V} \rightrightarrows \mathcal{V}_0$ be a VB-groupoid with core $C\to M$. Then $\mathcal{V}^*\rightrightarrows C^*$ is also a VB-groupoid, see \cite{macgen}. The groupoid structure on $\mathcal{V}^* \rightrightarrows C^*$ is determined by requiring that the natural pairing $ \langle \,,\, \rangle :H \times H^*  \rightarrow \mathbb{R} $ be a Lie groupoid morphism. In particular, the tangent and cotangent bundles of a Lie groupoid are VB-groupoids. A multivector field $\Psi\in \mathfrak{X}^k(\mathcal{G} ) $ on a Lie groupoid $\mathcal{G}  \rightrightarrows M$ is multiplicative if the contraction map $\Psi:\bigoplus^k_1 T^*\mathcal{G}  \rightarrow \mathbb{R} $ is a Lie groupoid morphism \cite{burcab}. A \emph{quasi-Poisson groupoid} \cite{quapoi} consists of a Lie groupoid $\mathcal{G} \rightrightarrows M$, a multiplicative bivector field $\pi\in \mathfrak{X}^2(\mathcal{G} )$ and $\psi\in\Gamma(\wedge^3 A_\mathcal{G} )$ such that
\begin{align*} [\pi,\pi]=\psi^r-\psi^l, \quad
	[\pi,\psi^l]=0, \end{align*} 
where $\psi^r$ and $\psi^l$ are the right and left-invariant extensions of $\psi$. The Lie algebroid of a quasi-Poisson groupoid inherits a canonical Lie quasi-bialgebroid structure \cite{quapoi}. The simplest example of a quasi-Poisson groupoid is a pair groupoid $M \times M \rightrightarrows M$ endowed with the multiplicative vector field $(\pi,-\pi)\in \mathfrak{X}^2(M \times M) $ for any $\pi\in \mathfrak{X}^2(M)$. 
\subsubsection{Morphisms of Manin pairs and quasi-Poisson structures on spaces of representations}\label{subsec:quapoimodspa} Take a marked surface $(\Sigma,V)$ and a Lie group $G$ with quadratic Lie algebra as in \S \ref{sec:quisur}. There is a canonical quasi-Poisson manifold structure on $M:=\hom(\Pi_1(\Sigma,V),G)$ determined by a morphism of Manin pairs
\[ \mathbf{R}_{\Sigma,V}:(\mathbb{T}M,TM) \rightarrow (\mathfrak{d}^V,\mathfrak{g}^V_\Delta);   \] 
see \cite[Thm. 2]{quisur} which extends the work in \cite{alemeikos} to surfaces with multiple marked points. We have a distinguished Lagrangian subspace $\mathfrak{g}_\star:=\{X\oplus -X|X \in\mathfrak{g} \} \subset \mathfrak{d} $ such that $\mathfrak{d}=\mathfrak{g}_\Delta\oplus \mathfrak{g}_\star$ and so we can identify $\mathfrak{d}^V=\left(\mathfrak{g}_\Delta\oplus \mathfrak{g}_\star\right)^V$ and its metric with $\left(\mathfrak{g}_\Delta \oplus \mathfrak{g}_\Delta^*\right)^V$ equipped with the canonical pairing defined by evaluation. Using this identification, we have that 
\begin{align}  \mathbf{R}_{\Sigma,V}=\{(\mathtt{a}_{\Sigma,V} (Z)+\pi_{\Sigma,V}(\alpha )\oplus \alpha ,Z\oplus \mathtt{a}_{\Sigma,V}^*\alpha )\in \mathbb{T}M \times  \mathfrak{d}^V|\alpha \oplus Z\in T^*M\times  \mathfrak{g}^V_\Delta   \}\hookrightarrow \mathbb{T}M \times  \overline{\mathfrak{d}^V};  \label{eq:quapoidir}   \end{align} 
where $\pi_{\Sigma,V}\in \mathfrak{X}^2(M)$ is characterized by $\text{Graph}(\pi_{\Sigma,V})=\mathfrak{g}^V_\star \circ \mathbf{R}_{\Sigma,V} $ and $\mathtt{a}_{\Sigma,V}:\mathfrak{g}^V \rightarrow \mathfrak{X}(M)   $ is the infinitesimal counterpart of the gauge action \eqref{eq:gauact}. In order to compute $\pi_{\Sigma,V}$ we can introduce a {\em skeleton} for $(\Sigma,V)$: an embedded 1-dimensional CW-complex $\mathfrak{S}  \hookrightarrow \Sigma$ equipped with an orientation on each of its 1-cells that has the following two properties: \begin{itemize}
    \item its set of vertices (0-dimensional cells) is $\mathfrak{S}_0= V$ and
    \item $\Sigma $ retracts to $\mathfrak{S} $ by deformation.
\end{itemize}  
Let us identify $\mathfrak{S} $ with an oriented graph $(\mathfrak{S}_1,\mathfrak{S}_0)$ in which we denote by $\mathfrak{S}_1 $ its set of edges (1-cells) and we denote by $\mathtt{S}(\mathbf{e} ), \mathtt{T}(\mathbf{e})\in \mathfrak{S}_0 $ respectively the source and target of $\mathbf{e} \in \mathfrak{S}_1$. Note that the choice of $\mathfrak{S}$ allows us to identify $M\cong G^{\mathfrak{S}_1 }$. Using such an identification, we will denote by $Y(\mathbf{e})$ for $Y\in \mathfrak{X}(G)$ the vector field on $M$ which is $Y$ on the factor corresponding to $\mathbf{e} \in \mathfrak{S}_1 $ and $0$ otherwise. The orientation of $\Sigma$ establishes a linear order on the half edges adjacent to $v\in V$, namely, one takes the outward direction to $\partial \Sigma$ as reference and then one orders the half edges in $\mathfrak{S}_1 $ that start or end at $v$ counterclockwise. We have then the following formula \cite[Equation 14]{quisur}:
\begin{align}  &\pi_{\Sigma,V}=\frac{1}{2} \sum_{v\in V}\sum_{\mathbf{e} <\mathbf{e}'} \sum_i \epsilon_i X_i(\mathbf{e}',v)\wedge X_i(\mathbf{e},v)  \label{eq:qpoimodspa}\\
& X_i(\mathbf{e},v)=\begin{cases} -X_i^r(\mathbf{e}), \quad &\text{if $v=\mathtt{T}(\mathbf{e}) $}, \\
X_i^l(\mathbf{e}), \quad &\text{if $v=\mathtt{S}(\mathbf{e}) $}; \end{cases}  \notag\end{align} 
where $\{X_i\}_i\subset \mathfrak{g} $ is an orthonormal basis with $ \epsilon_i=\langle X_i,X_i \rangle =\pm 1$ and $\mathbf{e},\mathbf{e}' \in \mathfrak{S}_1 $ run over the half edges adjacent to $v$. Formula \ref{eq:qpoimodspa} is closely related to the Fock-Rosly Poisson structure \cite{focros}, see \cite[Remark 6]{quisur}; note that in \eqref{eq:qpoimodspa} we follow the sign convention of \cite{quisur2} rather than the one in \cite{quisur}.  

Consider the Manin pair $(\mathcal{E}_\Gamma ,\mathcal{L}_\Gamma )$ corresponding to the boundary graph $\Gamma=(E,V) $ of $(\Sigma,V)$ and the canonical projection $\text{pr}_{(\mathfrak{d}^V,\mathfrak{g}^V_\Delta )}  :(\mathcal{E}_\Gamma ,\mathcal{L}_\Gamma ) \rightarrow (\mathfrak{d}^V,\mathfrak{g}_\Delta^V ) $. We can use \eqref{eq:quapoidir} to define the morphism of Manin pairs $R_{\Sigma,V}$ over $\mu$ as in \eqref{eq:canmp}  by prescribing that $\text{pr}_{(\mathfrak{d}^V,\mathfrak{g}_\Delta^V )}\circ R_{\Sigma,V} =\mathbf{R}_{\Sigma,V} $:
\[ R_{\Sigma,V}=\{(\mathtt{a}_{\Sigma,V} (Z)+\pi_{\Sigma,V}(\alpha )\oplus \alpha ,(Z\oplus \mathtt{a}_{\Sigma,V}^*\alpha,\mu(x)))\in \mathbb{T}M \times  \mathcal{E}_{\Gamma } |\alpha \oplus Z\in T^*_xM\times  \mathfrak{g}^V_\Delta, x\in M   \}. \]
This result is proven in \cite[\S 5.2]{quisur2}. Note that a marked surface in the sense of \cite{quisur2} is a fully marked surface in our convention. However, the arguments in \cite[Prop. 5.2]{quisur2} and the proof of \cite[Thm. 5.1]{quisur2} rely on \cite[Thm. 2, Thm. 4]{quisur} which were proven for marked surfaces (which may not be fully marked). Therefore, the results in \cite[\S 5.2]{quisur2} extend naturally to marked surfaces. When dealing with decorated moduli spaces, the difference between marked and fully marked surfaces is not essential as we can always introduce auxiliary marked points on the boundary of a surface to make it fully marked and then we can decorate these auxiliary points in a way that they are forgotten in the corresponding quotient, see \S \ref{subsec:forver}. For the sake of completeness, we will also give an alternative construction of $R_{\Sigma,V}$ in \S \ref{subsec:mpmor} using differential forms.

Finally, let us mention that an application of \eqref{eq:quapoidir} allows us also to describe the Poisson bivectors on decorated moduli spaces \cite[Thm. 1.4]{quisur2}. If $(H,\mathcal{A})$ is a decoration of $(\Sigma,V)$, the Poisson bivector on $\mathfrak{M}_G(\Sigma,V)_{H,\mathcal{A} }$ is given by the reduction of 
\begin{align} \pi_{{\Sigma},{V}  }+\sum_{v\in {V}} \mathtt{a}_{\Sigma,V} (\pi_v) \in \Gamma \left(\wedge^2\left( T\mu^{-1}(H)/ \mathtt{a}_{\Sigma,V}\left(\bigoplus_{v\in V} \mathfrak{g}_\Delta \cap \mathfrak{a}_v   \right)\right) \right),\label{eq:bivmodspa} \end{align}  
where $\pi_v\in \wedge^2 \left(\mathfrak{g}_\Delta/(\mathfrak{g}_\Delta\cap \mathfrak{a}_v )\right)$ is the bivector obtained as the quotient of the linear Dirac structure $\mathfrak{a}_v\subset \mathfrak{d}_v\cong \mathfrak{g}_\Delta\oplus \mathfrak{g}_\Delta^* $ by its kernel $\mathfrak{a}_v\cap \mathfrak{g}_\Delta$ as in \cite[Prop. 1.1.4]{coudir}. 
\begin{rema}\label{rem:cancomqpoi} Let us point out that a corollary of this discussion is that $\mathcal{L}_{\Gamma }$ admits a distinguished Lagrangian complement: $\mathcal{L}^\star_\Gamma=\mathfrak{g}^V_\star\times G^E$. \end{rema}


\subsubsection{The base manifold of the Poisson groupoid structure on $\mathfrak{M}_G(\widehat{\Sigma},\widehat{V})_{\widehat{H},\widehat{\mathcal{A}}} $}\label{subsec:poigro}\label{subsec:quoid} 
Formula \eqref{eq:bivmodspa} allows us to show that the 
 groupoid structure on $\mathfrak{M}_G(\widehat{\Sigma},\widehat{V})_{\widehat{H},\widehat{\mathcal{A}}} $ behaves well with respect to the decoration.
\begin{proof}[Proof of Proposition \ref{rem:baspoigro}] Consider the inclusion $i:\widehat{\mu}^{-1}(\widehat{H}) \hookrightarrow \hom(\Pi_1(\widehat{\Sigma},\widehat{V}),G)$. The kernel of $\mathfrak{B}_i(\widehat{\mathcal{A}  }\circ R_{\widehat{\Sigma},\widehat{V}})$ is the VB-groupoid $\mathcal{V}\rightrightarrows \mathcal{V}_0 $ given by the action of $ \mathcal{L}_{\widehat{ \Gamma}}\cap\widehat{\mathcal{A}} $ and then its unit vector bundle $\mathcal{V}_0$ is the distribution induced by the action of the Lie algebroid $\mathcal{L}^0_{\widehat{\Gamma } }\cap \widehat{A}_0=  \prod_{v\in V_0} (\mathfrak{a}_v \cap  \mathfrak{g}_\Delta)  \times G^{E'} $ on $\hom(\Pi_1({\Sigma},{V}_0),G)$. Let $\mu_0:\hom (\Pi_1(\Sigma,V_0),G) \rightarrow G^{E_0}$ be the corresponding moment map. By our choice of decoration and the form of the Lie groupoid \eqref{eq:gromodspa}, we have that $\mu^{-1}_0(\check{H})=\left(\widehat{\mu}|_{\hom (\Pi_1(\Sigma,V_0),G)}\right)^{-1}(\widehat{H}_0 ) $. Let $j:\mu^{-1}_0(\check{H}) \hookrightarrow \hom(\Pi_1({\Sigma},{V}_0),G) $ be the inclusion. We also have that $\ker \mathfrak{B}_j(\check{\mathcal{A}}\circ R_{\Sigma,V_0})$ is the distribution determined by the action of $ \prod_{v\in V_0} (\mathfrak{a}_v \cap  \mathfrak{g}_\Delta)  \times G^{E_0} $ and so it coincides with $\mathcal{V}_0 $. It follows that the leaf space of the null foliation of $ \mathfrak{B}_j(\check{\mathcal{A}}\circ R_{\Sigma,V_0})$ on $\mu^{-1}_0(\check{H})$ coincides with the quotient of $\left(\widehat{\mu}|_{\hom (\Pi_1(\Sigma,V_0),G)}\right)^{-1}(\widehat{H}_0 ) $ by the action of $\mathcal{L}^0_{\widehat{\Gamma } }\cap \widehat{A}_0$. In other words, the base of the groupoid structure on $\mathfrak{M}_G(\widehat{\Sigma},\widehat{V})_{\widehat{H},\widehat{L}  } $ is $\mathfrak{M}_G({\Sigma},{V}_0)_{\check{H},\check{L}  }$ as a manifold. 

We have to check that the Poisson structure induced by the Poisson groupoid structure on $\mathfrak{M}_G(\widehat{\Sigma},\widehat{V})_{\widehat{H},\widehat{\mathcal{A}}  } $ over $ \mathfrak{M}_G({\Sigma},{V_0})_{\check{H},\check{\mathcal{A}}  }$ is isomorphic to the canonical one. But this follows from the fact that the target map $\mathtt{t} $ of the groupoid \eqref{eq:gromodspa} is induced by the embedding $i_1:(\Sigma,V_0) \rightarrow (\widehat{\Sigma},\widehat{V} )$. As a consequence, $\pi_{\widehat{\Sigma},\widehat{V}}$ and $\pi_{{\Sigma},{V_0}}$ are $\mathtt{t} $-related, see \cite[Corollary 1]{quisur}. For the same reason, $\mathtt{t}$ commutes with the corresponding gauge actions \eqref{eq:gauact}. Since the decoration of $i_1(V_0)$ induced by $\widehat{\mathcal{A} }$ agrees with the decoration of $V_0$ given by $\check{\mathcal{A} }$, $\mathtt{t}$ is also compatible with the bivector fields induced by the gauge action as in \eqref{eq:bivmodspa} and so we are done. \end{proof}

\subsection{Examples} Now we shall illustrate Theorem \ref{thm:poigro} by giving some examples. 
\begin{exa}\label{exa:poigr} Let $\Sigma$ be the disk with three marked points on its boundary and let $S$ be given by one of the edges. Suppose that $\mathfrak{g} $ is the double of a Lie bialgebra $(\mathfrak{k} ,\mathfrak{k}^* )$ and let $K,K^* \hookrightarrow G$ be the connected integrations of $\mathfrak{k} $ and $\mathfrak{k}^* $ respectively. Notice that $\widehat{V}$ consists of only two points on the boundary of $\widehat{\Sigma}$. Let us decorate $\partial\widehat{\Sigma}$ symmetrically as follows: take $\mathfrak{a}_{v_1}=\mathfrak{k}\oplus \mathfrak{k}^* $ and $H_1=G$, $H_2=K$. We have that $\mathfrak{M}_G(\widehat{\Sigma},\widehat{V})_{\widehat{H},\widehat{\mathcal{A}}  } $ is isomorphic to the Poisson group structure on $K$, see \cite[Example 1.11]{quisur2}. The identification above is given as follows. Following the notation of Proposition \ref{pro:quoid}, we have a subgroupoid of the pair groupoid on $G^3$:
\[ (\mu,\mu)^{-1}(\mathcal{H})=\left\{ (a_i,b_i)_{i=1}^3\in G^3 \times G^3 |a_3=b_3,\, b_2^{-1}a_2\in K, \,\prod_{i=1}^3a_{3-i}=1=\prod_{i=1}^3b_{3-i}\right\} \hookrightarrow G^3 \times G^3 \] 
and we can put $[a_i,b_i]\mapsto a_2^{-1}b_2$ as the desired isomorphism $\mathfrak{M}_G(\widehat{\Sigma},\widehat{V})_{\widehat{H},\widehat{\mathcal{A}}  }=(\mu,\mu)^{-1}(\mathcal{H}) /G^2\cong K $. Note that the base manifold $\mathfrak{M}_G({\Sigma},{V}_0)_{\check{H},\check{\mathcal{A}}  } $ of the Poisson groupoid structure on $\mathfrak{M}_G(\widehat{\Sigma},\widehat{V})_{\widehat{H},\widehat{\mathcal{A}}  }$ is given by the moduli space of flat bundles on a disk with only one marked point so it is a point itself. 

The Poisson structure on $K$ can be computed as follows. Since we can take one of the boundary edges as a skeleton for $(\widehat{\Sigma},\widehat{V})$, \eqref{eq:qpoimodspa} shows that $\pi_{\widehat{\Sigma},\widehat{V}}=0$. Now $\mathfrak{a}_{v_1}\subset \mathfrak{d}\cong \mathfrak{g}_\Delta\oplus \mathfrak{g}_\Delta^* $ has trivial intersection with $\mathfrak{g}_\Delta$ so it can be viewed as the graph of a bivector $\pi_{v_1}\in \wedge^2 \mathfrak{g}_\Delta$ that is the canonical skew-symmetric bilinear form on $\mathfrak{g}_\Delta^* $ induced by the identification $\mathfrak{g}_\Delta\cong \mathfrak{k} \oplus \mathfrak{k}^* $  so $\pi_{v_1}=\frac{1}{2} \sum_i X_i \wedge \xi_i $ where $\{X_i\}_i\subset  \mathfrak{k} $ is any basis and $\{\xi_i\}_i\subset \mathfrak{k}^* $ is its dual basis. Applying \eqref{eq:bivmodspa} we get then that the Poisson structure on $K \hookrightarrow G$ is induced by restricting the bivector field
\[ \pi_G:=\frac{1}{2} \sum_i X_i^l \wedge \xi_i^l-X_i^r \wedge \xi_i^r \]
which is in fact a Poisson tensor on the whole $G$. Note that $(G,\pi_G)$ itself can be constructed as the decorated moduli space that we get by using the decoration $H_1=H_2=G $ and $\mathfrak{a}_{v_1}$ on $(\widehat{\Sigma},\widehat{V} )$ as before. The fact that $K \hookrightarrow (G,\pi_G)$ is a Poisson submanifold goes back to \cite[Prop. 2.36]{luphd}. \end{exa} 

\begin{exa}\label{exa:poigro0} Let ${\Sigma}$ be the disk with four marked points on its boundary and suppose that $S$ consists of exactly one edge and suppose that $\mathfrak{g} $ is the double of a Lie bialgebra $(\mathfrak{k} ,\mathfrak{k}^* )$ as in Example \ref{exa:poigr}. Let us decorate $\partial \widehat{\Sigma}$ as follows: $H_1=G$, $H_2=K$, $H_3=K^*$ and $\mathfrak{a}_{v_1}=\mathfrak{k} \oplus \mathfrak{k}\hookrightarrow \mathfrak{d} =\mathfrak{g} \oplus \mathfrak{g} $, $\mathfrak{a}_{v_2} =\mathfrak{k}^*\oplus \mathfrak{k} \hookrightarrow \mathfrak{d}   $, see Figure \ref{fig:squ}. The Poisson groupoid that we get in this situation is the Poisson group structure on $K^*$, see \cite[Example 8]{quisur}. At the level of the groupoid structure, the isomorphism $\mathfrak{M}_G(\widehat{\Sigma},\widehat{V} )_{\widehat{H},\widehat{\mathcal{A}}}\cong K^* $ is given as follows. In the notation of Proposition \ref{pro:quoid}, we have that 
\[ (\mu,\mu)^{-1}(\mathcal{H})=\left\{(a_i,b_i)_{i=1}^4\in G^4 \times G^4\left|\begin{aligned} &a_2,b_2\in K,\, a_3^{-1}b_3\in K^*,\,a_4=b_4,\\  
&\prod_{i=1}^4a_{4-i}=1=\prod_{i=1}^4b_{4-i} \end{aligned} \right.\right\}, \]
and we have an action of $G^2 \times K^2$ on $(\mu,\mu)^{-1}(\mathcal{H}) $ defined by 
\[ (x,y,u,v)\cdot (a_i,b_i)_{i=1}^4=(xa_4y^{-1},ya_3,a_2u^{-1},ua_1x^{-1},xb_4y^{-1},yb_3,b_2v^{-1},vb_1x^{-1}), \] 
for all $x,y\in G$ and all $u,v\in K$, where $(a_i,b_i)_{i=1}^4\in (\mu,\mu)^{-1}(\mathcal{H})$. Notice that the $G$-factors act by automorphisms while $K \times K \rightrightarrows K$ acts on $(\mu,\mu)^{-1}(\mathcal{H}) $ by an action map which is a Lie groupoid morphism. So the map $[a_i,b_i]\mapsto a^{-1}_3b_3\in K^*$ defines the desired groupoid isomorphism, for all $[a_i,b_i]\in \mathfrak{M}_G(\widehat{\Sigma},\widehat{ V})_{\widehat{H},\widehat{\mathcal{A}}} $. Following Proposition \ref{rem:baspoigro}, the base of the Poisson groupoid structure on $\mathfrak{M}_G(\widehat{\Sigma},\widehat{ V})_{\widehat{H},\widehat{\mathcal{A}}}$ is $\mathfrak{M}_G({\Sigma},{ V_0})_{\check{H},\check{\mathcal{A}}}=K/K=\{1\}$, note that $({\Sigma},{ V_0})$ is a disk with two marked points. \end{exa} 
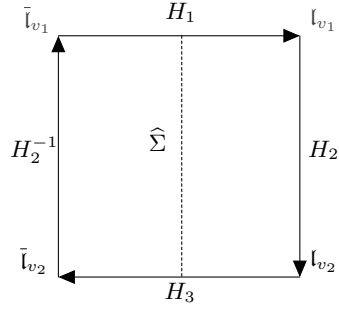
\begin{figure} 
\begin{center}  
 \definecolor{uuuuuu}{rgb}{0.26666666666666666,0.26666666666666666,0.26666666666666666}
\begin{tikzpicture}[line cap=round,line join=round,>=stealth,x=1cm,y=1cm,scale=0.4]
\clip(-4.3,-5.12) rectangle (11.7,5.3);
\draw [<-,line width=0.9pt] (0.44,2.64) -- (6.86,2.64);
\draw [<-,line width=0.9pt] (6.86,2.64) -- (6.86,-3.78);
\draw [<-,line width=0.9pt] (6.86,-3.78) -- (0.44,-3.78);
\draw [<-,line width=0.9pt] (0.44,-3.78) -- (0.44,2.64);
\draw [line width=0.9pt,dashed] (3.72,2.64)-- (3.72,-3.78);
\begin{scriptsize}

\draw[color=black] (-0.3,-3.35) node {${\mathfrak{a}}_{v_2}$};

\draw[color=black] (7.7,-3.35) node {${\mathfrak{a}^\vee_{v_2}}$};

\draw[color=uuuuuu] (7.7,3.07) node {${\mathfrak{a}^\vee_{v_1}}$};

\draw[color=uuuuuu] (-0.3,3.07) node {${\mathfrak{a}}_{v_1}$};
\draw[color=black] (3.7,3.27) node {$H_1$};
\draw[color=black] (7.8,-0.4) node {$H_2^{-1}$};
\draw[color=black] (3.68,-4.19) node {$H_3$};
\draw[color=black] (-0.2,-0.4) node {$H_2$};
\draw[color=black] (3.08,-0.11) node {$\widehat{\Sigma}$};
\end{scriptsize}
\end{tikzpicture}\caption{The decorated surface $(\widehat{\Sigma},\widehat{V} )$. }\label{fig:squ} \end{center}
\end{figure} 
\begin{exa}\label{exa:coipoigro} Let $\Sigma$ be the disk with $n+2$ marked points on its boundary and take $S=\{e_{n+2}\}$. Suppose that there exists a Lagrangian Lie subalgebra $\mathfrak{k} \hookrightarrow \mathfrak{d} $ which is a complement of the diagonal $\mathfrak{g}_\Delta$. Take a coisotropic Lie subalgebra $\mathfrak{c}_i \subset \mathfrak{g} $ for every $i=1\dots n-1$. Let us define
\begin{align}  \widetilde{\mathfrak{c}}_i=\{u\oplus v \in \mathfrak{c}_i \oplus \mathfrak{c}_i :u-v\in \mathfrak{c}_i^\perp\} \label{eq:lagcoi} \end{align} 
for all $i=1\dots n-1$. We have that $\widetilde{\mathfrak{c} }_i \hookrightarrow \mathfrak{d} $ is a Lagrangian Lie subalgebra. Now let us decorate the boundary of $\widehat{\Sigma}$ as follows: put $\mathfrak{a}_{v_i}=\widetilde{\mathfrak{c} }_i$ for all $i=1\dots n-1$ and $\mathfrak{a}_{v_n}=\mathfrak{k} $ at the vertices, take $H_e=G$ for each edge $e$ and let us extend this to a decoration of $\widehat{\Sigma} $ which is symmetric with respect to $S$. Let us suppose that the connected integration $C_i \hookrightarrow G$ of $\mathfrak{c}_i$ is closed for all $i=1\dots n-1$. Then the associated moduli space $\mathfrak{M}_G(\widehat{\Sigma},\widehat{V})_{\widehat{H},\widehat{\mathcal{A}}  }$ is the manifold
\[ \overbrace{G \times_{C_{n-1}} G \times_{C_{n-2}}\dots \times_{C_1} G \times_{C_1} G \times_{C_2} \dots \times_{C_{n-1}}G}^{2n-1} \] 
and its Poisson structure is a mixed product Poisson structure \cite{lumou} for special choices of $\mathfrak{g} $, $\mathfrak{k} $ and the $\mathfrak{c}_i$, see \cite[Example 15]{quisur} and Example \ref{exa:brucel}. \end{exa}  
\begin{exa} Let $(\Sigma,V)$ be an annulus with two marked points $v$ and $v'$, one on each boundary component as in Figure \ref{fig:dousur}. Let $\mathbf{e}_2$ be the oriented boundary component corresponding to $v'$ and let $\mathbf{e}_1 $ be a straight path that joins $v'$ with $v$. Then the map defined by $\rho\mapsto (a_1,a_2):=(\rho(\mathbf{e}_1),\rho(\mathbf{e}_2))$ for $\rho\in M:=\hom(\Pi_1(\Sigma,V),G)$ allows us to identify $M$ with $\{(a_1,a_2)\in G^2\}$. The moment map $\mu:M \rightarrow G^2$ is thus given by $(a_1,a_2)\mapsto (a_1a_2^{-1}a_1^{-1},a_2)$. Suppose that $\mathfrak{a}_v \hookrightarrow \mathfrak{d}$ is a Lagrangian Lie subalgebra which is complementary to the diagonal and take $S=\{\mathbf{e}_2 \}$. So $\mathfrak{a}_v$ and $H_e=G$ extend to a symmetric decoration of $(\widehat{\Sigma},\widehat{V})$ as in Figure \ref{fig:dousur}. In the notation of Proposition \ref{pro:quoid}, we have then that
\[ (\mu,\mu)^{-1}(\mathcal{H} )=\left\{(a_i,b_i)_{i=1}^2\in G^2 \times G^2|a_2=b_2 \right\}. \] 
There is a $G$-action on $(\mu,\mu)^{-1}(\mathcal{H} )$ corresponding to the single vertex adjacent to the loop in $S$:
\[ g\cdot (a_i,b_i)_{i=1}^2=(a_1g^{-1},ga_2g^{-1},b_1g^{-1},gb_2g^{-1}) \quad \forall g\in G, \, \forall (a_i,b_i)_{i=1}^2\in (\mu,\mu)^{-1}(\mathcal{H} ). \] 
Then the map $[a_i,b_i]_{i=1}^2\mapsto (x,y):=(a_1b_1^{-1},b_1b_2b_1^{-1}) $ for all $[a_i,b_i]_{i=1}^2\in \mathfrak{M}_G(\widehat{\Sigma},\widehat{V})_{\widehat{H},\widehat{\mathcal{A} }  }$ is a Lie groupoid isomorphism between $\mathfrak{M}_G(\widehat{\Sigma},\widehat{V})_{\widehat{H},\widehat{\mathcal{A} }  } \rightrightarrows \mathfrak{M}_G({\Sigma},{V_0})_{\check{H},\check{\mathcal{A} }  }$ and the action groupoid $G \times G \rightrightarrows G$ corresponding to the action of $G$ on itself by conjugation. Let us point out that $\mathfrak{M}_G(\widehat{\Sigma},\widehat{V})_{\widehat{H},\widehat{\mathcal{A} }  } \rightrightarrows \mathfrak{M}_G({\Sigma},{V_0})_{\check{H},\check{\mathcal{A} }  }$ integrates, as a Lie groupoid, the (untwisted) Cartan-Dirac structure $L_G$ of Example \ref{ex:quaham}. Note that $(\widehat{\Sigma}, \widehat{V} )$ is identified with $(\Sigma,V)$ itself and so we can use $\mathfrak{S}_1:= \{\mathbf{e}_1,\mathbf{e}_2  \}$ as the set of edges of a skeleton $\mathfrak{S}$ of $(\widehat{\Sigma}, \widehat{V} )$. Furthermore, $\hom(\Pi_1(\widehat{\Sigma}, \widehat{V} ),G) $ is also identified with $\mathfrak{M}_G(\widehat{\Sigma},\widehat{V})_{\widehat{H},\widehat{\mathcal{A} }  }\cong G \times G\cong G^{\mathfrak{S}_1 }$. Using the skeleton $\mathfrak{S}$ and formula \eqref{eq:qpoimodspa} (with the corresponding notation) we get that 
\[ \pi_{\widehat{\Sigma}, \widehat{V}}=\frac{1}{2} \sum_{i}\epsilon_i\left(X_i^l(\mathbf{e}_2) \wedge X_i^r(\mathbf{e}_2)-X_i^l(\mathbf{e}_2) \wedge X_i^l(\mathbf{e}_1) -X_i^r(\mathbf{e}_2)\wedge X_i^l(\mathbf{e}_1)   \right). \]
We are thus recovering the quasi-Poisson bivector field of \cite[Corollary 4.23]{quapoi} which is isomorphic to the one on the double $D(G)$ as in \cite[Example 5.3]{alemeimal}. So $\pi_{\widehat{\Sigma}, \widehat{V}}$ integrates the Lie quasi-bialgebroid structure $\delta: \wedge^\bullet (\mathfrak{g}_\Delta \times G) \rightarrow \wedge^{\bullet+1} (\mathfrak{g}_\Delta \times G)$ on $L_G \cong \mathfrak{g}_\Delta \times G$ determined by the splitting $\mathfrak{d} \times G= (\mathfrak{g}_\Delta \times G )\oplus (\mathfrak{g}_\star \times  G)$.  

On the other hand, we can view $\mathfrak{a}_v \subset \mathfrak{d}\cong  \mathfrak{g}_\Delta \oplus \mathfrak{g}_\Delta^*  $ itself as the graph of the bivector $\pi_v$ that appears in \eqref{eq:bivmodspa}. In fact, we can also view it as a (constant) section of $\wedge^2 (\mathfrak{g}_\Delta \times G )$. The gauge action \eqref{eq:gauact} on $G \times G$ can be written using the multiplication on the action groupoid $G \times G \rightrightarrows G$:
 \[ (g,h)\cdot (x,y)=(gxh^{-1},hyh^{-1})=\mathtt{m}(\mathtt{m} (  (g,xyx^{-1}),(x,y)),\mathtt{i} ((h,y))),\quad \forall (g,h)\in G^{\widehat{V} },\, \forall (x,y)\in G \times G. \]
Since the infinitesimal gauge action $\mathtt{a}_{\widehat{\Sigma},\widehat{V}  } $ is obtained by differentiating the formula above with respect to $(g,h)$, \eqref{eq:bivmodspa} implies that the Poisson bivector on $\mathfrak{M}_G(\widehat{\Sigma},\widehat{V})_{\widehat{H},\widehat{\mathcal{A} }  }$ can also be viewed as $\pi_{\widehat{\Sigma}, \widehat{V}}+\pi_v^r-\pi_v^l$. In other words, we got an integration of the twist of $\delta$ given by $\pi_v$ as in \cite[Prop. 4.15]{quapoi}. The decoration given by $\mathfrak{a}_v$ thus corresponds to equipping $L_G$ with a Lie bialgebroid structure which is a twist of the canonical Lie quasi-bialgebroid structure $\delta$ on $L_G$. An interesting concrete choice of $\mathfrak{a}_v$ is given by \eqref{eq:stapoigro} (for $G$ complex semi-simple). The corresponding Lie bialgebroid is the one given by the {\em Gauss-Cartan} splitting of \cite[\S 3.6]{purspi}. \end{exa} 
Now we shall see some examples which illustrate how Theorem \ref{thm:poigro} produces nontrivial symplectic groupoids even when considering the simplest possible surface.  
\begin{exa}\label{exa:dousym} Suppose that $\Sigma$, $S$ and $\mathfrak{g}$ are as in Example \ref{exa:poigro0}. Now let us decorate the boundary of $\widehat{\Sigma} =\Sigma \cup_S \Sigma$ as follows. Take $\mathfrak{a}_{v_1}=\mathfrak{k}\oplus \mathfrak{k}^*\subset \mathfrak{g} \oplus \mathfrak{g} $ and $\mathfrak{a}_{v_2}=\mathfrak{k}^*\oplus \mathfrak{k}\subset \mathfrak{g} \oplus \mathfrak{g} $. Let $K,K^*\subset G$ be the connected subgroups which integrate $\mathfrak{k} $ and $\mathfrak{k}^* $ respectively and let us put $H_1=H_3=K^*$ and $H_2=K$, see Figure \ref{fig:squ}. Then the symplectic groupoid that we get by applying Theorem \ref{thm:poigro} to this situation is the Lu-Weinstein symplectic groupoid which integrates $K$ \cite{luwei2}, see \cite[Example 1.2]{quisur2}. 

Let us identify $\hom(\Pi_1(\widehat{\Sigma},\widehat{V} ),G)$ with $\{ (x_i)_{i=1}^4\in G^4|\prod_{i=1}^4 x_{4-i}=1\}$. In order to apply the description of $R_{\widehat{\Sigma}, \widehat{V}  }$ provided in \S \ref{subsec:mpmor}, we can triangulate $\widehat{\Sigma}$ by decomposing it along the diagonal joining the vertices $v_3$ and $v_1$. Then $R_{\widehat{\Sigma}, \widehat{V}  }$ is determined by the 2-form on $\hom(\Pi_1(\widehat{\Sigma},\widehat{V} ),G)$ 
\begin{align}  \Omega|_{(x_i)_{i=1}^4}=\frac{1}{2}\left( \langle x_1^*\theta^l, x_4^* \theta^r \rangle +\langle x_3^*\theta^l,x_2^*\theta^r \rangle \right) \label{eq:2forsqu} \end{align} 
as in \eqref{eq:mpmodspa}. Following the proof of Proposition \ref{pro:mulmp}, we can check that the symplectic structure on $\mathfrak{M}_G(\widehat{\Sigma},\widehat{V}  )_{\widehat{H},\widehat{\mathcal{A} }  }=\widehat{\mu}^{-1}(\widehat{H} )\subset \hom(\Pi_1(\widehat{\Sigma},\widehat{V} ),G)$ is obtained by pulling back $\Omega$ to $\widehat{\mu}^{-1}(\widehat{H} )$. 

Interestingly, $\Omega $ is not multiplicative (as in \cite{burcab}) with respect to the groupoid structure on $\hom(\Pi_1(\widehat{\Sigma},\widehat{V} ),G)$ defined in \eqref{eq:gromodspa} and yet its pullback to $\widehat{\mu}^{-1}(\widehat{H} )$ is multiplicative, see \cite[Lemma 6.15]{intpoihom} for a closely related situation. The general explanation of this phenomenon comes from shifted symplectic geometry and it will be treated elsewhere. \end{exa}

\begin{exa} Let us consider the same setting as in the previous example but let us label $v_2$ differently. Let $\mathfrak{l}\subset \mathfrak{g} $ be a Lagrangian subalgebra such that the connected integration $B\subset K$ of $\mathfrak{b} =\mathfrak{k} \cap \mathfrak{l} $ is closed and let us take $\mathfrak{a}_{v_1}=\mathfrak{k} \oplus \mathfrak{k}^*$, $\mathfrak{a}_{v_2}=\mathfrak{l}\oplus \mathfrak{k}\subset \mathfrak{g} \oplus \mathfrak{g} $ and $H_1=K^*$, $H_2=K$, $H_3=L$, where $L \hookrightarrow G$ is the connected integration of $\mathfrak{l} \hookrightarrow \mathfrak{g}  $. Then Theorem \ref{thm:poigro} gives us a symplectic groupoid which integrates the Poisson homogeneous space $K/B$, thus recovering the construction of \cite{intpoihom} for this particular case, see \cite[\S 4.3]{sevmorqua}. The symplectic form on this symplectic groupoid is then obtained by reducing $\Omega $ as in \eqref{eq:2forsqu} along the quotient map of the gauge action of $B^2$ on $\widehat{\mu}^{-1}(\widehat{H} )$. \end{exa}

\begin{exa}\label{exa:brucel} Let us consider $(\Sigma,V)$ and $S=\{e_{n+2}\}$ as in Example \ref{exa:coipoigro} but suppose that $\mathfrak{g}$ is complex semi-simple. Let us use the notation used at the beginning of \S \ref{subsec:symdec}. Choose a Cartan subalgebra $\mathfrak{h} \hookrightarrow \mathfrak{g} $ and a system of positive roots. Denote by $\mathfrak{n}_{\pm}$ the sum of the positive (negative) root spaces and consider the associated Borel subalgebras $\mathfrak{b}_{\pm}=\mathfrak{h} \oplus\mathfrak{n}_{\pm}$. Let $B_{\pm}$ be the Borel subgroups integrating $\mathfrak{b}_{\pm} $. Let us decorate the boundary of $(\widehat{\Sigma},\widehat{V})$ as follows: take $H_1=B_+$ and $H_{i}=G$ for $i=2\dots n+1$. Since the Borel subalgebras are coisotropic, we can put $\mathfrak{a}_{v_i}=\widetilde{\mathfrak{b} }_+$ for all $i=1\dots n-1$ as in \eqref{eq:lagcoi} and 
\begin{align}  \mathfrak{a}_{v_n}=\mathfrak{g}^* := \{X\oplus Y\in \mathfrak{b}_- \oplus \mathfrak{b}_+|\text{pr}_{\mathfrak{h}}(X)=-\text{pr}_{\mathfrak{h} }(Y)\}. \label{eq:stapoigro} \end{align} We can use the $H_i$ and $\mathfrak{a}_{v_i} $ thus described to decorate $\partial \widehat{\Sigma}$ symmetrically with respect to $S$. Let $(\widehat{H},\widehat{\mathcal{A}})$ be such a decoration. Notice that $(\mathfrak{g} \oplus \mathfrak{g} ,\mathfrak{g}_\Delta,\mathfrak{g}^*)$ is a Manin triple which induces the {\em standard Poisson-Lie group structure $\pi_G$ on $G$}. Theorem \ref{thm:poigro} implies then that $\mathfrak{M}_G(\widehat{\Sigma},\widehat{V})_{\widehat{H},\widehat{\mathcal{A}}  }$ is a Poisson groupoid. 

We can find symplectic subgroupoids of $\mathfrak{M}_G(\widehat{\Sigma},\widehat{V})_{\widehat{H},\widehat{\mathcal{A}}  }$ as in Proposition \ref{thm:symfol} in the following manner. Take $u_i$ in the normalizer of the maximal torus $T=B_+\cap B_-$, $i=1\dots n-1$ and consider the $\widehat{\mathcal{A} }  $-orbit $\mathcal{O} $ of the element 
\[ x=(1,1,(u_1,\dots,u_{n-1}),(u_1^{-1},\dots,u_{n-1}^{-1}))\in G_{\mathtt{T}(e_{n+2})} \times G_{\mathtt{S}(e_{n+2})}\times G^{j_1(E')} \times G^{j_2(E')}\cong G^{\widehat{E} }; \] 
where the factors are the ones used in the description of the groupoid structure on $G^{\widehat{E} }$ as in \eqref{eq:gpdtar1} (with the structure maps specified in the proof of Lemma \ref{lem:mulca}). So $x$ is a unit of this groupoid structure and hence $\mathcal{O}\hookrightarrow G^{\widehat{E} } $ is a subgroupoid. As a consequence, \[ \mathcal{O}_{\widehat{H},\widehat{\mathcal{A}}} :=  \widehat{\mu}^{-1}(\mathcal{O} )/(\mathcal{L}_{\widehat{\Gamma } } \cap \widehat{\mathcal{A} } ) \] 
is a symplectic subgroupoid of the Poisson groupoid structure on $\mathfrak{M}_G(\widehat{\Sigma},\widehat{V})_{\widehat{H},\widehat{\mathcal{A}}}$ over the {\em generalized Bruhat cell}: 
\[ B_+u_1B_+/B_+\times_{B_+} B_+u_2B_+/B_+ \times_{B_+}\dots \times_{B_+} B_+u_{n-1}B_+ /B_+  . \] 
For $n=2$, what we obtain is an integration of the Poisson structure on the Bruhat cell $B_+u_1B_+/B_+$ induced by its inclusion as a Poisson submanifold in $G/B_+$, see also \cite{lumou2} for an alternative integration. More generally, the generalized Bruhat cells were also integrated by J.-H. Lu, V. Mouquin and S. Yu \cite{confla}. We shall compare their approach with our own for $n=2$ which corresponds to the surface depicted in Figure \ref{fig:squ}. 

Consider the Poisson structure $(\pi_G,-\pi_G)$ on the pair groupoid $ G \times{G} \rightrightarrows G$, we denote this Poisson manifold by $G \times \overline{G} $. We have that $B_+ \hookrightarrow (G,\pi_G)$ is a Poisson subgroup. Since the diagonal subgroup $B_+ \hookrightarrow B_+\times \overline{B_+} $ is coisotropic, the right $B_+$-action by automorphisms on $G\times \overline{G} $ induces a Poisson groupoid structure on the quotient $( {G}\times  \overline{G})/B_+\rightrightarrows G/B_+$. The integration of $P:=B_+u_1B_+/B_+$ described in \cite{confla} is given by the symplectic leaf of $({G} \times \overline{G} )/B_+\rightrightarrows G/B_+$ which contains $P$. Let us follow again the notation of Proposition \ref{pro:quoid} for the decorated surface $(\widehat{\Sigma},\widehat{V} )$ with decoration $(\widehat{H},\widehat{\mathcal{A}} )$. We have that
\[ (\mu,\mu)^{-1}(\mathcal{H})=\left\{ (a_i,b_i)_{i=1}^4\in G^4 \times G^4\left|\begin{aligned}  &a_4=b_4\in G, \, \prod_{i=1}^4 a_{4-i}=\prod_{i=1}^4 b_{4-i}=1 \\ &b_1a_1^{-1}\in B_+ \end{aligned} \right. \right\} \]
and so we can define the Lie groupoid morphism $\Psi:\mathfrak{M}_G(\widehat{\Sigma} ,\widehat{V})_{\widehat{H},\widehat{\mathcal{A}}}=(\mu,\mu)^{-1}(\mathcal{H} )/\mathcal{K} \rightarrow( {G} \times \overline{G})/B_+ $ as follows: $\Psi([a_i,b_i])=[a_2,b_2b_1a_1^{-1}]$ for all $[a_i,b_i]\in (\mu,\mu)^{-1}(\mathcal{H} )/\mathcal{K}$; we have that $\Psi$ is a actually an isomorphism. On the unit manifold, $\Psi$ is just the identity on $ G/B_+$. We just have to check that $\Psi$ is a  Poisson morphism. Let us describe $\mathfrak{M}_G(\widehat{\Sigma} ,\widehat{V})_{\widehat{H},\widehat{\mathcal{A} }}$ more conveniently as the quotient of 
\[ Q=\left\{(x_i)\in G^4\left| x_1\in B_+,\, \prod_{i=1}^4x_{4-i}=1\right. \right\} \]  
by the $B_+ \times B_+$-action defined by
$(u,v)\cdot (x_i)=(ux_1v^{-1},x_2u^{-1},x_3,vx_4)$ for all $u,v\in B_+$ and all $(x_i)\in Q$. In this description, $\Psi$ is given by $[x_i]\mapsto [x_2,x_4^{-1}x_1^{-1}]$. Let $\Delta_+\subset \mathfrak{h}^* $ be the set of positive roots associated to $\mathfrak{h} $. For each $\alpha\in  \Delta_+$, let $E_{\alpha }\in \mathfrak{g}_{\alpha } $ and $E_{-\alpha }\in \mathfrak{g}_{-\alpha }$ be such that $\langle E_{\alpha },E_{-\alpha}\rangle =1$. The standard {\em quasitriangular $r$-matrix} associated to $(\mathfrak{d},\mathfrak{g}_\Delta,\mathfrak{h} )$ is
\[ r=\frac{1}{2} \sum_a h_a\otimes h_a +\sum_{\alpha \in \Delta_+} E_{-\alpha } \otimes E_{\alpha }; \]  
where $\{h_a\}_a$ is an orthonormal basis of $\mathfrak{h} $. Let $\Lambda$ be the skew-symmetric part of $r$, then $\pi_G=\Lambda^r-\Lambda^l$. On the other hand, the quotient by the right $B_+$-action $G/B_+$ inherits a Poisson structure $\pi_{G/B_+}$ given by the image of $\Lambda$ under the residual left $G$-action, see \cite[Example 6.11]{lumou}. Let $\lambda$ be the left $\mathfrak{g}  \times \mathfrak{g}  $-action on $G/B_+ \times G$. The Poisson structure on $({G}\times \overline{G})/B_+$ is isomorphic to the mixed product Poisson structure on $(G/B_+,\pi_{G/B_+}) \times_{(\lambda,R)} (G,\pi_G)$ by means of the map $J$ defined by $[p,q]\mapsto ([p],pq^{-1})$, see \cite[Lemma 7.7]{lumou}. So it is enough to verify that $J\circ \Psi:(\mathfrak{M}_G(\widehat{\Sigma} ,\widehat{V})_{\widehat{H},\widehat{\mathcal{A}}},\widehat{\pi}) \rightarrow (G/B_+,\pi_{G/B_+}) \times_{(\lambda,R)} (G,\pi_G) $ is a Poisson morphism. Let us describe the Poisson (symplectic) structure $\widehat{\pi} $ on $\mathfrak{M}_G(\widehat{\Sigma} ,\widehat{V})_{\widehat{H},\widehat{\mathcal{A}}}$ using \eqref{eq:bivmodspa}. We can take the oriented boundary edges $\{\mathbf{b}_i  \}_{i=1,2,3}$ corresponding to the coordinates $x_1,x_2,x_3\in G$ as a skeleton for $(\widehat{\Sigma},\widehat{V} )$. Applying \eqref{eq:qpoimodspa} we get 
\[ \pi_{\widehat{\Sigma},\widehat{V}  }=-\frac{1}{2} \sum_{i}\epsilon_{i}( X^r_{i }(\mathbf{e}_1 ) \wedge X^l_{i }(\mathbf{e}_2) + X^r_{i }(\mathbf{e}_2 ) \wedge X^l_{i } (\mathbf{e}_3)). \]  
We can check that $\pi_{v_i}=0$ for $i=1,4$ and $\pi_{v_i}=\Lambda\in \wedge^2 \mathfrak{g} $ for $i=2,3$. Since $J\circ \Psi([x_i])=([x_2],x_3^{-1})$ for all $[x_i]\in \mathfrak{M}_G(\widehat{\Sigma},\widehat{V})_{\widehat{H},\widehat{\mathcal{A}}}$, a calculation shows that it is indeed a Poisson morphism as desired. As a consequence, $\mathfrak{M}_G(\widehat{\Sigma},\widehat{V})_{\widehat{H},\widehat{\mathcal{A}}}   $ is isomorphic to $({G} \times \overline{G})/B_+ $ as a Poisson groupoid over $G/B_+$. Therefore, the symplectic groupoid structure on $\mathcal{O}_{\widehat{H},\widehat{\mathcal{A} }}  \rightrightarrows P$ determined by Proposition \ref{thm:symfol} indeed integrates the Poisson structure on $P$ induced by being a Poisson submanifold of $G/B_+$. The explicit description of the symplectic leaves in $({G}\times \overline{G})/B_+$ is given in \cite{confla}. \end{exa}  

\section{Double Poisson groupoids and decorated surfaces}\label{sec:doupoigro}
The symplectic groupoid of Example \ref{exa:dousym} can be seen as a groupoid over both $K$ and $K^*$, this fact indicates that we may get an additional groupoid structure on a moduli space of flat bundles over a surface if the decoration of the boundary of the corresponding surface is symmetric with respect to two decompositions as a double surface. In \S \ref{subsec:sewdou} we shall see the general condition that ensures the existence of two compatible groupoid structures on the moduli spaces we have described, thereby allowing the systematic construction of double Poisson (symplectic) groupoids, see Theorem \ref{thm:doupoigro}. 

On the other hand, if we consider moduli spaces of flat $\mathcal{G}$-bundles over a compact oriented surface $\Sigma$ (or any manifold), where $\mathcal{G} $ is a Lie 2-group, that is, a group object in the category of Lie groupoids, then the associated representation space carries a natural groupoid structure. If the Lie 2-algebra of $\mathcal{G} $ is quadratic with respect to a bilinear form compatible with its groupoid structure, then we are in a position to use the associated $\Gamma $-twisted Dirac structure, where $\Gamma $ is the boundary graph of $\Sigma$ \S \ref{subsec:catgro}. In this situation, the $\Gamma $-twisted Cartan-Dirac structure is multiplicative and then it can be used to construct Poisson groupoids, see Theorem \ref{thm:twicardir} which is also based on Proposition \ref{pro:mulmp}. 

To conclude, we shall see that, combining the previous two observations, we can construct double and triple Poisson groupoid structures on the moduli space $\mathfrak{M}_{\mathcal{G} }(\Sigma,V)_{H,\mathcal{A}}$, as long as the boundary data $(H,\mathcal{A})$ are compatible with the groupoid structure on $\mathcal{G} $, see Theorems \ref{thm:doupoigro2} and \ref{thm:tripoigro}.
\subsection{Double Poisson groupoids and doubly symmetric decorations}\label{subsec:sewdou} 
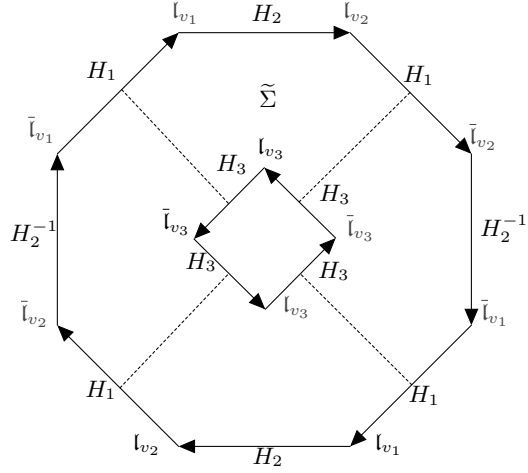
\begin{figure} \begin{center}   
\definecolor{uuuuuu}{rgb}{0.26666666666666666,0.26666666666666666,0.26666666666666666}
\begin{tikzpicture}[line cap=round,line join=round,>=stealth,x=1cm,y=1cm,scale=0.45]
\clip(-2.04,-9.82) rectangle (12.96,2.5);
\draw [<-,line width=0.9pt] (2.78,0.7740115370177616) -- (6.58,0.7740115370177616);
\draw [<-,line width=0.9pt] (6.58,0.7740115370177616) -- (9.26700576850888,-1.9129942314911195);
\draw [<-,line width=0.9pt] (9.26700576850888,-1.9129942314911195) -- (9.26700576850888,-5.71299423149112);
\draw [<-,line width=0.9pt] (9.26700576850888,-5.71299423149112) -- (6.58,-8.4);
\draw [<-,line width=0.9pt] (6.58,-8.4) -- (2.78,-8.4);
\draw [<-,line width=0.9pt] (2.78,-8.4) -- (0.09299423149111874,-5.712994231491119);
\draw [<-,line width=0.9pt] (0.09299423149111874,-5.712994231491119) -- (0.09299423149111874,-1.9129942314911186);
\draw [<-,line width=0.9pt] (0.09299423149111874,-1.9129942314911186) -- (2.78,0.7740115370177616);
\draw [<-,line width=0.9pt] (4.7,-5.36) -- (6.26,-3.78);
\draw [<-,line width=0.9pt] (6.26,-3.78) -- (4.68,-2.22);
\draw [<-,line width=0.9pt] (4.68,-2.22) -- (3.12,-3.8);
\draw [<-,line width=0.9pt] (3.12,-3.8) -- (4.7,-5.36);
\draw [line width=0.9pt,dashed] (1.5183852275020084,-0.48760323548022977)-- (3.8912749660375177,-3.0188368933722574);
\draw [line width=0.9pt,dashed] (7.906772753882146,-0.5527612168643847)-- (5.448923895135416,-2.9791906812729434);
\draw [line width=0.9pt,dashed] (7.945767112011651,-7.03423288798835)-- (5.477091655345839,-4.572945631124085);
\draw [line width=0.9pt,dashed] (1.4817488760843414,-7.101748876084342)-- (3.894586813885597,-4.564781917507299);
\begin{scriptsize}

\draw[color=black] (2.08,-8.45) node {$\mathfrak{a}_{v_2}$};

\draw[color=black] (7.44,-8.45) node {$\mathfrak{a}_{v_1}$};

\draw[color=uuuuuu] (10,-5.43) node {$\mathfrak{a}^\vee_{v_1}$};

\draw[color=uuuuuu] (9.84,-1.53) node {$\mathfrak{a}^\vee_{v_2}$};

\draw[color=uuuuuu] (6.74,1.21) node {$\mathfrak{a}_{v_2}$};

\draw[color=uuuuuu] (2.94,1.21) node {${\mathfrak{a}}_{v_1}$};

\draw[color=uuuuuu] (-0.44,-1.37) node {$\mathfrak{a}^\vee_{v_1}$};

\draw[color=uuuuuu] (-0.5,-5.45) node {$\mathfrak{a}^\vee_{v_2}$};

\draw[color=black] (4.84,-1.79) node {$\mathfrak{a}_{v_3}$};

\draw[color=black] (2.54,-3.51) node {$\mathfrak{a}^\vee_{v_3}$};

\draw[color=uuuuuu] (5,-5.81) node {$\mathfrak{a}_{v_3}$};

\draw[color=uuuuuu] (7.1,-3.7) node {$\mathfrak{a}^\vee_{v_3}$};
\draw[color=black] (4.74,1.2) node {$H_2$};
\draw[color=black] (8.3,0) node {$H_1$};
\draw[color=black] (10.1,-3.57) node {$H_2^{-1}$};
\draw[color=black] (8.52,-7.2) node {$H_1$};
\draw[color=black] (4.76,-8.9) node {$H_2$};
\draw[color=black] (0.96,-7.25) node {$H_1$};
\draw[color=black] (-0.8,-3.69) node {$H_2^{-1}$};
\draw[color=black] (1.08,-0.07) node {$H_1$};
\draw[color=black] (6.42,-4.61) node {$H_3$};
\draw[color=black] (6.34,-2.9) node {$H_3$};
\draw[color=black] (3.8,-2.1) node {$H_3$};
\draw[color=black] (3,-4.6) node {$H_3$};

\draw[color=black] (4.76,-0.61) node {$\widetilde{\Sigma} $};
\end{scriptsize}
\end{tikzpicture}\caption{A decorated surface which induces a double Poisson groupoid structure on $\mathfrak{M}_G(\widetilde{\Sigma},\widetilde{V})_{\widetilde{H},\widetilde{\mathcal{A}}} $} \label{fig:doupoigro}\end{center} \end{figure} 
\subsubsection{Double Lie and Poisson groupoids} A groupoid object in the category of Lie groupoids is denoted by a diagram of the following kind
\[ \xymatrix@-0.4pc{ \mathcal{G}\ar@<-.5ex>[r] \ar@<.5ex>[r]\ar@<-.5ex>[d] \ar@<.5ex>[d]& \mathcal{H} \ar@<-.5ex>[d] \ar@<.5ex>[d] \\ \mathcal{K} \ar@<-.5ex>[r] \ar@<.5ex>[r] & M, }\]  
where each of the sides represents a groupoid structure and the structure maps of $\mathcal{G} $ over $\mathcal{H}$ are groupoid morphisms with respect to $\mathcal{G} \rightrightarrows K$ and $\mathcal{H}  \rightrightarrows M$. 
\begin{defi}[\cite{browmac,macdou}] A groupoid object in the category of Lie groupoids as in the previous diagram is a {\em double Lie groupoid} if the double source map $(\mathtt{s}^h,\mathtt{s}^v):  \mathcal{G} \rightarrow \mathcal{H}  \times_M \mathcal{K}  $ is a submersion, the superindices $\quad^h,\quad^v$ denote the groupoid structures $\mathcal{G} \rightrightarrows \mathcal{H}$, $\mathcal{G} \rightrightarrows \mathcal{K}$ which are often called {\em horizontal} and {\em vertical} respectively. \end{defi}
Originally, the double source map was also required to be surjective but some of the examples that we shall encounter do not need to satisfy that condition, see \cite[Def. 2.2.1]{morla}.
\begin{exa} Let $G \rightrightarrows M$ be a Lie groupoid. Then the pair groupoid $G \times G$ is a double Lie groupoid with sides $M \times M$ and $G$ over $M$. \end{exa} 
\begin{exa} A Lie 2-group is a double Lie groupoid $\mathcal{G} $ with sides $\mathcal{H} $ and a point. So $\mathcal{G} $ and $\mathcal{H}$ are also Lie groups. A Lie 2-group can be equivalently described as a group object in the category of Lie groupoids. \end{exa} 
If $\mathcal{G} $ is a double Lie groupoid with sides $\mathcal{H}$ and $\mathcal{K}$ over $M$, we denote by $A^{\mathcal{K} }$ the Lie algebroid of $\mathcal{G} \rightrightarrows \mathcal{K}$ and by $A^{\mathcal{H} }$ the Lie algebroid of $\mathcal{G} \rightrightarrows \mathcal{H}$. By functoriality, $A^{\mathcal{K} } \rightrightarrows  A_{\mathcal{H} }$ is a VB-groupoid in which the structure maps are Lie algebroid morphisms and hence it is an {\em LA-groupoid} \cite{macdou}. 

Multiplicative Dirac structures are natural examples of LA-groupoids. The LA-groupoids induced by symmetric decorations of a surface $(\widehat{\Sigma},\widehat{V})$ as in \S \ref{subsec:symdec} naturally determine double Lie groupoids, see \S\ref{app:dougro}.

\begin{defi}[\cite{macdousym}] A double Lie groupoid $\mathcal{G} $ with sides $\mathcal{K}$ and $\mathcal{H}$ over $M$ is a {\em double Poisson groupoid} if there is a Poisson structure on $\mathcal{G} $ making it into a Poisson groupoid over both $K$ and $H$. \end{defi}
\begin{exa}[\cite{luwei2}] A {\em double symplectic groupoid} is a double Poisson groupoid whose Poisson bracket is nondegenerate. \end{exa} 
\begin{exa}[\cite{poi2gro}] A double Poisson groupoid structure on a Lie 2-group is called a {\em Poisson 2-group}. \end{exa}


\subsubsection{The double gluing construction} Let $(\Sigma,V)$ be a marked surface with boundary graph $\Gamma =(E,V)$ as in \S \ref{sec:quisur}. Let $S,T\subset E$ be two disjoint subsets such that 
\[ V_{0,0}=\{v\in V|\text{$v$ not adjacent to an edge in $S\cup T$} \}\neq \emptyset \] 
and $(\Sigma,V_{0,0})$ is a marked surface again. Consider the surfaces $\widehat{\Sigma}_S=\Sigma\cup_S \Sigma$ and $\widehat{\Sigma}_T=\Sigma\cup_T \Sigma$. Just as in \S \ref{subsec:symdec}, these are marked surfaces $(\widehat{\Sigma}_S,\widehat{V}_S)$, $(\widehat{\Sigma}_T,\widehat{V}_T)$ with the marked points coming from the corresponding inclusions of the two copies of $V$ after removing the points which are adjacent to edges in $S$ and in $T$, respectively. Let $\widehat{\Gamma  }_S=(\widehat{E}_S,\widehat{V}_S)$ and $\widehat{\Gamma  }_T=(\widehat{E}_T,\widehat{V}_T)$ be the corresponding boundary graphs of $(\widehat{\Sigma}_T,\widehat{V}_T)$ and $(\widehat{\Sigma}_T,\widehat{V}_T)$. The inclusions of the two copies of $\Sigma$ in $\widehat{\Sigma}_S$ (respectively, in $\widehat{\Sigma}_T$) restricted to the edges in $T$ (respectively, in $S$) induce two families of edges $\widehat{T} \hookrightarrow \widehat{E}_S $, $\widehat{S} \hookrightarrow \widehat{E}_T$. So we can consider the following gluing of four copies of $\Sigma$: 
\begin{align} \widetilde{\Sigma} =\widehat{\Sigma  }_S\cup_{\widehat{T} } \widehat{\Sigma  }_S=\widehat{\Sigma  }_T\cup_{\widehat{S} } \widehat{\Sigma  }_T, 
\label{eq:surdougro} \end{align}
see Figure \ref{fig:doupoigro}. Let $\widetilde{V}$ be the set of  points obtained by considering the image of the four copies of $V$ included in $\widetilde{\Sigma} $ and removing from it the vertices which are adjacent to edges either in $S$ or in $T$. Let us suppose that $V$ is such that $(\widetilde{\Sigma},\widetilde{V})$ is a fully marked surface as in \S \ref{subsec:symdec}. Now let us take a decoration $(\widetilde{H}, \widetilde{\mathcal{A} })$ of the boundary graph of $(\widetilde{\Sigma},\widetilde{V} ) $ which is symmetric with respect to both of its decompositions in \eqref{eq:surdougro} as a double surface, see Figure \ref{fig:doupoigro}. Notice that the decoration $(\widetilde{H}, \widetilde{\mathcal{A} } )$ induces decorations $(\check{H}_S,\check{\mathcal{A} }_S)$, $(\check{H}_T,\check{\mathcal{A}  }_T)$ of the boundary graphs associated to the sets of vertices 
\begin{align*}  &\widehat{V}_S^0=\{v\in \widehat{V}_S |\text{$v$ is not adjacent to an edge in $\widehat{T} $} \}\subset \partial \widehat{\Sigma}_S \\
&\widehat{V}_T^0=\{v\in \widehat{V}_T |\text{$v$ is not adjacent to an edge in $\widehat{S} $} \}\subset \partial \widehat{\Sigma}_T \end{align*}  
as in Proposition \ref{rem:baspoigro}. Both of these decorations induce in their turn the same decoration $(\check{H},\check{\mathcal{A}})$ of $(\Sigma,V_{0,0})$ as in Proposition \ref{rem:baspoigro} again. As an application of Theorem \ref{thm:poigro}, we get the following result.
\begin{thm}\label{thm:doupoigro} Suppose that $\mathfrak{M}_G(\widetilde{\Sigma},\widetilde{V})_{\widetilde{H},\widetilde{\mathcal{A} }} $ is smooth. Then there is a double Lie groupoid structure on $\mathfrak{M}_G(\widetilde{\Sigma},\widetilde{V})_{\widetilde{H},\widetilde{\mathcal{A}}} $ with sides $\mathfrak{M}_G(\widehat{\Sigma}_S,\widehat{V}_S^0)_{\check{H}_S,\check{\mathcal{A} }_S }$ and $\mathfrak{M}_G(\widehat{\Sigma}_T,\widehat{V}_T^0)_{\check{H}_T,\check{\mathcal{A}}_T }$ over $\mathfrak{M}_G(\Sigma,V_{0,0})_{\check{H},\check{\mathcal{A} }}$  which together with its canonical Poisson structure makes it into a double Poisson groupoid. \end{thm}
\begin{proof} Since Theorem \ref{thm:poigro} gives us Poisson groupoid structures over both 
\[ \mathfrak{M}_G(\widetilde{\Sigma},\widetilde{V})_{\widetilde{H},\widetilde{\mathcal{A}}} \rightrightarrows \mathfrak{M}_G(\widehat{\Sigma}_S,\widehat{V}_S^0)_{\check{H}_S,\check{\mathcal{A}}_S }, \quad \mathfrak{M}_G(\widetilde{\Sigma},\widetilde{V})_{\widetilde{H},\widetilde{\mathcal{A}}} \rightrightarrows \mathfrak{M}_G(\widehat{\Sigma}_T,\widehat{V}_T^0)_{\check{H}_T,\check{\mathcal{A}}_T }, \] 
we just have to check that they are compatible in the sense of a double Lie groupoid structure. Let us put $M=\hom(\Pi_1(\Sigma,V),G)$ and let $\mu:M \rightarrow G^E$ be the moment map as in \S \ref{sec:quisur}. We have that the moment map $\mu^4:M^4 \rightarrow G^{4E}$ is a morphism of double Lie groupoids between the pair groupoid $M^4$ with both sides equal to $M^2$ over $M$ and the pair groupoid $G^{4E}$ with both sides equal to $G^{2E}$ over $G^E$. So the multiplications on $M^4$ are given by
\begin{align*} &\mathtt{m}^{h}((a,b,c,d),(c,d,a',b'))= (a,b,a',b') \\
&\mathtt{m}^{v}((a,b,c,d),(b,x,d,y))= (a,x,c,y) \end{align*} 
for all $a,b,c,d,a',b',x,y\in M$ and analogous formulae hold for $G^{4E}$. Let us adopt the notation that we used in \S \ref{subsec:quoid} and in Proposition \ref{pro:quoid}. First of all, we have to verify that the subgroupoid $\mathcal{H} \hookrightarrow G^{4E}$ is in fact a double Lie subgroupoid. But we can see that so is each of the factors $\mathcal{H}_e$ and hence their product is a double Lie groupoid. Indeed, the double Lie groupoids that we get are the following in each of the respective cases of \eqref{eq:holgro}:
\[ \xymatrix@-1pc{ H^{4}_e \ar@<-.5ex>[r] \ar@<.5ex>[r]\ar@<-.5ex>[d] \ar@<.5ex>[d]& H_e^{2}\ar@<-.5ex>[d] \ar@<.5ex>[d] \\ H_e^{2} \ar@<-.5ex>[r] \ar@<.5ex>[r] & H_e, } \quad \xymatrix@-1pc{ G_\Delta \times G_\Delta \ar@<-.5ex>[r] \ar@<.5ex>[r]\ar@<-.5ex>[d] \ar@<.5ex>[d]& G_{\Delta}\ar@<-.5ex>[d] \ar@<.5ex>[d] \\ G \times G \ar@<-.5ex>[r] \ar@<.5ex>[r] & G, } \quad \xymatrix@-1pc{ (H_e \times G )\times (H_e \times G)\ar@<-.5ex>[r] \ar@<.5ex>[r]\ar@<-.5ex>[d] \ar@<.5ex>[d]&H_e \times G \ar@<-.5ex>[d] \ar@<.5ex>[d] \\ G \times G \ar@<-.5ex>[r] \ar@<.5ex>[r] & G, }\]
\[\xymatrix@-1pc{ (G \times H_e) \times (G \times H_e) \ar@<-.5ex>[r] \ar@<.5ex>[r]\ar@<-.5ex>[d] \ar@<.5ex>[d]& G \times H_e\ar@<-.5ex>[d] \ar@<.5ex>[d] \\ G \times G \ar@<-.5ex>[r] \ar@<.5ex>[r] & G; }\]
where each of the horizontal groupoid structures is a pair groupoid groupoid structure and the vertical ones are product groupoid structures. We get an analogous description of the corresponding $\mathcal{H}_e $ if $e$ lies in $T$ or is adjacent to an edge in $T$. Then $(\mu^4)^{-1}(\mathcal{H} )$ is also a double Lie groupoid. 

Secondly, we have to check that the action of $\mathcal{K} $ on $(\mu^4)^{-1}(\mathcal{H} )$ induces a double Lie groupoid structure on the quotient. Notice that this action is the restriction of the action of the gauge group which is both a double Lie groupoid, seen as a pair groupoid in two different ways:
\[ \xymatrix@-0.9pc{ G^{4V} \ar@<-.5ex>[r] \ar@<.5ex>[r]\ar@<-.5ex>[d] \ar@<.5ex>[d]& G^{2V}\ar@<-.5ex>[d] \ar@<.5ex>[d] \\ G^{2V} \ar@<-.5ex>[r] \ar@<.5ex>[r] & G^V, }\]      
and it is also a group and all its double Lie groupoid structure maps are Lie group morphisms. The $G^{4V}$-action \eqref{eq:gauact} on $M^4$ is a double Lie groupoid morphism with respect to the structures previously mentioned and then, when restricted to a double Lie subgroupoid which acts freely, it induces a double Lie groupoid structure on the quotient. So it only remains to verify that $\mathcal{K}$ is a double Lie subgroupoid of $G^{4V}$ but this is straightforward. \end{proof}
\begin{rema} The key point of the previous argument is that $G^{4V}$ can be seen as a {\em triple Lie groupoid} and then, when it acts freely on a double Lie groupoid by a double Lie groupoid morphism (which is then a categorified {\em morphic action} \cite{morla}), it induces a double Lie groupoid structure on the quotient. \end{rema} 
\begin{coro} In the situation of Theorem \ref{thm:doupoigro}, if the Poisson structure on $\mathfrak{M}_G(\widetilde{\Sigma},\widetilde{V})_{\widetilde{H},\widetilde{\mathcal{A}}} $ is symplectic, then it is a double symplectic groupoid. \qed\end{coro}
\begin{exa}\label{exa:dousym2} Let us consider $\Sigma$ and $\mathfrak{g} $ as in Example \ref{exa:dousym} and take $S=\{e_3\}$, $T=\{e_4\}$. We can sew four copies of $\Sigma$ to obtain the square $\widetilde{\Sigma}$ as in Figure \ref{fig:dousymsqu}. Since we want a doubly symmetric decoration of $\partial \widetilde{\Sigma}$, it is enough to decorate $v_1$ with $\mathfrak{a}_{v_1}=\mathfrak{k} \oplus \mathfrak{k}^* $ and the edges as follows: $H_1=K$, $H_2=K^*$. Then we have that $\mathfrak{M}_G(\widetilde{\Sigma},\widetilde{V})_{\widetilde{H},\widetilde{\mathcal{A}}} $ is a double symplectic groupoid. Let us see an explicit isomorphism with the following double symplectic groupoid which was introduced in \cite{luwei2}. Let us define the double Lie groupoid
\[ \mathcal{G}=\{(w,x,y,z)\in K^* \times K \times K^* \times K:xy=wz\} \] 
with sides $K$ and $K^*$ over a point in which the source and target maps are the respective projections and the two multiplications are defined as
\[ \mathtt{m}_i((w,x,y,z),(w',x',y',z'))= \begin{cases} (w,xx',y',zz'),\quad \text{if $i=1$ and $y=w'$,} \\
(ww',x,yy',z'),\quad \text{if $i=2$ and $z=x'$}. \end{cases}  \]
Following the notation in the proof of Theorem \ref{thm:doupoigro}, let us put $M=\hom(\Pi_1(\Sigma,V),G)$. Then we have that 
\[ (\mu^4)^{-1}(\mathcal{H})=\left\{(a_i,b_i,c_i,d_i)\in M^4\left|\begin{aligned} &a_4=b_4, \, c_4=d_4, \, a_3=c_3,\, b_3=d_3, \\
 &b_1a_1^{-1},c_1d^{-1}_1\in K^*,\, a_2^{-1}c_2,d_2^{-1}b_2\in K \end{aligned} \right. \right\} \]
and the map $[a_i,b_i,c_i,d_i] \mapsto (a_1b_1^{-1},a_2^{-1}c_2,c_1d_1^{-1},b_2d_2^{-1})$ defines the desired double Lie groupoid isomorphism $\mathfrak{M}_G(\widetilde{\Sigma},\widetilde{V})_{\widetilde{H},\widetilde{\mathcal{A}}}=(\mu^4)^{-1}(\mathcal{H})/\mathcal{K}\cong \mathcal{G} $. So we have that the double symplectic groupoid structure on $\mathcal{G} $ arises as a subquotient of the double pair groupoid $M^4$. \end{exa}
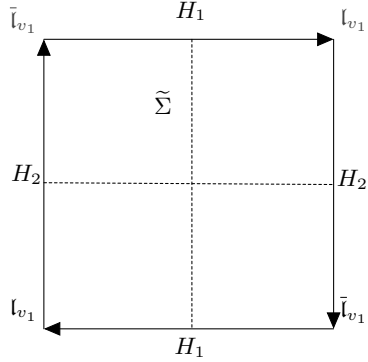
\begin{figure} \begin{center}
\definecolor{uuuuuu}{rgb}{0.26666666666666666,0.26666666666666666,0.26666666666666666}
\begin{tikzpicture}[line cap=round,line join=round,>=stealth,x=1cm,y=1cm,scale=0.4]
\clip(-4.3,-5) rectangle (11.7,4.5);
\draw [<-,line width=0.9pt] (0.44,2.64) -- (6.86,2.64);
\draw [<-,line width=0.9pt] (6.86,2.64) -- (6.86,-3.78);
\draw [<-,line width=0.9pt] (6.86,-3.78) -- (0.44,-3.78);
\draw [<-,line width=0.9pt] (0.44,-3.78) -- (0.44,2.64);
\draw [line width=0.9pt,dashed] (3.72,2.64)-- (3.72,-3.78);
\draw [line width=0.9pt,dashed] (0.44,-0.54)-- (6.86,-0.58);
\begin{scriptsize}
;
\draw[color=black] (-0.2,-3.35) node {${\mathfrak{a}}_{v_1}$};

\draw[color=black] (7.6,-3.35) node {$\mathfrak{a}^\vee_{v_1}$};

\draw[color=uuuuuu] (7.6,3.07) node {$\mathfrak{a}_{v_1}$};

\draw[color=uuuuuu] (-0.2,3.07) node {$\mathfrak{a}^\vee_{v_1}$};
\draw[color=black] (3.7,3.27) node {$H_1$};
\draw[color=black] (7.5,-0.47) node {$H_2$};
\draw[color=black] (3.68,-4.19) node {$H_1$};
\draw[color=black] (-0.2,-0.35) node {$H_2$};
\draw[color=black] (3.08,0.6) node {$\widetilde{\Sigma}$};

\end{scriptsize}
\end{tikzpicture} \caption{A decorated square which induces a double Poisson structure on $\mathfrak{M}_G(\widetilde{\Sigma},\widetilde{V})_{\widetilde{H},\widetilde{\mathcal{A}}} $} \label{fig:dousymsqu}
\end{center}   \end{figure} 
\subsubsection{Double quasi-Poisson groupoids} Now we shall see what is the geometric structure that naturally appears on top of the double Lie groupoid $M^4$ that we used in the proof of Theorem \ref{thm:doupoigro}. \begin{defi} A {\em double quasi-Poisson groupoid} is a double Lie groupoid $\mathcal{G} $ with sides $\mathcal{H}  $ and $\mathcal{K}  $ over $M$ which is equipped with a bivector field $\Pi\in \mathfrak{X}^2(\mathcal{G} )$ and $\Psi\in \Gamma (\wedge^3 A^{\mathcal{K}  })$ such that:
\begin{itemize} \item $\Pi$ is multiplicative over $\mathcal{H}  $ and over $\mathcal{K} $,
\item $[\Pi,\Pi]=\Psi^{\mathcal{R} }-\Psi^{\mathcal{L} }$ and $[\Pi,\Psi^{\mathcal{R} }]=0$, where $\quad^{\mathcal{R}}$ and $\quad^{\mathcal{L} }$ denote the right and left-invariant extensions with respect to $\mathcal{G} \rightrightarrows \mathcal{K} $;
\item $\Psi$ is multiplicative with respect to the VB-groupoid $\bigoplus^3_1 A^{\mathcal{K}  } \rightrightarrows \bigoplus^3_1 A_{\mathcal{H}  }$: i.e. the contraction map
\[ \overline{\Psi}: \left(\bigoplus^3_1 (A^{\mathcal{K}  })^* \rightrightarrows \bigoplus^3_1 C^*\right) \rightarrow \mathbb{R} \] is a groupoid morphism, where $C=\mathtt{u}_{\mathcal{K}  }^* \ker \mathtt{s}_{A^{\mathcal{K} }} $ is the core of the LA-groupoid $ A^{\mathcal{K} } \rightrightarrows A_{\mathcal{H} }$. \end{itemize}  \end{defi}
\begin{rema} Notice that the fact that $\Psi$ is multiplicative implies immediately that $\Psi^{\mathcal{R} }-\Psi^{\mathcal{L} }$ is also multiplicative with respect to $\mathcal{G} \rightrightarrows \mathcal{H} $. \end{rema} 
The double pair groupoid $M^4$ that appears in the proof of Theorem \ref{thm:doupoigro} is naturally endowed with the bivector field 
\[ (\pi_{\Sigma,V},-\pi_{\Sigma,V},-\pi_{\Sigma,V},\pi_{\Sigma,V}) \] 
which is multiplicative with respect to both groupoid structures, see \S \ref{subsec:quapoi}. Besides, the multivector field $([\pi_{\Sigma,V},\pi_{\Sigma,V}],[\pi_{\Sigma,V},\pi_{\Sigma,V}])$ is multiplicative with respect to any of the side pair groupoids $M \times M \rightrightarrows M$ so $M^4$ is a double quasi-Poisson groupoid in two different ways.   

\subsection{Double Poissson groupoids obtained by categorifying the structure group}\label{subsec:catgro}
\subsubsection{Lie 2-groups, Lie 2-algebras and crossed modules} Now we shall see an alternative description for Lie 2-groups which allows us to construct a number of examples.
\begin{defi} A {\em Lie crossed module} consists of a morphism of Lie groups $\Phi:H\rightarrow G$ and an action of $G$ on $H$ by automorphisms such that for all $g\in G$, and $h,k\in H$:
\begin{align} 
\Phi(g\cdot h)&=g\Phi(h)g^{-1} \label{cm1} \\
\Phi(h)\cdot k&=hkh^{-1}. \label{cm2}
\end{align} \end{defi}
A Lie crossed module $\Phi:H\rightarrow G$ induces an action groupoid $H \times G \rightrightarrows G$ with action map $(a,x)\mapsto \Phi(a)x$. The semidirect product structure on $H \rtimes G$ makes this action groupoid into a Lie group and all its structure maps as a groupoid are group morphisms with respect to this group structure. All the Lie 2-groups are of this form, see \cite{brsp}. For instance, the pair groupoid over a Lie group $G$ is a Lie 2-group whose corresponding Lie crossed module is the identity on $G$. Another fundamental example is the {\em automorphism Lie 2-group} of a Lie group $G$ which corresponds to the Lie crossed module $\Phi:G \rightarrow \text{Aut}(G)$ given by $\Phi(a)(b)=aba^{-1}$ for all $a,b\in G$ and the tautological action of $\text{Aut}(G)$ on $G$.
\begin{defi}[\cite{bacr}] A groupoid in the category of Lie algebras is called a {\em (strict) Lie 2-algebra}. An abelian Lie 2-algebra is called a {\em 2-vector space}. \end{defi}

It is straightforward to see that the Lie functor establishes an equivalence of categories between the categories of 1-connected Lie 2-groups and Lie 2-algebras.

Lie 2-algebras are equivalent to crossed modules of Lie algebras.
\begin{defi}[\cite{bacr}]\label{def:difcromod} Let $\phi:\mathfrak{h}\rightarrow\mathfrak{g}$ be a Lie algebra morphism and let $\mathfrak{g}$ act on $\mathfrak{h}$ by derivations in such a way that for all $x\in \mathfrak{g}$ and $a,b\in\mathfrak{h}$ we have that
\begin{align*} 
\phi(x\cdot a)&=[x,\phi(a)] \\
\phi(a)\cdot b&=[a,b]; 
\end{align*}
then this structure is called a {\em differential crossed module or a crossed module of Lie algebras}. The Lie 2-algebra associated to $\phi$ is the action groupoid $\mathfrak{G} :\mathfrak{h}\rtimes \mathfrak{g} \rightrightarrows \mathfrak{g}$ with action map $\mathtt{t} (a,x)=\phi(a)+x$, for $a\in \mathfrak{h}$, $x\in \mathfrak{g}$. Here $\mathfrak{h}\rtimes \mathfrak{g}$ is the semidirect product Lie algebra given by the $\mathfrak{g}$-action. \end{defi}
\subsubsection{The categorified $\Gamma $-twisted Cartan-Dirac structure}\label{subsec:multwidir} A Lie algebra $\mathfrak{g}$ equipped with $s\in S^2(\mathfrak{g})$ which is $\mathfrak{g}$-invariant but possibly degenerate as a bilinear form on $\mathfrak{g}^*$ gives rise to a differential crossed module $s^\sharp:\mathfrak{g}^* \rightarrow \mathfrak{g}$ where the action of $\mathfrak{g}$ on $\mathfrak{g}^*$ is the infinitesimal coadjoint action: $(X,\alpha )\in \mathfrak{g} \times \mathfrak{g}^*$ $(X,\alpha)\mapsto -\ad^*_X \alpha $. The associated Lie 2-algebra $\mathfrak{G}:\mathfrak{g}^* \rtimes \mathfrak{g} \rightrightarrows \mathfrak{g} $ carries an Ad-invariant nondegenerate symmetric bilinear form (a metric for short) compatible with the groupoid structure:
\begin{align}\label{met}  (\alpha +X, \beta+ Y)\mapsto \mathfrak{s} (\alpha +X, \beta+ Y):=\langle  \alpha ,Y \rangle  + \langle  \beta, X \rangle +s(\alpha ,\beta ), \end{align}   
for all $\alpha ,\beta \in \mathfrak{g}^*$ and $X,Y\in \mathfrak{d}$. The compatibility means that the graph of the composition of $\mathfrak{G} $ (which is a Lie subalgebra) is Lagrangian in the Courant algebroid over a point $\mathfrak{G} \times \mathfrak{G}\times \overline{\mathfrak{G} }$. The Lie 2-algebra $\mathfrak{G} $ equipped with the metric $\mathfrak{s}  $ shall be called a {\em split quadratic Lie 2-algebra}. To summarize: isomorphism classes of split quadratic Lie 2-algebras are classified by Lie algebras $\mathfrak{g}$ endowed with $s\in S^2(\mathfrak{g})$ which is $\mathfrak{g}$-invariant. Split quadratic Lie 2-algebras are the same objects as CA-groupoids over a point (regarding the latter as the trivial groupoid). It was in this sense that this classification was obtained in \cite[Rem. 3.6]{liedir} (reformulating a remark in \cite[\S 3]{driquahop}). If $\mathcal{G} \rightrightarrows G$ is a Lie 2-group which integrates $\mathfrak{G} $, we can consider the $\Gamma $-twisted Cartan-Dirac structure $(\mathcal{E}_\Gamma ,\mathcal{L}_\Gamma )$ on $\mathcal{G}^E$ associated to a boundary graph $(E,V)$. Since $(\mathcal{E}_\Gamma ,\mathcal{L}_\Gamma )$ is defined in terms of maps which are all groupoid morphisms, we get the following result.
\begin{prop}\label{pro:catcardir} The $\Gamma $-twisted Cartan-Dirac structure $(\mathcal{E}_\Gamma ,\mathcal{L}_\Gamma )$ associated to the boundary graph $\Gamma =(E,V)$ of a marked surface is a multiplicative Manin pair over the Lie groupoid $\mathcal{G}^E \rightrightarrows G^E$. \end{prop}
\begin{proof} This is straightforward except, possibly, for the fact that the graph of the multiplication map inside $ \mathcal{E}_\Gamma  \times \mathcal{E}_\Gamma   \times \overline{\mathcal{E}  }_\Gamma $ is a Dirac structure with support the graph of the multiplication in $\mathcal{G}^E \times \mathcal{G}^E \times \mathcal{G}^E$. But this follows from the fact that the Courant bracket on $\mathcal{E}_V$ restricted to constant sections is obtained by differentiating the following action of $\mathcal{G}^{2V}$ on $\mathcal{G}^E$:
\[ (g_v,h_v)_{v\in V}\cdot (a_e)_{e\in E}= (h_{\mathtt{T} (e)}a_eg_{\mathtt{S}(e)}^{-1})_{e\in E}, \] 
for all $g_v,h_v,a_e\in \mathcal{G} $, and this action is a Lie groupoid morphism. \end{proof} 
\begin{rema} Notice that the Lagrangian complement $\mathcal{L}^\star_\Gamma  \hookrightarrow \mathcal{E}_\Gamma $ of $\mathcal{L}_\Gamma $ defined in Remark \ref{rem:cancomqpoi} is also multiplicative, i.e. it is a VB-subgroupoid of $(\mathcal{E}_\Gamma ,\mathcal{L}_\Gamma  )$. \end{rema}
Let $(\Sigma,V)$ be a marked compact oriented surface with nonempty boundary and consider a Lie 2-group $\mathcal{G} \rightrightarrows G $ endowed with a split quadratic Lie 2-algebra $\mathfrak{G} $. Let $\mathfrak{D}=\mathfrak{G} \oplus\overline{ \mathfrak{G}} $ be the double of $\mathfrak{G} $. We have that the structure maps of $\mathcal{G} $ endow $M:=\hom(\Pi_1(\Sigma,V),\mathcal{G} ) $ with a Lie groupoid structure over $M_0:=\hom(\Pi_1(\Sigma,V),G)$ 
\begin{align} (M \rightrightarrows M_0)=\left( \hom(\Pi_1(\Sigma,V),\mathcal{G}  ) \rightrightarrows \hom(\Pi_1(\Sigma,V), G) \right) \label{eq:modspa2gr}\end{align}  
and the moment map $\mu: \hom(\Pi_1(\Sigma,V),\mathcal{G}) \rightarrow \mathcal{G}^E$ is a Lie groupoid morphism with respect to this groupoid structure.
\begin{exa} Suppose that $V$ contains exactly one point for each component of $\partial\Sigma$ and suppose that our Lie 2-group is $ G \rtimes \text{Aut}(G)\rightrightarrows \text{Aut}(G) $, the automorphism 2-group of a connected semi-simple Lie group $G$. We can use the groupoid structure on $M$ discussed above to view the spaces of twisted representations \cite{twimodspa} as the source-fibers of $M$ and their twisted quasi-Hamiltonian structures as restrictions of the canonical quasi-Hamiltonian space structure on $M$ provided by \cite{alemeimal}. 

The Lie 2-algebra of $G \rtimes \text{Aut}(G)$ is determined by the differential crossed module corresponding to the map $\mathfrak{g}  \rightarrow \{D:\mathfrak{g} \rightarrow \mathfrak{g} |\text{$D$ is a derivation of $\mathfrak{g} $} \}$ defined by $u\mapsto \ad_u$. But the inner automorphisms of $G$ induce a Lie subgroup of finite index in $\text{Aut}(G)$ \cite{autssG} so every derivation of $\mathfrak{g} $ is inner and the map above can be identified with the identity $\text{id}_{\mathfrak{g} }: \mathfrak{g} \rightarrow \mathfrak{g} $. As a consequence, this Lie 2-algebra is isomorphic to the pair groupoid Lie 2-algebra $\mathfrak{G}:=\left(\mathfrak{g} \times \mathfrak{g} \rightrightarrows \mathfrak{g}\right) $. In fact, the map determined by $(u,v)\mapsto (u+v,v)$ is an isomorphism between the Lie 2-algebra corresponding to the differential crossed module $\text{id}_{\mathfrak{g} }$ as in Definition \ref{def:difcromod} and the pair groupoid Lie 2-algebra over $\mathfrak{g} $ with the direct product Lie algebra structure on $\mathfrak{g} \times \mathfrak{g} $. But then we can equip $\mathfrak{G} $ with the pairing $\mathfrak{s} = s\ominus s$, where $s$ is the Killing form on $\mathfrak{g} $. The pairing $\mathfrak{s} $ automatically makes the pair groupoid $\mathfrak{g} \times \overline{\mathfrak{g} } \rightrightarrows \mathfrak{g} $ into a split quadratic Lie 2-algebra (or, equivalently, as a CA-groupoid over a point). 

In this situation, the $\mathtt{s}$-fiber of the Lie groupoid \eqref{eq:modspa2gr} over $\sigma\in \hom(\Pi_1(\Sigma,V),\text{Aut} (G)) $ is exactly the $\sigma$-twisted moduli space associated to $(\Sigma,V)$ \cite{twiwilcha,twimodspa} which is often denoted by $\hom_\sigma(\Pi_1(\Sigma,V),G)$. 

Since $V$ contains only one point for each component of $\partial \Sigma$, there is a copy of the Cartan-Dirac structure $(\mathcal{E} ,\mathcal{L})$ (see Example \ref{ex:quaham}) over $\mathcal{G}:=G \rtimes \text{Aut}(G)  $ for each of these components. So $(\mathcal{E}_{\Gamma},\mathcal{L}_{\Gamma })$ is isomorphic to the product Courant algebroid $(\mathcal{E} ,\mathcal{L})^E$. Consider the $\mathfrak{G} \oplus \mathfrak{G} $-action on $\mathcal{G} $ given by \eqref{eq:infact} but restricted to the $\mathtt{s}$-fiber $\mathtt{s}^{-1}(0\oplus 0) \hookrightarrow \left( \mathfrak{G} \oplus \mathfrak{G} \rightrightarrows \mathfrak{g} \oplus \mathfrak{g}\right)  $. What we obtain then is the $\tau$-twisted Cartan-Dirac structure over $G\cong \mathtt{t}^{-1}(\tau)$ for every $\tau\in \text{Aut}(G)$ \cite{twimodspa}:
\[ \xymatrix@-0.7pc{(\mathcal{E}_\tau,\mathcal{L}_\tau):=( \mathfrak{g} \oplus \mathfrak{g} ,\mathfrak{g}_\Delta)\times G \ar@{^{(}->}[r]\ar[d] & (\mathcal{E} ,\mathcal{L}) \ar[d] \\
G \times \{\tau\} \ar@{^{(}->}[r] & \mathcal{G}; } \]
the $\tau$-twisted Cartan-Dirac structure is given by the diagonal inclusion of $\mathcal{L}_\tau=\mathfrak{g} \times G$ inside the action Courant algebroid $\mathcal{E}_\tau= (\mathfrak{g}\oplus \mathfrak{g} )\times G$ associated to the metric $\mathfrak{s} $ on $\mathfrak{d}=\mathfrak{g} \oplus \mathfrak{g} $ and the action map $u\oplus v\mapsto v^r-\tau(u)^l$ for all $u,v\in \mathfrak{g} $ (it can be checked that $\mathcal{E}_\tau$ is isomorphic to the usual untwisted action Courant algebroid over $G$ defined as in \S \ref{subsec:twicar}). Since the moment map 
\[ \mu: \hom(\Pi_1(\Sigma,V),\mathcal{G}) \rightarrow \mathcal{G}^E \]
is a Lie groupoid morphism, it takes source fibers to source fibers. We can check that the Manin pair morphism $R_{\Sigma,V}$ over $\mu$ (which is determined by a 2-form as in \S \ref{subsec:mpmor}) 
endows $\hom_\sigma(\Pi_1(\Sigma,V),G)$ with a {\em tq-Hamiltonian $G^E$-space structure}. The twisting automorphism is given by taking the value of $\sigma\in \hom(\Pi_1(\Sigma,V),\text{Aut} (G))$ along each component of $\partial \Sigma$, see \cite[Thm. 4.2.4]{twimodspa}:
\[ \xymatrix@-0.7pc{(\mathbb{T}\hom(\Pi_1(\Sigma,V),\mathcal{G}),T\hom(\Pi_1(\Sigma,V),\mathcal{G} ) \ar[r]^-{R_{\Sigma,V}} & (\mathcal{E},\mathcal{L}  )^E  \\ (\mathbb{T}\hom_\sigma(\Pi_1(\Sigma,V),G),T \hom_\sigma(\Pi_1(\Sigma,V),G)) \ar[u]^{R_j} \ar[r]^-{R_{\Sigma,V}|} & (\mathcal{E}_\tau,\mathcal{L}_\tau)^E; \ar@{^{(}->}[u] } \]
where $R_j$ is the canonical morphism \eqref{eq:canmpm} associated to the inclusion 
\[ j: \hom_\sigma(\Pi_1(\Sigma,V),G) \hookrightarrow \hom(\Pi_1(\Sigma,V),\mathcal{G}  )\] 
as the $\mathtt{s}$-fiber over $\sigma$. As a concrete example, take $\Sigma$ as the annulus and two generators for $\Pi_1(\Sigma,V)$ as in \cite[Fig. 4.2.1]{twimodspa}. Then $\hom(\Pi_1(\Sigma,V),\mathcal{G})$ can be identified with $\mathcal{G}^2$. The quasi-Hamiltonian 2-form $\omega$ on $\hom(\Pi_1(\Sigma,V),\mathcal{G}  )$ induced by the canonical identification of $\mathcal{E}_{\Gamma }$ with $\mathbb{T}_\eta \mathcal{G} \times  \mathbb{T}_\eta \mathcal{G}$ \cite{purspi} is given by the formula 
\[ \Omega_{(x,y)}=-\frac{1}{2}\left(\mathfrak{s}  (x^*\theta^l_{\mathcal{G} },y^*\theta^r_{\mathcal{G} })+ \mathfrak{s} (x^*\theta^r_{\mathcal{G} },y^*\theta^l_{\mathcal{G} }) \right), \] 
for all $(x,y)\in \mathcal{G}^2$, where $\theta^l_{\mathcal{G} }$ and $\theta^r_{\mathcal{G} }$ are, respectively, the left and right-invariant Maurer-Cartan forms on $\mathcal{G} $. Take $\sigma=(\tau,\kappa)\in \hom(\Pi_1(\Sigma,V),\text{Aut} (G))$, then we have that \[ (j^* \Omega)_{(a,b)}=-\frac{1}{2}\left({s} (\tau^{-1}(a^*\theta^l_{G }),b^*\theta^r_{{G} })+ {s}(a^*\theta^r_{{G} },\kappa^{-1}(b^*\theta^l_{{G} })) \right), \]
for all $(a,b)\in G^2\cong \hom_\sigma(\Pi_1(\Sigma,V),G)$. This follows from the fact that $i_\tau^*\theta^l_{\mathcal{G} }=\tau^{-1}(\theta^l_G)$ and $i_\tau^*\theta^r_{\mathcal{G} }=\theta^r_G$, where $i:G \times \{\tau\} \hookrightarrow \mathcal{G} $ is the inclusion, see \cite[Ex. 4.2.5]{twimodspa}. \end{exa}
\subsubsection{Multiplicative boundary decorations} From now on we will assume that we work with a connected Lie 2-group $\mathcal{G} $ over a connected base Lie group $G$ with split quadratic Lie 2-algebra and $(\Sigma,V)$ is a fully marked surface. Let us denote $M:=\hom(\Pi_1(\Sigma,V),\mathcal{G} ) $ as in \eqref{eq:modspa2gr}.
\begin{prop}\label{pro:mulmp2} The canonical exact Manin pair morphism 
\[ R_{\Sigma,V}: (\mathbb{T}M,T M) \rightarrow (\mathcal{E}_{\Gamma } ,\mathcal{L}_{\Gamma }) \]
is a morphism of multiplicative Manin pairs over $\mu$. \end{prop}
\begin{proof} By the construction of $R_{\Sigma,V}$, it is enough to verify the statement for polygons and then we can check that, if it holds for a fully marked surface $(\Sigma',V')$, it continues to hold for the new marked surface which is obtained by gluing two edges in the boundary graph of $(\Sigma',V')$ as in \S \ref{subsec:qhammodspa}. First of all, if $\Sigma$ is a polygon, then $i:M\cong\{(g_e)\in \mathcal{G}^E|\prod_{e\in E} g_e=1\} \hookrightarrow \mathcal{G}^E$ is a Lie subgroupoid. Then the Courant morphism $R_{\Sigma,V}$ is given in this case by \[ \mathtt{a}(X)+i^* \alpha \sim_{R_{\Sigma,V}} X+ \mathtt{a}^* \alpha \] 
for all $X\in \mathcal{L}_{\Gamma }|_\rho$ and all $\alpha \in T^*_{i(\rho)} G^E$, where $\rho\in M$ \cite[Ex. 2.9]{quisur2}. Then $R_{\Sigma,V}$ is a multiplicative Dirac structure with support on $\text{Graph}(i)$ because both the anchor $\mathtt{a}:\mathcal{E}_{\Gamma }\rightarrow T \mathcal{G}^E $ and $Ti$ are VB-groupoid morphisms. 

Now suppose that $(\Sigma',V')$ is a fully marked compact oriented surface for which the canonical morphism $R_{\Sigma',V'}$ is a multiplicative Manin pair morphism over the corresponding moment map $\mu':M':=\hom(\Pi_1(\Sigma',V'),\mathcal{G} ) \rightarrow \mathcal{G}^{E'}$. Let $\Gamma'=(E',V')$ be the boundary graph of $(\Sigma',V')$. For the sake of keeping the notation simple, let us suppose that $\Sigma$ is obtained from $\Sigma'$ by sewing a pair of boundary edges $e,e'\in E'$ as in \S \ref{subsec:sew}. Now we just have to observe that, in this situation, the reductive data \eqref{eq:sew1} and \eqref{eq:sew2} are automatically multiplicative with respect to the multiplicative Manin pair structure on $(\mathcal{E}_{\Gamma'} ,\mathcal{L}_{\Gamma'}  ) $ in the sense of Proposition \ref{pro:cared}. This follows from the fact that $\mathfrak{c}_{\text{sew}_{\{e,e'\}} }\hookrightarrow  \mathfrak{D}^{V'}  $ is a Lie 2-subalgebra and $N \hookrightarrow  \mathcal{G}^{E'}$ is a Lie subgroupoid. Since $R_{\Sigma',V'}  $ is supposed to be multiplicative, Proposition \ref{pro:cared} implies that $R_{\Sigma,V}=\left( R_{\Sigma',V'}\right)_{\mathfrak{c}_{\text{sew}_{\{e,e'\}} },N}$ is multiplicative as well. \end{proof}   Let 
\[ (H,\mathcal{A})= \left(\prod_{e\in E} H_e,\prod_{v\in V} \mathfrak{a}_v \times \mathcal{G}^E\right) \] 
be a decoration of the boundary graph $\Gamma =(E,V)$ of $\Sigma$ with respect to the structure group $\mathcal{G} $. We say that the boundary data are {\em multiplicative} if each of the $H_e \hookrightarrow \mathcal{G} $ is a Lie subgroupoid over $H_e'\hookrightarrow G$ and each of the $\mathfrak{a}_v \hookrightarrow \mathfrak{G} $ is a Lie 2-subalgebra over $\mathfrak{a}_v'\hookrightarrow \mathfrak{g} $. In this way we also get a decoration $(H', \mathcal{A}')=\left(\prod_{e\in E} H_e' , \prod_{v\in V} \mathfrak{a}_v' \times {G}^E\right)$ of the boundary graph $\Gamma =(E,V)$ of $\Sigma$ with respect to ${G} $. As a simple consequence of Proposition \ref{pro:mulmp}, we have the following result. 
\begin{thm}\label{thm:twicardir} Suppose that the moduli space $\mathfrak{M}_\mathcal{G}(\Sigma,V)_{H,\mathcal{A}}$ is smooth and the boundary data are multiplicative, where $\mathcal{G} \rightrightarrows G$ is a connected Lie 2-group endowed with a split quadratic Lie 2-algebra. Then there is a Lie groupoid structure on $\mathfrak{M}_\mathcal{G}(\Sigma,V)_{H,\mathcal{A}} \rightrightarrows \mathfrak{M}_{G}(\Sigma,V)_{H',\mathcal{A}'}$ which makes it into a Poisson groupoid with its canonical Poisson structure. \end{thm} 
\begin{proof} Since the boundary data are multiplicative, $\mu^{-1}(H)$ is a Lie subgroupoid of \eqref{eq:modspa2gr} and the action of the Lie 2-group $K$ on this Lie groupoid is multiplicative and hence it endows the quotient with a Lie groupoid structure. Proposition \ref{pro:mulmp2} implies that $R_{\Sigma,V}$ is a morphism of multiplicative Manin pairs so the existence of a compatible Poisson groupoid structure follows from Proposition \ref{pro:mulmp}. \end{proof}
\begin{rema} Since the canonical Lagrangian complement $\mathcal{L}^\star_\Gamma  \hookrightarrow \mathcal{E}_\Gamma  $ of $\mathcal{L}_\Gamma $ defined in Remark \ref{rem:cancomqpoi} is multiplicative, the induced bivector field $\pi_{\Sigma,V}$ on $M= \hom(\Pi_1(\Sigma,V),\mathcal{G} )$ described in \eqref{eq:qpoimodspa} is also multiplicative. \end{rema}

\subsubsection{The compatibility between gluing and categorification} In order to get a double Poisson groupoid by applying Theorem \ref{thm:poigro} and categorifying the structure group, we have to also consider boundary data which are multiplicative. So let $(\Sigma,V)$ be a fully marked compact and oriented surface as in the statement of Theorem \ref{thm:poigro} with boundary graph $\Gamma =(E,V)$ and let $\mathcal{G} \rightrightarrows G$ be a connected Lie 2-group equipped with a split quadratic Lie 2-algebra. Now let $(\widehat{H},\widehat{\mathcal{A} })$ be a symmetric decoration of the doubled surface $\widehat{\Sigma}=\Sigma\cup_S \Sigma$ for some subset $S\subset E$ as in \S \ref{subsec:symdec}. Assume that $(\widehat{H},\widehat{\mathcal{A}})$ are multiplicative and let $(\widehat{H}',\widehat{\mathcal{A} }')$ be the induced boundary data for $\widehat{\Sigma}$ with respect to $G$ as in Theorem \ref{thm:twicardir}.  
\begin{thm}\label{thm:doupoigro2} Suppose that the conditions of Theorem \ref{thm:poigro} are satisfied. Then the moduli space $\mathfrak{M}_{\mathcal{G} }(\widehat{\Sigma},\widehat{V})_{\widehat{H},\widehat{\mathcal{A}}} $ becomes a double Poisson groupoid with sides $\mathfrak{M}_{{G} }(\widehat{\Sigma},\widehat{V})_{\widehat{H}',\widehat{\mathcal{A} }'} $ and $\mathfrak{M}_{\mathcal{G} }({\Sigma},{V_0})_{\check{H},\check{\mathcal{A}}} $ over the base manifold $\mathfrak{M}_G(\Sigma,V_0)_{\check{H}',\check{\mathcal{A} }'} $ with respect to its canonical Poisson structure. \end{thm}
\begin{rema} Let us observe that $H_e \hookrightarrow \mathcal{G} $ is a Lie 2-subgroup whenever $e$ is adjacent to an edge in $S$, since it is both a subgroup and a subgroupoid. According to Proposition \ref{rem:baspoigro}, the boundary data $(\widehat{H},\widehat{\mathcal{A} })$ induce boundary data $(\check{H},\check{\mathcal{A} })$ for the marked surface $(\Sigma,V_0)$ with respect to $\mathcal{G} $, while $(\widehat{H}',\widehat{\mathcal{A}}')$ induce boundary data $(\check{H}',\check{\mathcal{A} }')$ for $(\Sigma,V)$ with respect to $G$.   \end{rema}
\begin{proof}[Proof of Theorem \ref{thm:doupoigro2}] Notice that Theorem \ref{thm:poigro} gives us a Poisson groupoid structure on 
\[ \mathfrak{M}_{\mathcal{G} }(\widehat{\Sigma},\widehat{V})_{\widehat{H},\widehat{\mathcal{A} }} \rightrightarrows \mathfrak{M}_{\mathcal{G} }({\Sigma},{V}_0)_{\check{H},\check{\mathcal{A} }} , \]
while the multiplicativity of the boundary data guarantees, thanks to Theorem \ref{thm:twicardir}, that we have a Poisson groupoid structure on \[ \mathfrak{M}_{\mathcal{G} }(\widehat{\Sigma},\widehat{V})_{\widehat{H},\widehat{\mathcal{A} }} \rightrightarrows \mathfrak{M}_{{G} }(\widehat{\Sigma},\widehat{V})_{\widehat{H}',\widehat{\mathcal{A} }'} ,\] 
we have to check the compatibility of the corresponding structure maps. Let us follow the notation introduced in the proof of Theorem \ref{thm:poigro}. We have that the moment map $(\mu,\mu):M \times M \rightarrow \mathcal{G}^E \times \mathcal{G}^E$ is a morphism of double Lie groupoids, where $M=\hom(\Pi_1(\Sigma,V),\mathcal{G} )$ and we consider the pair double Lie groupoid $M \times M$ with sides $M_0 \times M_0$ and $M$ over $M_0=\hom(\Pi_1(\Sigma,V),{G} )$ as in \eqref{eq:modspa2gr}. The Lie groupoids $\mathcal{H}_e$ defined in \eqref{eq:holgro} are all double Lie groupoids:
\[ \xymatrix@-0.7pc{ H_e \times H_e \ar@<-.5ex>[r] \ar@<.5ex>[r]\ar@<-.5ex>[d] \ar@<.5ex>[d]& H_e' \times H_e'\ar@<-.5ex>[d] \ar@<.5ex>[d] \\ H_e \ar@<-.5ex>[r] \ar@<.5ex>[r] & H_e', } \quad\quad \xymatrix@-0.7pc{  \mathcal{G}_\Delta \ar@<-.5ex>[r] \ar@<.5ex>[r]\ar@<-.5ex>[d] \ar@<.5ex>[d]& G_{\Delta}\ar@<-.5ex>[d] \ar@<.5ex>[d] \\ \mathcal{G} \ar@<-.5ex>[r] \ar@<.5ex>[r] & G, } \quad\quad \xymatrix@-0.7pc{ H_e \times \mathcal{G}  \ar@<-.5ex>[r] \ar@<.5ex>[r]\ar@<-.5ex>[d] \ar@<.5ex>[d]&H_e' \times G \ar@<-.5ex>[d] \ar@<.5ex>[d] \\ \mathcal{G}  \ar@<-.5ex>[r] \ar@<.5ex>[r] & G, }
\quad\quad \xymatrix@-0.7pc{ \mathcal{G} \times H_e \ar@<-.5ex>[r] \ar@<.5ex>[r]\ar@<-.5ex>[d] \ar@<.5ex>[d]& G \times H_e'\ar@<-.5ex>[d] \ar@<.5ex>[d] \\ \mathcal{G}  \ar@<-.5ex>[r] \ar@<.5ex>[r] & G; }\]
where the action groupoids above correspond to the action by left and right translations of the Lie 2-subgroup $H_e \rightrightarrows H_e'$ on $\mathcal{G} \rightrightarrows G$. As a consequence, the product $\mathcal{H} $ of the $\mathcal{H}_e $ constitutes a double Lie subgroupoid of the pair double Lie groupoid 
\[  \xymatrix @-0.7pc { \mathcal{G}^E \times \mathcal{G}^E \ar@<-.5ex>[r] \ar@<.5ex>[r]\ar@<-.5ex>[d] \ar@<.5ex>[d]& G^E \times G^E \ar@<-.5ex>[d] \ar@<.5ex>[d] \\ \mathcal{G}^E \ar@<-.5ex>[r] \ar@<.5ex>[r] & G^E }\] 
and hence $(\mu,\mu)^{-1}(\mathcal{H} )$ is also a double Lie groupoid. Analogously, $\mathcal{K}$ is a double Lie subgroupoid of the double Lie groupoid $\mathcal{G}^{2V}$ with sides $\mathcal{G}^V$ and $G^{2E}$ over $G^V$, and all its structure maps are Lie group morphisms. Since the action \eqref{eq:gauact} is a double Lie groupoid morphism $(\mathcal{G}^V \times \mathcal{G}^V ) \times( M \times M) \rightarrow M \times M$, the associated quotient is a double Lie groupoid. \end{proof}
\begin{rema} Notice that the bivector field $(\pi_{\Sigma,V}, -\pi_{\Sigma,V})$ on the double pair groupoid $M \times M$ which appears in the previous argument is multiplicative with respect to both groupoid structures. Moreover, $[\pi_{\Sigma,V},\pi_{\Sigma,V}]$ is also multiplicative with respect to the groupoid structure $M \rightrightarrows M_0$. Therefore, $M \times M$ naturally becomes a double quasi-Poisson groupoid as well. \end{rema} 

\begin{rema} In the previous theorem, if $S=E$ and the boundary data are trivial, then we would get a double symplectic groupoid structure on the full moduli space of flat $\mathcal{G} $-bundles over $\widehat{\Sigma}$, denoted by $\mathfrak{M}_\mathcal{G}(\widehat{\Sigma})$, with sides $ \mathfrak{M}_{\mathcal{G}}({\Sigma})$ and $\mathfrak{M}_G(\widehat{\Sigma})$ over $\mathfrak{M}_G(\Sigma)$, provided that these spaces were smooth. In particular, the symplectic groupoid structure $\mathfrak{M}_\mathcal{G}(\widehat{\Sigma}) \rightrightarrows  \mathfrak{M}_{G}(\widehat{\Sigma})$ should be an integration of the Poisson structure considered in \cite{huepoimodspa} but, since the associated moduli spaces are singular in general, they are out of the scope of this work and shall be studied elsewhere. \end{rema} 
\begin{exa} Consider the situation of Example \ref{exa:poigro0} but suppose that the structure group is a Lie 2-group $\mathcal{G} $ endowed with a split quadratic Lie 2-algebra $\mathfrak{G} $ which splits as the direct sum of two Lagrangian Lie 2-subalgebras $\mathfrak{G} =\mathfrak{A} \oplus \mathfrak{B} $; then $(\mathfrak{G} ,\mathfrak{A} ,\mathfrak{B} )$ can be called a {\em Manin triple of Lie 2-algebras}. The associated moduli space is then the connected Lie 2-subgroup $\mathcal{B} \hookrightarrow \mathcal{G}  $ which integrates $\mathfrak{B} $ endowed with its Poisson 2-group structure \cite{poi2gro}. \end{exa} 
If the decoration of the boundary is symmetric in two different ways as in Theorem \ref{thm:doupoigro} and if the conditions of Theorem \ref{thm:twicardir} are also satisfied, then the associated moduli spaces inherit three different Lie groupoid structures such that all the structure maps with respect to any of them is a double Lie groupoid morphism with respect to the other two. In other words, what we obtain is a groupoid object in the category of double Lie groupoids, an object which is called a {\em triple Lie groupoid} and is represented by a cube:
\[ \xymatrix @-1.7pc { A \ar@<-.5ex>[dd]\ar@<.5ex>[dd] \ar@<-.5ex>[rd] \ar@<.5ex>[rd] \ar@<-.5ex>[rr] \ar@<.5ex>[rr] && B \ar@<.5ex>[dd]|\hole \ar@<-.5ex>[dd]|\hole \ar@<.5ex>[rd] \ar@<-.5ex>[rd] \\ 
& C \ar@<.5ex>[dd] \ar@<-.5ex>[dd] \ar@<.5ex>[rr] \ar@<-.5ex>[rr] && D \ar@<.5ex>[dd] \ar@<-.5ex>[dd] \\
A' \ar@<.5ex>[rr]|\hole \ar@<-.5ex>[rr]|\hole \ar@<.5ex>[rd] \ar@<-.5ex>[rd] && B' \ar@<.5ex>[rd] \ar@<-.5ex>[rd] \\ & C' \ar@<-.5ex>[rr]\ar@<.5ex>[rr] && D'; } \]
where each of the faces is a double Lie groupoid. We say that a triple Lie groupoid is a {\em triple Poisson groupoid} if it is equipped with a Poisson structure which is multiplicative with respect to its three Lie groupoid structures. 

The following result follows from combining the arguments of Theorems \ref{thm:doupoigro} and \ref{thm:doupoigro2}.
\begin{thm}\label{thm:tripoigro} Suppose that the conditions of Theorem \ref{thm:doupoigro} are satisfied but assume that the structure group is a connected Lie 2-group $\mathcal{G} \rightrightarrows G $ endowed with a split quadratic Lie 2-algebra and the boundary data are multiplicative. Then the moduli space $\mathfrak{M}_{\mathcal{G} }(\widetilde{\Sigma},\widetilde{V})_{\widetilde{H},\widetilde{\mathcal{A}}} $ becomes a triple Poisson groupoid with respect to its canonical Poisson structure as in the following diagram:
\[ \xymatrix @-1.8pc { \mathfrak{M}_{\mathcal{G} }(\widetilde{\Sigma},\widetilde{V})_{\widetilde{H},\widetilde{\mathcal{A}}} \ar@<-.5ex>[dd]\ar@<.5ex>[dd] \ar@<-.5ex>[rd] \ar@<.5ex>[rd] \ar@<-.5ex>[rr] \ar@<.5ex>[rr] && \mathfrak{M}_{\mathcal{G} }(\widehat{\Sigma}_S,\widehat{V}^0_S)_{\check{H}_S,\check{\mathcal{A}}_S} \ar@<.5ex>[dd]|\hole \ar@<-.5ex>[dd]|\hole \ar@<.5ex>[rd] \ar@<-.5ex>[rd] \\ 
& \mathfrak{M}_{\mathcal{G} }(\widehat{\Sigma}_T,\widehat{V}_T^0)_{\check{H}_T,\check{\mathcal{A}}_T} \ar@<.5ex>[dd] \ar@<-.5ex>[dd] \ar@<.5ex>[rr] \ar@<-.5ex>[rr] && \mathfrak{M}_{\mathcal{G} }({\Sigma},{V}_{0,0})_{\check{H},\check{\mathcal{A}}} \ar@<.5ex>[dd] \ar@<-.5ex>[dd] \\
\mathfrak{M}_{{G} }(\widetilde{\Sigma},\widetilde{V})_{\widetilde{H}',\widetilde{\mathcal{A}}'} \ar@<.5ex>[rr]|!{[r];[r]}\hole \ar@<-.5ex>[rr]|!{[r];[r]}\hole \ar@<.5ex>[rd] \ar@<-.5ex>[rd] && \mathfrak{M}_{{G} }(\widehat{\Sigma}_S,\widehat{V}_S^0)_{\check{H}_S',\check{\mathcal{A}}_S'} \ar@<.5ex>[rd] \ar@<-.5ex>[rd] \\ & \mathfrak{M}_{{G} }(\widehat{\Sigma}_T,\widehat{V}_T^0)_{\check{H}_T',\check{\mathcal{A}}_T'} \ar@<-.5ex>[rr]\ar@<.5ex>[rr] && \mathfrak{M}_{{G} }({\Sigma},{V}_{0,0})_{\check{H}',\check{\mathcal{A}}'}; } \]
where $\widetilde{\Sigma}$ is the surface obtained by gluing four copies of $\Sigma$ as in \eqref{eq:surdougro} and the boundary data are defined according to the notation of Theorems \ref{thm:doupoigro}, \ref{thm:twicardir} and Proposition \ref{rem:baspoigro}. \end{thm}
\begin{proof} Theorem \ref{thm:doupoigro} gives us a double Poisson groupoid structure on the upper face of the previous diagram, while Theorem \ref{thm:doupoigro2} gives us a double Poisson groupoid structure on each of the side faces of the diagram. The compatibility of these groupoid structures follows from the fact that they are obtained from reduction of a triple Lie groupoid. Let us follow the notation introduced in the proof of Theorem \ref{thm:doupoigro2}. We have a triple Lie groupoid 
\[ \xymatrix @-1.7pc { M^4 \ar@<-.5ex>[dd]\ar@<.5ex>[dd] \ar@<-.5ex>[rd] \ar@<.5ex>[rd] \ar@<-.5ex>[rr] \ar@<.5ex>[rr] && M^2 \ar@<.5ex>[dd]|\hole \ar@<-.5ex>[dd]|\hole \ar@<.5ex>[rd] \ar@<-.5ex>[rd] \\ 
& M^2 \ar@<.5ex>[dd] \ar@<-.5ex>[dd] \ar@<.5ex>[rr] \ar@<-.5ex>[rr] && M \ar@<.5ex>[dd] \ar@<-.5ex>[dd] \\
M_0^4 \ar@<.5ex>[rr]|!{[r];[r]}\hole \ar@<-.5ex>[rr]|!{[r];[r]}\hole \ar@<.5ex>[rd] \ar@<-.5ex>[rd] && M_0^2 \ar@<.5ex>[rd] \ar@<-.5ex>[rd] \\ & M_0^2 \ar@<-.5ex>[rr]\ar@<.5ex>[rr] && M_0; } \]
where the top and the bottom faces are pair double Lie groupoids and the side faces are products of $M \rightrightarrows M_0$ as in \eqref{eq:modspa2gr}. The moment map $\mu^4:M^4 \rightarrow \mathcal{G}^{4E}$ is a triple Lie groupoid morphism with target the triple Lie groupoid
\[ \xymatrix @-1.8pc { \mathcal{G}^{4E} \ar@<-.5ex>[dd]\ar@<.5ex>[dd] \ar@<-.5ex>[rd] \ar@<.5ex>[rd] \ar@<-.5ex>[rr] \ar@<.5ex>[rr] && \mathcal{G}^{2E} \ar@<.5ex>[dd]|\hole \ar@<-.5ex>[dd]|\hole \ar@<.5ex>[rd] \ar@<-.5ex>[rd] \\ 
& \mathcal{G}^{2E}  \ar@<.5ex>[dd] \ar@<-.5ex>[dd] \ar@<.5ex>[rr] \ar@<-.5ex>[rr] && \mathcal{G}^E \ar@<.5ex>[dd] \ar@<-.5ex>[dd] \\
G^{4E} \ar@<.5ex>[rr]|!{[r];[r]}\hole \ar@<-.5ex>[rr]|!{[r];[r]}\hole \ar@<.5ex>[rd] \ar@<-.5ex>[rd] && G^{2E} \ar@<.5ex>[rd] \ar@<-.5ex>[rd] \\ & G^{2E} \ar@<-.5ex>[rr]\ar@<.5ex>[rr] && G^E; } \]
which is again just the pair groupoid over the double Lie groupoid $\mathcal{G}^{2E}$ with sides $\mathcal{G}^E$ and $G^{2E}$. Each of the groupoids $\mathcal{H}_e $ in \eqref{eq:holgro} is a triple Lie subgroupoid of the previous triple Lie groupoid in this case:
\[ \xymatrix @-1.8pc { H^{4}_e \ar@<-.5ex>[dd]\ar@<.5ex>[dd] \ar@<-.5ex>[rd] \ar@<.5ex>[rd] \ar@<-.5ex>[rr] \ar@<.5ex>[rr] && H^{2}_e \ar@<.5ex>[dd]|\hole \ar@<-.5ex>[dd]|\hole \ar@<.5ex>[rd] \ar@<-.5ex>[rd] \\ 
& H_e^{2}  \ar@<.5ex>[dd] \ar@<-.5ex>[dd] \ar@<.5ex>[rr] \ar@<-.5ex>[rr] && H_e \ar@<.5ex>[dd] \ar@<-.5ex>[dd] \\
(H'_e)^{4} \ar@<.5ex>[rr]|!{[r];[r]}\hole \ar@<-.5ex>[rr]|!{[r];[r]}\hole \ar@<.5ex>[rd] \ar@<-.5ex>[rd] && (H'_e)^{2} \ar@<.5ex>[rd] \ar@<-.5ex>[rd] \\ & (H'_e)^{2} \ar@<-.5ex>[rr]\ar@<.5ex>[rr] && H_e', }\quad \quad 
\xymatrix @-1.7pc { \mathcal{G}_\Delta^2 \ar@<-.5ex>[dd]\ar@<.5ex>[dd] \ar@<-.5ex>[rd] \ar@<.5ex>[rd] \ar@<-.5ex>[rr] \ar@<.5ex>[rr] && \mathcal{G} \ar@<.5ex>[dd]|\hole \ar@<-.5ex>[dd]|\hole \ar@<.5ex>[rd] \ar@<-.5ex>[rd] \\ 
& \mathcal{G}  \ar@<.5ex>[dd] \ar@<-.5ex>[dd] \ar@<.5ex>[rr] \ar@<-.5ex>[rr] && \mathcal{G} \ar@<.5ex>[dd] \ar@<-.5ex>[dd] \\
G^{2}_\Delta \ar@<.5ex>[rr]|!{[r];[r]}\hole \ar@<-.5ex>[rr]|!{[r];[r]}\hole \ar@<.5ex>[rd] \ar@<-.5ex>[rd] && G \ar@<.5ex>[rd] \ar@<-.5ex>[rd] \\ & G \ar@<-.5ex>[rr]\ar@<.5ex>[rr] && G, } \quad\quad
 \xymatrix @-1.9pc { (H_e \times \mathcal{G} )^2 \ar@<-.5ex>[dd]\ar@<.5ex>[dd] \ar@<-.5ex>[rd] \ar@<.5ex>[rd] \ar@<-.5ex>[rr] \ar@<.5ex>[rr] && H_e \times \mathcal{G}  \ar@<.5ex>[dd]|\hole \ar@<-.5ex>[dd]|\hole \ar@<.5ex>[rd] \ar@<-.5ex>[rd] \\ 
& \mathcal{G}^2  \ar@<.5ex>[dd] \ar@<-.5ex>[dd] \ar@<.5ex>[rr] \ar@<-.5ex>[rr] && \mathcal{G} \ar@<.5ex>[dd] \ar@<-.5ex>[dd] \\
(H_e' \times {G} )^2 \ar@<.5ex>[rr]|!{[r];[r]}\hole \ar@<-.5ex>[rr]|!{[r];[r]}\hole \ar@<.5ex>[rd] \ar@<-.5ex>[rd] && H'_e \times G \ar@<.5ex>[rd] \ar@<-.5ex>[rd] \\ & G^{2} \ar@<-.5ex>[rr]\ar@<.5ex>[rr] && G, }  \]
\[ \xymatrix @-1.9pc { ( \mathcal{G} \times H_e)^2 \ar@<-.5ex>[dd]\ar@<.5ex>[dd] \ar@<-.5ex>[rd] \ar@<.5ex>[rd] \ar@<-.5ex>[rr] \ar@<.5ex>[rr] && \mathcal{G}\times  H_e  \ar@<.5ex>[dd]|\hole \ar@<-.5ex>[dd]|\hole \ar@<.5ex>[rd] \ar@<-.5ex>[rd] \\ 
& \mathcal{G}^2  \ar@<.5ex>[dd] \ar@<-.5ex>[dd] \ar@<.5ex>[rr] \ar@<-.5ex>[rr] && \mathcal{G} \ar@<.5ex>[dd] \ar@<-.5ex>[dd] \\
(   G \times H_e')^2 \ar@<.5ex>[rr]|!{[r];[r]}\hole \ar@<-.5ex>[rr]|!{[r];[r]}\hole \ar@<.5ex>[rd] \ar@<-.5ex>[rd] &&    G  \times H'_e\ar@<.5ex>[rd] \ar@<-.5ex>[rd] \\ & G^{2} \ar@<-.5ex>[rr]\ar@<.5ex>[rr] && G; } \]where the top double Lie groupoid structures in the previous diagrams are described in the proof of Theorem \ref{thm:doupoigro2} and the vertical structure is given by the multiplicativity of the boundary data. Analogously, $\mathcal{K} $ constitutes a triple Lie groupoid and its action on $(\mu^4)^{-1}(\mathcal{H} )$ is by a triple Lie groupoid morphism, then the quotient is a triple Lie groupoid as desired. \end{proof} 
\begin{exa} Consider the situation of Example \ref{exa:dousym2} but suppose that the structure group is a Lie 2-group whose Lie 2-algebra is the double of a Manin triple of Lie 2-algebras. Then the corresponding moduli space is a triple symplectic groupoid. \end{exa}  
\appendix 
\section{Reduction of morphisms between Manin pairs and the gluing of surfaces} The quasi-Hamiltonian space structure on $\hom(\Pi_1(\Sigma,V),G)$ for a marked surface $(\Sigma,V)$ as in \S \ref{sec:quisur} can be constructed using the following version of Courant reduction.
\subsection{Reduction of morphisms between Manin pairs}\label{subsec:redcou} Let $(\mathfrak{q},\mathfrak{u} )$ be a Manin pair of Lie algebras with $\langle\,,\, \rangle $ being the metric on $\mathfrak{q} $. Suppose that there is a Lie algebra action of $\mathfrak{q} $ on a manifold $P$ with coisotropic stabilizers. Then $(\mathcal{E},\mathcal{L}):=(\mathfrak{q} \times P,\mathfrak{u} \times P) $ naturally becomes a Manin pair, where $\mathcal{E}=\mathfrak{q} \times P $ is an action Courant algebroid in the sense of \cite{coupoi}. Suppose that $\mathcal{E} $ is an exact Courant algebroid and let $R:(\mathbb{T} M,TM) \rightarrow (\mathcal{E},\mathcal{L})$ be a morphism of Manin pairs over a map $\mu:M \rightarrow P$. Recall that a coisotropic Lie subalgebra of a quadratic Lie algebra $\mathfrak{q} $ as above is a Lie subalgebra $\mathfrak{c} \hookrightarrow \mathfrak{q} $ such that $\mathfrak{c}^\perp \hookrightarrow \mathfrak{c} $, where $\mathfrak{c}^\perp$ is the subspace of $\mathfrak{q} $ which is $\langle\,,\, \rangle $-orthogonal to $\mathfrak{c} $. A coisotropic Lie subalgebra $\mathfrak{c} \hookrightarrow \mathfrak{q} $ and a submanifold $j:N \hookrightarrow P$ form {\em reductive data} $(\mathfrak{c},N) $ \cite[Def. 3.2]{quisur2} for $R$ if the following conditions hold:
\begin{itemize} \item the foliation on $N$ induced by the image of the isotropic subbundle $\mathfrak{c}^\perp \times P \hookrightarrow  \mathcal{E}$ in $TP$ is simple,
\item $\mu$ intersects $N$ cleanly and the $\left(  \mathfrak{u}\cap\mathfrak{c}^\perp\right)$-action on $\mu^{-1}(N)$ determined by $R$ induces a simple foliation as well. 
\end{itemize} 
In such a situation, consider the diagram
\[ \begin{tikzcd}[column sep=4ex,row sep=3ex]
{(\mathbb{T}M,TM)} \arrow[rrr, "R"]                                                                             &  &  & {(\mathcal{E},\mathcal{L})} \arrow[dd, "R_{\mathfrak{c}} \times (Q\circ J^\top)"] \\
{(\mathbb{T}\mu^{-1}(N),T \mu^{-1}(N))} \arrow[u, "R_i"]                                                        &  &  &                                                                                   \\
{(\mathbb{T}M_{\mathfrak{c},N},TM_{\mathfrak{c},N})} \arrow[rrr, "{R_{\mathfrak{c},N}}"'] \arrow[u, "R_q^\top"] &  &  & {(\mathcal{E},\mathcal{L})_{\mathfrak{c},N}}                                     
\end{tikzcd}\]
where $i:\mu^{-1}(N) \hookrightarrow M $ s the inclusion and the Courant morphisms are defined as follows.
\begin{itemize} \item $R_j$ and $R_q$ are the canonical Courant relations over the corresponding maps and $q:M \rightarrow M_{\mathfrak{c},N }$ is the projection to the leaf space $ M_{\mathfrak{c},N }$ of the foliation on $\mu^{-1}(N)$ induced by the $( \mathfrak{u} \cap\mathfrak{c}^\perp)$-action.
\item $R_{\mathfrak{c} }$ is the quotient Courant morphism $\mathfrak{q} \rightarrow \mathfrak{c} /\mathfrak{c}^\perp$ determined by the {\em coisotropic reduction of $\mathfrak{c}$}: 
\[ u\sim_{R_{\mathfrak{c} }} (u+\mathfrak{c}^\perp) \quad  \text{for all $u\in \mathfrak{c}$} \quad \text{(equivalently, $R_{\mathfrak{c} }=\{(u,u+\mathfrak{c}^\perp)\in \mathfrak{d} \times(\mathfrak{c}/\mathfrak{c}^\perp ) |u\in \mathfrak{c} \}$)}. \]
\item $Q$ is the graph of the leaf space projection $N \rightarrow N/\mathfrak{c}^\perp$ of the foliation induced by $\mathfrak{c}^\perp$, $J$ is the graph of the inclusion $j: N \rightarrow P$.
\item $(\mathcal{E},\mathcal{L})_{\mathfrak{c},N}=\left((\mathfrak{c}/\mathfrak{c}^\perp) \times( N/\mathfrak{c}^\perp),\left( (\mathfrak{u}\cap \mathfrak{c} +\mathfrak{c}^\perp)/ \mathfrak{c}^\perp\right) \times (N/\mathfrak{c}^\perp)\right)$ is the Manin pair obtained by coisotropic reduction from $(\mathcal{E},\mathcal{L} )$ as in \cite[\S 2.1]{coupoi}. \end{itemize}     
\begin{thm}[\cite{quisur2}]\label{thm:coured1} The composite $R_{\mathfrak{c},N}:= \left(R_{\mathfrak{c}} \times (Q\circ J^\top) \right)\circ R\circ R_i\circ R_q^\top$ is a morphism of Manin pairs. If $R$ is exact and the anchor map of $\mathcal{E}$ restricted to $\mathfrak{c} \times N$ covers $T N$ surjectively, then $R_{\mathfrak{c},N}$ is exact as well. \end{thm}  
This result is the key technical tool introduced in \cite[Thm. 3.1, Thm. 3.2]{quisur2}. Now we shall state a simple corollary that is useful for reducing multiplicative quasi-Hamiltonian spaces. 
\begin{defi}\label{def:mulred} We say that the reductive data $(\mathfrak{c},N)$ for $R$ as in Theorem \ref{thm:coured1} are {\em multiplicative} if the following conditions hold:
\begin{itemize} \item $(\mathfrak{q},\mathfrak{u})$ is a Manin pair of Lie 2-algebras ($\mathfrak{q}$ is a split quadratic Lie 2-algebra as in \S \ref{subsec:multwidir} and $\mathfrak{u} \hookrightarrow \mathfrak{q} $ is a Lagrangian Lie 2-subalgebra), 
\item $P$ is a Lie groupoid and $(\mathcal{E},\mathcal{L})$ is a multiplicative Manin pair, 
\item $M$ is a multiplicative quasi-Hamiltonian space with Manin pair morphism $R$ as in Definition \ref{def:mulquaham},
\item $\mathfrak{c} \hookrightarrow \mathfrak{q}$ is a Lie 2-subalgebra and $N \hookrightarrow P$ is a Lie subgroupoid,
\item the foliation on $N$ determined by $\mathfrak{c}^\perp \times N$ is multiplicatively simple as in Definition \ref{def:mulfol} and so is the foliation on $M$ induced by $\mathfrak{u} \cap \mathfrak{c}^\perp$. \end{itemize} \end{defi} 
\begin{prop}\label{pro:cared} Suppose that the reductive data in Theorem \ref{thm:coured1} are multiplicative. Then $M_{\mathfrak{c},N}$ is a multiplicative quasi-Hamiltonian space with Manin pair morphism $R_{\mathfrak{c},N}$. \end{prop} \begin{proof} This follows from the fact that 
\[ (\mathcal{E},\mathcal{L})_{\mathfrak{c},N}=\left((\mathfrak{c}/\mathfrak{c}^\perp) \times( N/\mathfrak{c}^\perp),\left( (\mathfrak{u}\cap \mathfrak{c} +\mathfrak{c}^\perp)/ \mathfrak{c}^\perp\right) \times (N/\mathfrak{c}^\perp)\right) \] 
is automatically a multiplicative Manin pair (note that $(\mathfrak{c}/\mathfrak{c}^\perp,(\mathfrak{u}\cap \mathfrak{c} +\mathfrak{c}^\perp)/\mathfrak{c}^\perp) $ is a Manin pair of Lie 2-algebras) and $R_{\mathfrak{c},N}$ is the composite of multiplicative Courant morphisms. \end{proof}
\subsection{Applications to spaces of representations}\label{subsec:qhammodspa} 

\subsubsection{Forgetting vertices in a marked surface}\label{subsec:forver} We shall follow the notation and assumptions of \S \ref{subsec:quapoimodspa} in what follows. Let $(\Sigma,V)$ be a marked surface and consider a finite subset $V'\subset \partial \Sigma$ such that $V\subset V'$. Then the quasi-Poisson bivector $\pi_{\Sigma,V}$ and corresponding $G^V$-action on $M=\hom(\Pi_1(\Sigma,V),G)$ are obtained by reduction from $\pi_{\Sigma,V'}$ on $M':=\hom(\Pi_1(\Sigma,V'),G)$ \cite[Corollary 2]{quisur}. We can interpret this observation using Courant morphims and thus we obtain that 
\[ R_{\mathfrak{c} }\circ \mathbf{R}_{\Sigma,V'}\circ R_q^\top=\mathbf{R}_{\Sigma,V}, \]
where $R_{\mathfrak{c} }: \mathfrak{d}^{V'} \rightarrow \mathfrak{d}^V  $ is the quotient Courant morphism corresponding to the coisotropic reduction of $\mathfrak{c}:=\left(\prod_{v\in V} \mathfrak{d}_v \times \prod_{v\in V'-V} \mathfrak{g}_\Delta\right) \hookrightarrow \mathfrak{d}^{V'} $ and $R_q: \mathbb{T} M'\rightarrow \mathbb{T}M $ is the canonical Courant morphism over the quotient map $q:M'\rightarrow M'/G^{V'-V}=M$. But this implies that $\mathfrak{c}$ and $N=G^{E'}$ are reductive data for $R_{\Sigma,V'}$ in the sense of \S \ref{subsec:redcou} and also that 
\begin{align}  R_{\Sigma,V}=  \left(R_{\Sigma,V'} \right)_{\mathfrak{c},N }. \label{eq:forver} \end{align}  
Recall from the discussion in \S \ref{subsec:quapoimodspa} that we can construct $R_{\Sigma,V}$ using $\mu$ and $\mathbf{R}_{\Sigma,V}$ as in \cite[\S 5.2.1]{quisur2}. This implies, in particular, that we can contruct ${R}_{\Sigma,V}$ (and hence $\mathbf{R}_{\Sigma,V}$ as in \eqref{eq:quapoidir}) starting from the case of a fully marked surface. In fact, we can choose $V'$ in such a way that $(\Sigma,V')$ is fully marked to prove our claim. Note that it also follows from \eqref{eq:bivmodspa} that a marked point $v$ which is decorated with $\mathfrak{g}_\Delta \hookrightarrow \mathfrak{d} $ has no effect on the Poisson tensor of the corresponding decorated moduli space.

\subsubsection{Sewing a surface along a pair of edges}\label{subsec:sew} Let us describe the reduction corresponding to identifying two edges in a marked surface $(\Sigma',V')$ as in \cite[\S 4.1]{quisur2}. Let $\Gamma'=(E',V')$ be the boundary graph of $(\Sigma',V')$ for $i=1,2$ and take two positively parameterized boundary edges $j_i:[0,1] \hookrightarrow \partial \Sigma'$ for $i=1,2$. Consider the surface 
\[ \Sigma:= \Sigma'/j_1(t)\sim j_2(1-t), \quad \forall t\in [0,1]; \]  
let $e$, $e'$ be the boundary edges determined by $j_1([0,1])$ and $j_2([0,1])$ respectively. Let $\Gamma =(E,V)$ be the boundary graph of $(\Sigma,V)$, where $V$ is obtained by including the vertices of $V'$ in $\partial\Sigma$, identifying the endpoints $\mathtt{S}( e)\sim \mathtt{T}(e') $, $\mathtt{T}(e)\sim \mathtt{S}(e')$ and removing one of these points if it lies in the interior of $\Sigma$ which is denoted by $\mathring{\Sigma}$. The Manin pair morphism $R_{\Sigma,V}$ is obtained in the following manner. Let us introduce the notation $M':=\hom(\Pi_1(\Sigma',V'),{G} )$ and let $\mu':M' \rightarrow G^{E'}$ be the corresponding moment map. Then we can consider the reductive data $(\mathfrak{c}_{\text{sew}_{\{e,e'\}} },N)$ for $ R_{\Sigma',V'}$: 
\begin{align} &\mathfrak{c}_{\text{sew}_{\{e,e'\}} }= \left(\mathfrak{c}_1 \times  \mathfrak{c}_2 \times \left( \prod_{v\in V'-\{\mathtt{T}(e),\mathtt{S}(e),\mathtt{T}(e'),\mathtt{S}(e')\}  } \mathfrak{d}_v  \right)\right)\hookrightarrow \prod_{v\in V'} \mathfrak{d}_v   \label{eq:sew1}\\ 
&\mathfrak{c}_1=\begin{cases} \{(V\oplus W, U\oplus V)| U,V,W\in \mathfrak{g}\} \hookrightarrow \mathfrak{d}_{\mathtt{S}(e)}  \times \mathfrak{d}_{\mathtt{T}(e')}\quad   \text{if $\mathtt{S}(e)\sim \mathtt{T}(e') $ does not lie in $\mathring{\Sigma}$} \\
\{(V\oplus W, V\oplus W)| V,W\in \mathfrak{g}\} \hookrightarrow \mathfrak{d}_{\mathtt{S}(e)}  \times \mathfrak{d}_{\mathtt{T}(e')} \quad \text{otherwise, } \end{cases} \notag \\  
&\mathfrak{c}_2=\begin{cases} \{(U\oplus V, V\oplus W)| U,V,W\in \mathfrak{g}\} \hookrightarrow \mathfrak{d}_{\mathtt{T}(e)}  \times \mathfrak{d}_{\mathtt{S}(e') } \quad   \text{if $\mathtt{T}(e)\sim \mathtt{S}(e') $ does not lie in $\mathring{\Sigma}$} \\
\{(V\oplus W, V\oplus W)| V,W\in \mathfrak{g}\} \hookrightarrow \mathfrak{d}_{\mathtt{T}(e)}  \times \mathfrak{d}_{\mathtt{S}(e')} \quad \text{otherwise, } \end{cases} \notag \\
&N=\left(\{(g,g^{-1})|g\in G\} \times\left( \prod_{f\in E'-\{e,e'\}} G_f  \right) \right)
\hookrightarrow  G_e \times G_{e'} \times \left(\prod_{f\in E'-\{e,e'\}} G_f \right)=G^{E'}. \label{eq:sew2} \end{align}  
We have that $R_{\Sigma,V}=\left(R_{\Sigma',V'} \right)_{\mathfrak{c}_{\text{sew}_{\{e,e'\}} },N}$ is obtained then by implementing Theorem \ref{thm:coured1}; this procedure is called {\em sewing} in \cite[\S 4.1]{quisur2}.    
\subsubsection{Construction of the canonical Manin pair morphism $R_{\Sigma,V}$}\label{subsec:mpmor} Let us sketch now how we can use the sewing construction above to produce the morphism of Manin pairs $R_{\Sigma,V}$ for a general fully marked surface $(\Sigma,V)$ with boundary graph $\Gamma =(E,V)$. According to \S \ref{subsec:forver}, this is sufficient to deal with the situation of a marked but not necessarily fully marked surface as well. The starting point is the morphism of Manin pairs $R_{\Delta^2}=R_{\Delta^2,\{v_i\}_{i=1,2,3}}$ corresponding to a triangle $(\Delta^2,\{v_i\}_{i=1,2,3})$ as in Example \ref{exa:modspatri} that we identify here with the 2-dimensional simplex $\Delta^2$. Let us denote $M_{\Delta^2}:=\hom(\Pi_1({\Delta^2},\{v_i\}_{i=1,2,3}),G)$. 

Choose a triangulation $\mathcal{T}=(\mathcal{T}_2,\mathcal{T}_1,\mathcal{T}_0   ) $ of $\Sigma$, where $\mathcal{T}_i$ denotes the set of $i$-dimensional simplices, in such a way that $V=\partial \Sigma \cap \mathcal{T}_0$. Let $V'$ be the set of vertices in $\mathcal{T}_0 $ that lie in the interior of $\Sigma$. One can then perform sewing as above on the product relation $R_{\Delta^2}^{\mathcal{T}_2 }$ using the reductive data \eqref{eq:sew1} and \eqref{eq:sew2} determined by every edge in $\mathcal{T}_1 $ that is the common edge of two triangles in $\mathcal{T}_2 $, see Figure \ref{fig:glu2sur}. Let us recall now what the output of such a process is in terms of differential forms \cite[\S 1.2.1]{quisur2}. The space $\hom(\Pi_1(\Sigma,V),G)$ can be represented as the quotient $M= \mathcal{M} /G^{V'}$, where $\mathcal{M}  \hookrightarrow M_{\Delta^2}^{ \mathcal{T}_2  } $ is given by 
\[ \mathcal{M} =\{(\rho_T)_{T\in \mathcal{T}_2 }|\rho_T(e)=\rho_{T'}(e')^{-1} \text{ if $e=e'$ as elements of $\mathcal{T}_1 $ and $T,T'\in \mathcal{T}_2$ satisfy $T\neq T'$ } \} \]
and the action is the gauge action \eqref{eq:gauact} of $G^{V'}$ on $\mathcal{M} $. Note that the moment map $\mu$ on $M$ is also induced by the underlying moment map of $R_{\Delta^2}^{\mathcal{T}_2 }$ via reduction. Let $\Omega_{\Delta^2}\in \Omega^2 \left(M_{\Delta^2}\right)$ be determined by pulling back the 2-form $\omega $ defined in \eqref{eq:polwie} using any identification 
\[ M_{\Delta^2}\cong G^2 \] 
as in Example \ref{exa:modspatri} ($\Omega_{\Delta^2}$ is independent of the chosen identification). Let $\text{pr}_T:M_{\Delta^2}^{ \mathcal{T}_2  } \rightarrow M_{\Delta^2}$ be the projection to the factor corresponding to $T\in \mathcal{T}_2$ and denote by $\iota: \mathcal{M} \hookrightarrow M_{\Delta^2}^{ \mathcal{T}_2  } $ the inclusion. Then $\iota^*\sum_{T\in \mathcal{T}_2  }\text{pr}_T^* \Omega_{\Delta^2}$ is basic with respect to the quotient map $\mathcal{M} \rightarrow {M} $ \cite[Thm. 1.2]{quisur2}. Let $\Omega  \in \Omega^2({M}) $ be the induced form on the quotient. In terms of a suitable splitting of $\mathcal{E}_{\Gamma }  $, we can then represent
\begin{align}  R_{\Sigma,V}\cong  \{(u\oplus \mu^*\alpha -i_u \Omega   ,T\mu(u)\oplus \alpha )\in \mathbb{T}M \times \mathbb{T}_\Theta G^E|u\oplus \alpha \in TM \oplus \mu^*T^*G^E  \}; \label{eq:mpmodspa} \end{align} 
where $\Theta$ is the pullback of the Cartan 3-form on each factor $G$. Let us point out that one can also apply this construction to a closed surface $\Sigma$ thus recovering the Atiyah-Bott symplectic structure as in \cite[\S 5]{weisymmod}.
\section{Double Lie groupoids and Dirac structures of symmetric decorations}\label{app:dougro}
\begin{prop}\label{pro:intla} Let us assume that $(\widehat{\Sigma},\widehat{V})$ is as in Theorem \ref{thm:poigro}. Then the following statements hold.
\begin{enumerate} \item the LA-groupoid $\widehat{\mathcal{A} }$ corresponding to a symmetric decoration of $(\widehat{\Sigma},\widehat{V})$ as in Remark \ref{rem:muldir} is integrable by a double Lie groupoid $\mathcal{G}$ which has source-connected fibers over the Lie groupoid $G^{\widehat{E} } \rightrightarrows G^{E'}$ of \eqref{eq:gpdtar1}. \item the $\widehat{\mathcal{A}}$-orbits in $G^{\widehat{E} } \rightrightarrows G^{E'}$ that pass through $G^{E'}$ are subgroupoids of $G^{\widehat{E} } \rightrightarrows G^{E'}$. \end{enumerate}   \end{prop}
\begin{proof} Let us use the notation introduced in \S \ref{subsec:symdec}. Take $v\in V_0$ and let ${A}_v \subset G \times G$ be the connected subgroup which integrates ${\mathfrak{a}}_v\subset \mathfrak{d}_v $ for all $v\in V_0$. Consider the pair groupoid
\[ \mathcal{G}_v:= \left( {A}_v  \times {A}^\vee_v \rightrightarrows {A}_v \right) , \quad A_v^\vee:=\{(x,y)\in G \times G|(y,x)\in A_v \}, \, \forall v\in V_0; \] 
where, to ensure compatibility with the CA-groupoid structure on $\mathcal{E}_{\widehat{\Gamma } } $ as in Lemma \ref{lem:mulca}, the multiplication on $\mathcal{G}_v$ has to be given by 
\[ \mathtt{m}( ((x,y),(w,z)),((z,w),(x',y')))=((x,y),(x',y')), \quad  \forall (x,y),(x',y'),(z,w)\in A_v. \] 
Consider the direct product $\mathcal{G} =\prod_{v\in V_0} \mathcal{G}_v$ which is a groupoid over $\mathcal{G}_0= \prod_{v\in V_0}A_v $. We can view $\mathcal{G} $ as a subgroup of $ G^{2i_1(V_0)} \times G^{2i_2(V_0)}=G^{2\widehat{V} }$ by including the subgroups $A_v \hookrightarrow G^2_{i_1(v)}$, $A_v^\vee \hookrightarrow G^2_{i_2(v)}$ inside the copy of $G^2_{i_a(v)} \hookrightarrow G^{2\widehat{V} } $ that corresponds to $i_a(v)\in\widehat{V}=i_1(V_0)\cup i_2(V_0)$ for each $v\in V_0$ and $a=1,2$. Then the natural integration of the infinitesimal action \eqref{eq:infact} restricted to $\widehat{\mathcal{A} }$ gives us an action Lie groupoid. Namely, the action of $\mathcal{G} $ on $G^{\widehat{E} } $ is 
\[ (x_v,y_v)_{v\in \widehat{V} }\cdot (g_e)_{e\in \widehat{E} }= \left(y_{\mathtt{T}(e) }g_ex_{\mathtt{S}(e) }^{-1}\right)_{e\in \widehat{E} }, \quad \forall (x_v,y_v)_{v\in \widehat{V} }\in \mathcal{G} \hookrightarrow G^{2\widehat{V} },\quad \forall  (g_e)_{e\in \widehat{E} } \in G^{\widehat{E} }. \]
It is straightforward to verify that the vertical action groupoid described above together with the horizontal product groupoid structure determine a double Lie groupoid:
\[ \xymatrix@-0.6pc{ \mathcal{G} \times G^{\widehat{E} } \ar@<-.5ex>[r] \ar@<.5ex>[r]\ar@<-.5ex>[d] \ar@<.5ex>[d]& \mathcal{G}_0 \times G^{E'} \ar@<-.5ex>[d] \ar@<.5ex>[d] \\ G^{\widehat{E} } \ar@<-.5ex>[r] \ar@<.5ex>[r] & G^{{E'} }. }\]

 By construction, the source-fibers of the double Lie groupoid described above over $G^{\widehat{E} } $ are connected. Then a $\mathcal{G} $-orbit in $G^{\widehat{E} } $ is also an $\widehat{\mathcal{A}}$-orbit and viceversa. Therefore, item (2) follows from \cite[Prop. 7]{folpoigro}.  \end{proof}
\subsection*{Acknowledgements} The author is very grateful to H. Bursztyn for inspiring conversations that motivated this work. The author thanks P. Boalch, M. Gualtieri, E. Meinrenken and J. Pulmann for their valuable comments and observations and the author also thanks the referee whose attention to detail and opportune suggestions have led to substantial improvements.

\printbibliography
\end{document}